\documentclass{article}

\PassOptionsToPackage{numbers, compress}{natbib}

\usepackage[preprint]{neurips_2020}

\usepackage[utf8]{inputenc} 
\usepackage[T1]{fontenc}    
\usepackage{hyperref}       
\usepackage{url}            
\usepackage{natbib} 
\usepackage{booktabs}       
\usepackage{amsfonts}       
\usepackage{nicefrac}       
\usepackage{microtype}      

\usepackage{amsmath,amsfonts,amssymb,amsthm}
\usepackage{pifont}
\usepackage{color}
\usepackage{url}
\usepackage{enumitem}
\usepackage[labelformat=simple]{subcaption}

\usepackage{algorithm}
\usepackage{algorithmic}
\usepackage{comment}
\usepackage{multirow}
\usepackage{multicol}
\usepackage{empheq}

\usepackage{hf-tikz} 
\usetikzlibrary{patterns}
\usetikzlibrary{matrix,decorations.pathreplacing}
\pgfkeys{tikz/mymatrixenv/.style={decoration={brace},every left delimiter/.style={xshift=8pt},every right delimiter/.style={xshift=-8pt}}}
\pgfkeys{tikz/mymatrix/.style={matrix of math nodes,nodes in empty cells,left delimiter={[},right delimiter={]},inner sep=1pt,outer sep=1.5pt,column sep=2pt,row sep=2pt,nodes={minimum width=20pt,minimum height=10pt,anchor=center,inner sep=0pt,outer sep=0pt}}}
\pgfkeys{tikz/mymatrixbrace/.style={decorate,thick}}

\newcommand*\mymatrixbraceright[4][m]{
    \draw[mymatrixbrace] (#1.west|-#1-#3-1.south west) -- node[left=2pt] {#4} (#1.west|-#1-#2-1.north west);
}
\newcommand*\mymatrixbraceleft[4][m]{
    \draw[mymatrixbrace] (#1.east|-#1-#2-1.north east) -- node[right=2pt] {#4} (#1.east|-#1-#2-1.south east);
}
\newcommand*\mymatrixbracetop[4][m]{
    \draw[mymatrixbrace] (#1.north-|#1-1-#2.north west) -- node[above=2pt] {#4} (#1.north-|#1-1-#3.north east);
}

\usepackage{ulem}

\newcommand\independent{\protect\mathpalette{\protect\independenT}{\perp}}
\def\independenT#1#2{\mathrel{\rlap{$#1#2$}\mkern2mu{#1#2}}}

\usepackage[toc,page]{appendix}

\usepackage[textwidth=2.3cm]{todonotes}

\newtheorem{theorem}{Theorem}
\newtheorem{lem}[theorem]{Lemma}
\newtheorem{remark}[theorem]{Remark}
\newtheorem{Def}[theorem]{Definition}
\newtheorem{prop}[theorem]{Proposition}

\usepackage{graphicx,tikz}
\graphicspath{{./img/}}
\usetikzlibrary{bayesnet}
\usepackage{pgfplots}
\usetikzlibrary{spy,calc}

\newif\ifblackandwhitecycle

\usepackage{mathabx}

\definecolor{green}{RGB}{20, 148, 20}

\title{Estimation and imputation in Probabilistic Principal Component Analysis with Missing Not At Random data}

\author{ 
  Aude Sportisse$^{(1)}$,
   Claire Boyer$^{(1,2)}$,
   Julie Josse$^{(3,4)}$ \\
   $^{(1)}$ Sorbonne Université, $^{(2)}$ ENS Paris, $^{(3)}$ Ecole Polytechnique, $^{(4)}$ Google
}

\begin{document}

\maketitle

\begin{abstract}

Missing Not At Random (MNAR) values lead to significant biases in the data, since the probability of missingness depends on the unobserved values.
They are "not ignorable" in the sense that they often require defining a model for the missing data mechanism, which makes inference or imputation tasks more complex. Furthermore, this implies a strong \textit{a priori} on the parametric form of the distribution.
However, some works have obtained guarantees on the estimation of parameters in the presence of MNAR data, without specifying the distribution of missing data \citep{mohan2018estimation, tang2003analysis}. This is very useful in practice, but is limited to simple cases such as self-masked MNAR values in data generated according to linear regression models.
We continue this line of research, but extend it to a more general MNAR mechanism, in a more general model of the probabilistic principal component analysis (PPCA), \textit{i.e.}, a low-rank model with random effects. We prove identifiability of the PPCA parameters. We then propose an estimation of the loading coefficients and a data imputation method. They are based on estimators of means, variances and covariances of missing variables, for which consistency is discussed. These estimators have the great advantage of being calculated using only the observed data, leveraging the underlying low-rank  structure of the data. We illustrate the relevance of the method with numerical experiments on synthetic data and also on real data collected from a medical register.
	
\end{abstract}

\section{Introduction} \label{sec:intro}
The problem of missing data is ubiquitous in the practice of data analysis. Theoretical guarantees of estimation strategies or imputation methods rely on assumptions regarding the missing-data mechanism, \textit{i.e.} the cause of the lack of data. \citet{rubin1} introduced three missing-data mechanisms. The data are said (i) Missing Completely At Random (MCAR) if the probability of being missing  is the same for all observations, (ii) Missing At Random (MAR) if the probability of being missing only depends on observed values, (iii) Missing Not At Random (MNAR) if the unavailability of the data depends on both observed and unobserved data such as its value itself. We focus on this later case, which is  frequent in practice, and theoretically challenging. A classic example of MNAR data is surveys about salary for which rich people would be less willing to disclose their income.

When the data is MCAR or MAR, statistical inference is carried out by ignoring the missing data mechanism \citep{rubin_little}. In the MNAR case, the observed variables are no longer representative of the population, which leads to selection bias in the sample, and therefore to bias in the parameters estimation. Therefore, it is usually necessary to take into account the missing data distribution. Most of the time,  the missing-data mechanism distribution is specified by logistic regression models \citep{ibrahim1999missing,morikawa2017semiparametric,sportisse2018imputation}. This comes at the price of an important computational burden to perform inference and is often restricted to a limited number of MNAR variables. In the recommender system community, there are some works \citep{marlin2009collaborative,hernandez2014probabilistic,ma2019missing,wang2019doubly} proposing a joint modelling of the data and mechanism distributions using matrix factorization by debiasing existing methods for MCAR data, for instance with inverse probability weighting approaches.

In addition, a key issue of MNAR data is to establish identifiability, which is not always guaranteed \citep{miao2016identifiability}. There is a huge litterature on this, both in the non-parametric \citep{mohan2013graphical,mohan2014graphical,ilya2015missing,shpitser2016consistent,nabi2020full}, and semi-parametric settings \citep{wang2014instrumental,miao2018identification}. For parametric models, in the case of multivariate regression, \citet{tang2003analysis} and \citet{miao2016identifiability} 
guarantee the identifiability of the coefficients of the conditional distribution of $ Y|X $, whereas $Y$ is missing. \citet{tang2003analysis} estimate them by calculating those of the distributions of $X$ and $X| Y$ in the full case (using only observations with no missing values). Besides, assuming a self-masked mechanism, \textit{i.e.}, the lack depends only on the missing variable itself, \citet{mohan2018estimation} consider a related approach based on graphical models, adopting a causal point of view.
Despite the great advantage of not modeling the distribution of missing values, the hypothesis of a self-masked MNAR mechanism can be strong in certain contexts as well as that of considering simple models. 

\paragraph{Contributions.} We consider that data are generated under the latent variable model,  probabilistic principal components analysis (PPCA) \citep{tipping1999probabilistic} and contain missing values. providing linear embedding, used in practice both for data visualization and as a powerful imputation tool  \citep{ilin2010practical, josse2012handling}.  Contrary to available works that handle only MAR data \citep{ilin2010practical}, we perform PPCA with MNAR values (on several variables) and with the possibility of having different mechanisms in the same data (MNAR and MAR).
\begin{itemize}
\item 
We discuss identifiability of the PPCA model parameters, and prove it considering self-masked MNAR encompassing a large set of self-masked mechanism distributions. 
\item  For more general MNAR mechanism, we  suggest a strategy to estimate the PPCA loading matrix without any modeling of the missing-data mechanism and use it to impute missing values in this non-ignorable missing data setting. 

\item The proposed method is based on estimators for the mean, the variance and the covariance of the variables with MNAR values. We show that they can be consistently estimated, only using the complete-case analysis. 
Two strategies can lead to the proposed estimators: (i) the first one is made of algebraic arguments based on partial linear models derived from the PPCA model;
(ii) the second one is inspired by \citep{mohan2018estimation} and uses graphical models tools and the so-called  missingness graph.  
\item We derive an algorithm implementing our proposal. We show that it outperforms the state-of-the-art methods on synthetic data and on a real data set  collected from a medical registry (Traumabase$^{\mbox{\normalsize{\textregistered}}}$).
The code to reproduce all the simulations and numerical experiments is available on \url{https://github.com/AudeSportisse/PPCA_MNAR}.

\end{itemize}



\section{PPCA model with informative missing values: identifiability issues} \label{sec:model}
\paragraph{Setting} 
The data matrix $Y \in \mathbb{R}^{n \times p}$ is assumed to be generated under a fully-connected PPCA model \citep{tipping1999probabilistic} (a.k.a.\ low-rank random effects model), \textit{i.e.} by the factorization of the loading matrix $B\in \mathbb{R}^{r \times p}$ and $r$ latent variables grouped in the matrix $W \in \mathbb{R}^{n\times r}$, 
\begin{equation}
	\label{eq:model}
	Y=\mathbf{1}\alpha + WB + \epsilon,  \textrm{with}
	\left\{
	\begin{array}{ll}
		W=(W_{1.} | \hdots | W_{n.})^T,  \textrm{ with } \: W_{i.} \sim \mathcal{N}(0_r,\mathrm{Id}_{r\times r}) \in \mathbb{R}^r , \\
		B \textrm{ of rank } r < \min \{n,p\}, \\
		\alpha \in \mathbb{R}^{p} \textrm{ and } \mathbf{1}=(1 \dots 1)^T \in \mathbb{R}^n, \\
		\epsilon=(\epsilon_{1.} | \hdots | \epsilon_{n.})^T,  \textrm{ with } \: \epsilon_{i.} \sim \mathcal{N}(0_p,\sigma^2\mathrm{Id}_{p\times p}) \in \mathbb{R}^p ,
	\end{array}
	\right.
\end{equation}
for $\sigma^2$ and $r$ known. 
In the sequel, $Y_{.j}$ and $Y_{i.}$ respectively denote the column $j$ and  the row $i$ of $Y$. The rows of $Y$ are identically distributed, $\forall i \in \{1,\dots,n\}, \quad  Y_{i.} \sim \mathcal{N}(\alpha,B^TB+\sigma^2\mathrm{Id}_{p\times p}).$



Some variables $Y_{.m_1}, \dots, Y_{.m_d}$, indexed by $\mathcal{M}:= \{m_1, \hdots, m_d\}\subset \{1,\hdots ,p\}$ (with $d<p$), are supposed to have MNAR values. The other variables are considered to be observed (or M(C)AR see Appendix \ref{sec:extensionmecha}). 
We let $\Omega \in \{0,1\}^{n\times p}$ denote the missing-data pattern (or mask) as
\begin{equation}
	\forall i \in \{1,\dots,n\}, \: \forall j \in \{1,\dots,p\}, \quad  \Omega_{ij}=
	\begin{cases}
		0 &\textrm{ if } Y_{ij} \textrm{ is missing,} \\
		1 &\textrm{ otherwise.}
	\end{cases}
\end{equation}
In the sequel, let us denote the complementary of a set $\mathcal{A}$ as $\widebar{\mathcal{A}}:=\{1,\hdots , p\}\setminus \mathcal{A}$.
The MNAR mechanism we consider is defined as follows, with $\mathcal{J}\subset \widebar{\mathcal{M}}$ and $|\mathcal{J}|=r$,
\begin{equation}\label{eq:MNARmechanism}
\forall m \in \mathcal{M}, \forall i \in \{1,\dots, n\},  \quad  \mathbb{P}(\Omega_{im}=1|Y_{i.})=\mathbb{P}(\Omega_{im}=1|(Y_{ik})_{k \in \widebar{\mathcal{J}}}),
\end{equation}
which implies that the distribution of the mechanism may depend on all variables (missing or observed) except $r$ of them, that we will call \emph{pivot variables}. Note that \eqref{eq:MNARmechanism} implies that $d<p-r$. 



\paragraph{Model identifiability} 

Under 
a self-masked MNAR mechanism, one can prove the identifiability of the PPCA model, \textit{i.e.} the joint distribution of $Y$ can be uniquely determined from available information.

\begin{prop}
\label{prop:identifiability}
Consider that $d$ variables are self-masked MNAR indexed by $\mathcal{M}$ and $p-d$ variables are MCAR (or observed), indexed by $\widebar{\mathcal{M}}$, as follows
\begin{align}
\label{eq:selfmaskedMNARmechanism}
\forall m \in \mathcal{M}, \forall i \in \{1,\dots, n\},  \quad  \mathbb{P}(\Omega_{im}=1|Y_{i.})&=\mathbb{P}(\Omega_{im}=1|Y_{im})=F_m(\phi^0_{m}+\phi^1_{m}Y_{im}), \\
\label{eq:MCARmechanism}
\forall j \in \widebar{\mathcal{M}}, \forall i \in \{1,\dots, n\},  \quad  \mathbb{P}(\Omega_{ij}=1|Y_{i.})&=\mathbb{P}(\Omega_{ij}=1)=F_j(\phi_{j}),
\end{align}
with $\phi_j \in \mathbb{R}$ and $\phi_m=(\phi^0_{m},\phi^1_{m})\in \mathbb{R}^2$ the mechanism parameters. $F_j$ and $F_m$ are assumed to be strictly monotone functions with a finite support.
Assume also that
\begin{equation}\label{eq:hypID}
\forall (k,\ell) \in \{1,\dots,p\}^2, \quad  k\neq l, \qquad \Omega_{.k} \independent \Omega_{.\ell} | Y
\end{equation}

The parameters $(\alpha, \Sigma)$ of 
 the PPCA model \eqref{eq:model} and the mechanism parameters $\phi=(\phi_l)_{l \in \{1,\dots p\}}$, despite self-masked MNAR values as in \eqref{eq:selfmaskedMNARmechanism} and MCAR values as in \eqref{eq:MCARmechanism}, are identifiable. Assuming that the noise level $\sigma^2$ is known, the parameter $B$ is identifiable up to a row permutation. 
\end{prop}
The proof is given in Appendix \ref{sec:proofidentif}. What is striking here is that mild assumptions are added about the mechanism definition. Indeed, no standard function for $F_m, m \in \mathcal{M}$ is discarded. In particular, the logistic function can be used, whereas \citep{miao2016identifiability} presented many counterexamples when identification fails considering this distribution. 

\section{Estimators with theoretical guarantees}\label{sec:generalcase}
In this section, we provide estimators of the means, variances and covariances for the MNAR variables, when data are generated as described in Section \ref{sec:model}. 
These estimators can be used to perform PPCA with informative missing values, providing an estimator of the loading matrix $B$ in \eqref{eq:model}.  This latter can be in turn used to predict missing values. This new imputation method for MNAR variables in a low-rank context is detailed in Algorithm \ref{alg:imputation}. 

For the rest of this section, denoting $\mathcal{J}_{-j}:= \mathcal{J} \setminus \{ j\}$, assume the following
\begin{enumerate}[label=\textbf{A\arabic*.}]
    \setcounter{enumi}{0}
    \item \label{hyp1} $\forall m \in \mathcal{M}$, $\forall j \in \mathcal{J}$, \, $\begin{pmatrix} B_{.m} & (B_{.j'})_{j' \in \mathcal{J}_{-j}} \end{pmatrix}$ is invertible,
    \item \label{hyp2} $\forall m \in \mathcal{M}$, $\forall j \in \mathcal{J}$, \, $Y_{.j} \independent \Omega_{.m} | (Y_{.k})_{k \in \widebar{\{j\}}}$.
\end{enumerate}

Note that Assumption \ref{hyp1} implies that $B$ has a full rank $r$ and that any variable is generated by all the latent variables. Assumption \ref{hyp2} follows from the missing-data mechanism in \eqref{eq:MNARmechanism}. 

For the sake of clarity, we start by illustrating these assumptions presenting the methodology in small dimension, before showing results in the general case. 



\subsection{Toy example: estimation of the mean of a  MNAR variable}
\label{sec:toyexample}

Consider the toy example where $p= 3,r= 2$, in which only one variable can be missing, and fix $\mathcal{M}=\{1\}$ and $\mathcal{J}=\{2,3\}$. Note that the MNAR mechanism is self-masked in such a context, because Equation \eqref{eq:MNARmechanism} leads to $\mathbb{P}(\Omega_{.1}=0|Y_{.1},Y_{.2},Y_{.3})=\mathbb{P}(\Omega_{.1}=0|Y_{.1}).$  but the method can be extended to other cases. A first goal is to estimate the mean of $Y_{.1}$, without specifying the distribution of the missing-data mechanism and using only the observed data.   

\subparagraph{Using algebraic arguments}  
We proceed in three steps:
(i) \ref{hyp1} allows to obtain  linear link between the pivot variables ($Y_{.2},Y_{.3}$) and the MNAR variable $Y_{.1}$. In particular, one has
\begin{equation}\label{eq:linear_toy}
Y_{.2} = \mathcal{B}_{2\rightarrow1,3[0]} + \mathcal{B}_{2\rightarrow1,3[1]}Y_{.1} + \mathcal{B}_{2\rightarrow1,3[3]}Y_{.3} + \zeta,
\end{equation}
with $\mathcal{B}_{2\rightarrow1,3[0]}$, $\mathcal{B}_{2\rightarrow1,3[1]}$ and $\mathcal{B}_{2\rightarrow1,3[3]}$ the intercept and coefficients standing for the effects of $Y_{.2}$ on $Y_{.1}$ and $Y_{.3}$, and with $\zeta$ a noise term;
(ii)  \ref{hyp2}, \textit{i.e.} $Y_{.2} \independent \Omega_{.1} | Y_{.1}, Y_{.3}$, is required to obtain identifiable and consistent parameters of the distribution of $Y_{.2}$ given $Y_{.1},Y_{.3}$ in the complete-case when $\Omega_{.1}=1$, denoted as $\mathcal{B}_{2\rightarrow1,3[0]}^c$, $\mathcal{B}_{2\rightarrow1,3[1]}^c$ and $\mathcal{B}_{2\rightarrow1,3[3]}^c$,
\begin{equation}\label{eq:linear_toy}
(Y_{.2})_{| \Omega_{.1}=1} = \mathcal{B}_{2\rightarrow1,3[0]}^c + \mathcal{B}_{2\rightarrow1,3[1]}^c Y_{.1} + \mathcal{B}_{2\rightarrow1,3[3]}^c Y_{.3} + \zeta^c,
\end{equation}
(note that the regression of $Y_{.1}$ on $(Y_{.2},Y_{.3})$ is prohibited, as \ref{hyp2} does not hold);
(iii) using again \ref{hyp2}, 
$$\mathbb{E}\left[Y_{.2}|Y_{.1}, Y_{.3},\Omega_{.1}=1\right]=\mathbb{E}\left[\mathcal{B}_{2\rightarrow1,3[0]}^c + \mathcal{B}_{2\rightarrow1,3[1]}^c Y_{.1} + \mathcal{B}_{2\rightarrow1,3[3]}^c Y_{.3}|Y_{.1},Y_{.3}\right],$$
and taking the expectation leads to
$$\mathbb{E}\left[Y_{.2}\right]=\mathcal{B}_{2\rightarrow1,3[0]}^c + \mathcal{B}_{2\rightarrow1,3[1]}^c\mathbb{E}\left[Y_{.1}\right] + \mathcal{B}_{2\rightarrow1,3[3]}^c\mathbb{E}\left[Y_{.3}\right].$$
The latter expression can be reshuffled so that
the expectation of $Y_{.1}$ can be estimated: the means of $Y_{.2}$ and $Y_{.3}$ are estimated by standard empirical estimators (it will be  Assumption \ref{hyp4} in the sequel).

\begin{multicols}{2}
\subparagraph{Using graphical arguments}
The PPCA model can be represented as structural causal graphs \citep{pearl2003causality}, as illustrated in Figure \ref{fig:proofToy}. Starting from the top left graph (in which each variable is generated by a combination of all latent variables, see \ref{hyp1}), one gets the top right one, as $Y_{.1}\leftarrow W_{.1} \rightarrow Y_{.2}$ is equivalent (see \citep[page 52]{pearl2003causality}) to $Y_{.1}\leftrightarrow Y_{.2}$. 
Then,  six reduced graphical models can be derived from the top right graph (two instances are represented in the bottom).
Indeed, a bidirected edge $Y_{.1} \leftrightarrow Y_{.2}$ can be interchanged (see \citep[rule 1, page 147]{pearl2003causality}) with an oriented edge $Y_{.1} \rightarrow Y_{.2}$, if each neighbor of $Y_{.2}$ (\textit{i.e.} $Y_{.1}$ or $Y_{.3}$) is inseparable of $Y_{.1}$ (see \citep[page 17]{pearl2003causality}). The bottom left graph can also be represented by Equation \eqref{eq:linear_toy}, which gives a connection between the algebraic and graphical approaches. 

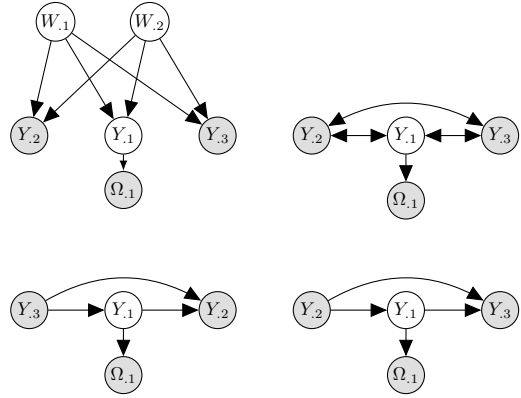
\begin{figure}[H]
        \vspace{0.3cm}
		\begin{tikzpicture}[scale=.7,every node/.style={scale=.7}]
		
		\node[obs]                               (y1) {$Y_{.2}$};
		\node[latent, right=0.75cm of y1]            (y2) {$Y_{.1}$};
		\node[obs, right=0.75cm of y2]            (y3) {$Y_{.3}$};
		\node[latent, above=of y1, xshift=0.5cm] (w1) {$W_{.1}$};
		\node[latent, above=of y2, xshift=0.5cm]  (w2) {$W_{.2}$};
		\node[obs,below of=y2, yshift=-1cm]                               (r2) {$\Omega_{.1}$};
		
		\edge{w1,w2} {y1,y2,y3} ; %
		\draw[->,>=latex] (y2) edge (r2);
		\node[obs, right=0.75cm of y3]                               (y1ter) {$Y_{.2}$};
		\node[latent, right=0.75cm of y1ter]            (y2ter) {$Y_{.1}$};
		\node[obs, right=0.75cm of y2ter]            (y3ter) {$Y_{.3}$};
		\node[obs,below of=y2ter, yshift=-1.2cm]                               (r2ter) {$\Omega_{.1}$};
		\node[left=0.6cm of y1ter]                               (placter) {};
		\node[left=0.4cm of r2ter]                               (placter2) {};
		
		\draw[<->] (y1ter) edge (y2ter);
		\draw[<->] (y1ter) edge[bend left] (y3ter);
		\draw[<->] (y2ter) edge (y3ter);
		\draw[->] (y2ter) edge (r2ter);
		\end{tikzpicture}
		
		\begin{tikzpicture}[scale=.6,every node/.style={scale=.7}]
		\node[obs] (y1cinq) {$Y_{.3}$};
		\node[latent, right=0.75cm of y1cinq]            (y2cinq) {$Y_{.1}$};
		\node[obs, right=0.75cm of y2cinq]            (y3cinq) {$Y_{.2}$};
		\node[obs,below of=y2cinq, yshift=-1.2cm]                               (r2cinq) {$\Omega_{.1}$};
		\node[above=0.75cm of y2cinq]                               (plac3) {};
		
		\draw[->] (y1cinq) edge (y2cinq);
		\node [right=0.75cm of y1cinq]    (y1cinqpoint) {};
		\draw (y1cinqpoint) node[above] {};
		\draw[->] (y1cinq) edge[bend left] (y3cinq);
		\node [above=0.4cm of y2cinq]    (y3cinqpoint) {};
		\draw (y3cinqpoint) node[above] {};
		\draw[->] (y2cinq) edge (y3cinq);
		\node [right=0.75cm of y2cinq]    (y2cinqpoint) {};
		\draw (y2cinqpoint) node[above] {};
		\draw[->] (y2cinq) edge (r2cinq);		\node[obs,right=0.75cm of y3cinq]    (y1cinqbis) {$Y_{.2}$};
		\node[latent, right=0.75cm of y1cinqbis]            (y2cinqbis) {$Y_{.1}$};
		\node[obs, right=0.75cm of y2cinqbis]            (y3cinqbis) {$Y_{.3}$};
		\node[obs,below of=y2cinqbis, yshift=-1.2cm]                               (r2cinqbis) {$\Omega_{.1}$};
		
		\draw[->] (y1cinqbis) edge (y2cinqbis);
		\node [right=0.85cm of y1cinqbis]    (y1cinqpointbis) {};
		\draw (y1cinqpointbis) node[above] {};
		\draw[->] (y1cinqbis) edge[bend left] (y3cinqbis);
		\node [above=0.4cm of y2cinqbis]    (y3cinqpointbis) {};
		\draw (y3cinqpointbis) node[above] {};
		\draw[->] (y2cinqbis) edge (y3cinqbis);
		\node [right=0.7cm of y2cinqbis]    (y2cinqpointbis) {};
		\draw (y2cinqpointbis) node[above] {};
		\draw[->] (y2cinqbis) edge (r2cinqbis);
		\end{tikzpicture}
	\caption{\label{fig:proofToy} Graphical models for the toy example with one missing variable $Y_{.1}$, $p=3$ and $r=2$.} 
\end{figure}
\end{multicols}

\subsection{Estimation of the mean, variance and covariances of the MNAR variables}

Estimators of the mean, variance and covariances of the  variables with MNAR values can be computed one by one.
That is why in the following, we detail the results only for a single variable, but we still  consider the case where several variables have MNAR values. It can easily be extended to the case where other variables can have MCAR and MAR values, as explained in Appendix \ref{sec:extensionmecha}.

We adopt an algebraic strategy to derive estimators (see Appendix \ref{sec:proofresults} for proofs) but graphical arguments can be used to obtain similar results (see Appendix \ref{sec:graph}).
 The starting point  is to exploit the linear links between variables, as described in the next lemma.

\begin{lem}\label{lem:PCAlinear}
	Under the PPCA model \eqref{eq:model} and Assumption \ref{hyp1}, choose $j\in \mathcal{J}$.
	One has
		\begin{equation}\label{eq:linear_jr}
		Y_{.j}=\mathcal{B}_{j\rightarrow m,\mathcal{J}_{-j}[0]}+\sum_{j'\in \mathcal{J}_{-j}} \mathcal{B}_{j\rightarrow m,\mathcal{J}_{-j}[j']} Y_{.j'} + \mathcal{B}_{j\rightarrow m,\mathcal{J}_{-j}[m]} Y_{.m} + \zeta,
    \end{equation}
	where $\zeta=-\sum_{j'\in\mathcal{J}_{-j}} \mathcal{B}_{j\rightarrow m,\mathcal{J}_{-j}[j']} \epsilon_{.j'} - \mathcal{B}_{j\rightarrow m,\mathcal{J}_{-j}[m]} \epsilon_{.m}+ \epsilon_{.j}.$ is a noise term.
	
	$\mathcal{B}_{j\rightarrow m,\mathcal{J}_{-j}[0]}$, $\mathcal{B}_{j\rightarrow m,\mathcal{J}_{-j}[j']}$ and $\mathcal{B}_{j\rightarrow m,\mathcal{J}_{-j}[m]}$ are given in Appendix \ref{sec:proofPCAlinar} and depend on the coefficients of $B$ given in \eqref{eq:model}.
\end{lem}

We then define the regression coefficients of $Y_{.j}$ on $Y_{.m}$ and $Y_{.k}$, for $k \in \mathcal{J}_{-j}$ in the complete case, that will be used to express the mean of a  variable with MNAR values. 

\begin{Def}[Coefficients in the complete case]
\label{def:coeff_complete_Case}
For $j\in \mathcal{J}$ and $k \in \mathcal{J}_{-j}$, let $\mathcal{B}_{j\rightarrow m,\mathcal{J}_{-j}[0]}^c$, $\mathcal{B}_{j\rightarrow m,\mathcal{J}_{-j}[m]}^c$ and $\mathcal{B}_{j\rightarrow m,\mathcal{J}_{-j}[j']}^c$ be respectively the intercept and the coefficients standing for the effects of $Y_{.j}$ on $(Y_{.m}, (Y_{.j'})_{j' \in \mathcal{J}_{-j}})$ in the complete case, \textit{i.e.} when $\Omega_{.m}=1$:
\begin{equation}\label{eq:coeffcompletecase}
    \left(Y_{.j}\right)_{|\Omega_{.m}= 1} := \mathcal{B}^c_{j\rightarrow m,\mathcal{J}_{-j}[0]}
    +\sum_{j'\in \mathcal{J}_{-j}} \mathcal{B}^c_{j\rightarrow m,\mathcal{J}_{-j}[j']} Y_{.j'} + \mathcal{B}^c_{j\rightarrow m,\mathcal{J}_{-j}[m]} Y_{.m} + \zeta^c, 
\end{equation}
with $\zeta^c=-\sum_{j'\in\mathcal{J}_{-j}} \mathcal{B}^c_{j\rightarrow m,\mathcal{J}_{-j}[j']} \epsilon_{.j'} - \mathcal{B}^c_{j\rightarrow m,\mathcal{J}_{-j}[m]} \epsilon_{.m}+ \epsilon_{.j}.$
\end{Def}

   Then, we make the two following assumptions:
\begin{enumerate}[label=\textbf{A\arabic*.}]
    \setcounter{enumi}{2}
    \item \label{hyp3}  $\forall j \in \mathcal{J}, \forall m \in \mathcal{M}$, the complete-case coefficients $\mathcal{B}_{j\rightarrow m,\mathcal{J}_{-j}[0]}^c$, $\mathcal{B}_{j\rightarrow m,\mathcal{J}_{-j}[m]}^c$ and $\mathcal{B}_{j\rightarrow m,\mathcal{J}_{-j}[k]}^c, k\neq j, k \in \mathcal{J}_{-j}$ can be consistently estimated. 
\end{enumerate}
\begin{enumerate}[label=\textbf{A\arabic*.}]
    \setcounter{enumi}{3}
    \item \label{hyp4}The means $(\alpha_{j})_{j\in\mathcal{J}}$, variances $(\mathrm{Var}(Y_{.j}))_{j\in\mathcal{J}}$ and covariances $(\mathrm{Cov}(Y_{.j},Y_{.j'}))_{j, j' \in\mathcal{J}}$, for $j\neq j'$ of the $r$ pivot variables can be consistently estimated.
\end{enumerate}
Note that Assumption \ref{hyp4} is  met whether the $r$ pivot variables are fully observed.

\begin{prop}[Mean estimator] \label{prop:mean_formula_general}
	Consider the PPCA model \eqref{eq:model}. Under Assumptions \ref{hyp1} and \ref{hyp2}, an estimator of the mean of a MNAR variable $Y_{.m}$, for $m\in\mathcal{M}$, can be constructed as follows: choose $j\in\mathcal{J}$, and compute
	\begin{equation}\label{eq:expectation_main}
		\hat{\alpha}_m :=\frac{\hat{\alpha}_{j}-\hat{\mathcal{B}}_{j\rightarrow m, \mathcal{J}_{-j}[0]}^c-\sum_{j'  \in \mathcal{J}_{-j}}\hat{\mathcal{B}}_{j\rightarrow m, \mathcal{J}_{-j}[j']}^c\hat{\alpha}_{j'}}{ \hat{\mathcal{B}}_{j\rightarrow m, \mathcal{J}_{-j}[m]}^c},
	\end{equation}
	with $(\hat{\mathcal{B}}^c_{j\rightarrow m, \mathcal{J}_{-j}[k]})_{k \in \{0,m\}\cup\mathcal{J}_{-j}}$ estimators of the coefficients obtained from Definition \ref{def:coeff_complete_Case}.
	
	Under the additional Assumptions \ref{hyp3} and \ref{hyp4}, this estimator is consistent. 
\end{prop}
The proof is given in Appendix \ref{sec:proofmean}.
Proposition \ref{prop:mean_formula_general} provides an estimator easily computable from complete observations.  Furthermore, different choices of $Y_{.j}$, $j\in\mathcal{J}$ can be done in Equation \eqref{eq:expectation_main}: all the resulting estimators may be aggregated to stabilize the estimation of $\alpha_m$.

\begin{prop}[Variance and covariances estimators] \label{prop:var_formula_general}
	Consider the PPCA model \eqref{eq:model}. Under Assumptions \ref{hyp1} and \ref{hyp2}, an estimator of the variance of a MNAR variable $Y_{.m}$, for $m\in\mathcal{M}$, ,and its covariances with the pivot variables, can be constructed as follows: choose a pivot variable $Y_{.j}$ for $j \in \mathcal{J}$ and compute
	\begin{equation}\label{eq:estimcov_main}
		\begin{pmatrix}
			\widehat{\mathrm{Var}}(Y_{.m}) &
			\widehat{\mathrm{Cov}}(Y_{.m},(Y_{.j'})_{j'\in \mathcal{J}})
		\end{pmatrix}^T:=(\widehat{M}_j)^{-1}\widehat{P}_j,
	\end{equation}
	assuming that $\sigma^2$ tends to zero,
	with $\widehat{M}_j^{-1} \in \mathbb{R}^{(r+1)\times(r+1)}$, $\widehat{P}_j \in \mathbb{R}^{r+1}$ 
	detailed in Appendix \ref{sec:proofvar}. These quantities depend on $(\hat{\alpha}_{j'})_{j'\in\mathcal{J}}$, $\hat{\alpha}_m$ given in Proposition \ref{prop:mean_formula_general}, and on $(\widehat{\mathrm{Var}}(Y_{.j})_{j \in \mathcal{J}}$ and 
	on complete-case coefficients such as $(\hat{\mathcal{B}}^c_{j'\rightarrow m,\mathcal{J}_{-j'}[k]})_{k \in \{m\}\cup\mathcal{J}_{-j'}}$ for $j' \in \mathcal{J}$.
	
	Under the additional Assumptions \ref{hyp3} and \ref{hyp4}, 
	the estimators of the variance of $Y_{.m}$ and its covariances with the pivot variables given in \eqref{eq:estimcov_main} are consistent.
\end{prop}
The proof is given in Appendix \ref{sec:proofvar}. Note that to estimate the variance of a MNAR variable, only $r$ pivot variables are required  to solve \eqref{eq:estimcov_main} and $r$ tasks have to be performed for estimating the coefficients of the effects of $Y_{.k}$ on $(Y_{.\ell})_{\ell \in \{m\}\cup\mathcal{J}_{-k}}$ for all $k \in \mathcal{J}$. 

All the ingredients can be combined to form an estimator $\hat{\Sigma}$ for the covariance matrix \eqref{eq:model}. Define
\begin{equation}\label{eq:cov_matrix_general}
    \hat{\Sigma} := \left(\widehat{\mathrm{Cov}}(Y_{.k},Y_{.\ell})\right)_{k,\ell \in \{1,\dots,p\}},
\end{equation}
where
\begin{itemize}
	\item if $Y_{.k}$ and  $Y_{.\ell}$ have both consistent mean/variance estimators, then $\widehat{\mathrm{Cov}}(Y_{.k},Y_{.\ell})$ can be trivially evaluated by standard empirical covariance estimators.
	\item if $Y_{.k}$ is a MNAR variable and $Y_{.\ell}$ is a pivot variable, then $\widehat{\mathrm{Cov}}(Y_{.k},Y_{.\ell})$ is given by \eqref{eq:estimcov_main},
 	\item if $Y_{.k}$ is a MNAR variables and $Y_{.\ell}$ is not a pivot variable, \textit{i.e.} $\ell \in \widebar{\mathcal{J}}\setminus \{k\}$, a similar strategy as the one above can be devised. Then $\widehat{\mathrm{Cov}}(Y_{.k},Y_{.\ell})$ is given by \eqref{eq:covmissvar} detailed in Appendix \ref{sec:proofcov2} and for which some additional assumptions similar as the ones above are required.
	This estimator relies on the choice of $r-1$ pivot variables indexed by $j$ and $\mathcal{H}\subset\mathcal{J}$, and only necessitates to evaluate the effects of $Y_{.j}$ on $(Y_{.j'})_{j' \in \{k,\ell\}\cup\mathcal{H}}$ in the complete case.
\end{itemize}

\subsection{Performing PPCA with MNAR variables}

With the estimator $\hat{\Sigma}$ given in \eqref{eq:cov_matrix_general}, one performs the estimation of the loading matrix $B$ in \eqref{eq:model}.
\begin{Def}[Estimation of the loading matrix]\label{def:estPPCAcoeff}
	Given the estimator $\hat{\Sigma}$ of the covariance matrix in \eqref{eq:cov_matrix_general}, let the orthogonal matrix $\hat{U} = (\hat{u}_1 | \hdots | \hat{u}_p) \in \mathbb{R}^{p\times p}$ and the diagonal matrix $\hat{D} = \mathrm{diag}(\hat{d}_1, \hat{d}_2, \hdots, \hat{d}_p)\in \mathbb{R}^{p\times p}$ with $d_1\geq d_2 \geq \hdots \geq d_p \geq 0$ form the singular value decomposition of the following matrix $\hat{\Sigma} - \sigma^2 \mathrm{Id}_{p\times p} =: \hat{U} \hat{D} \hat{U}^T.$
	An estimator $\hat{B}$ of $B$ can be defined using the $r$ first singular values and vectors, as follows
	\begin{align}
		\label{eq:Bhat}
		\hat{B} =  \hat{D}_{|r}^{1/2} \hat{U}_{|r}^T = \mathrm{diag}({\hat{d}_1},\dots,{\hat{d}_r})^{1/2} (\hat{u}_1^T | \hdots | \hat{u}_r^T)^T
	\end{align}
\end{Def}
The estimation of the loading matrix is used to impute the variables with missing values. 
More precisely, a classical strategy to impute missing values is to estimate their conditional expectation given the observed values. 
One can note that with  $\Sigma=B^TB+\sigma^2\mathrm{Id}_{p\times p}$, the conditional expectation of $Y_{.m}$ for $m\in \mathcal{M}$ given $(Y_{.k})_{k\in \widebar{\mathcal{M}}}$ 
 reads as follows
$$\mathbb{E}[Y_{.m}|(Y_{.k})_{k\in \widebar{\mathcal{M}}}]=
{\alpha}_m 
+ \Sigma_{m,\widebar{\mathcal{M}}} \Sigma_{\widebar{\mathcal{M}},\widebar{\mathcal{M}}}^{-1} 
\left( Y_{. \widebar{\mathcal{M}}}^T-\alpha_{\widebar{\mathcal{M}}} \right), 
$$
with $\Sigma_{m,\widebar{\mathcal{M}}}:=(\Sigma_{m,k})_{k\in \widebar{\mathcal{M}}}^T$, $\Sigma_{\widebar{\mathcal{M}},\widebar{\mathcal{M}}} :=(\Sigma_{k,k'})_{k,k'\in \widebar{\mathcal{M}}}$, $Y_{. \widebar{\mathcal{M}}} := (Y_{.k})_{k\in \widebar{\mathcal{M}}}$, and $\alpha_{\widebar{\mathcal{M}}} := (\alpha_{k})_{k\in \widebar{\mathcal{M}}}$.

\begin{Def}[Imputation of a MNAR variable]
Set $\hat{\Gamma}:=\hat{B}^T\hat{B}+\sigma^2\mathrm{Id}_{p\times p}$ for $\hat{B}$ given in Definition \ref{def:estPPCAcoeff}. The MNAR variable $Y_{.m}$ with $m\in \mathcal{M}$ can be imputed as follows: for $i$ such that $\Omega_{i,m}=0$,
\begin{align}
    \label{eq:imputation_MNAR_var}
    \hat{Y}_{im} &= \hat{\alpha}_m + \hat{\Gamma}_{m,\widebar{\mathcal{M}}}
		    \hat{\Gamma}_{\widebar{\mathcal{M}},\widebar{\mathcal{M}}}^{-1} 
		    \left(Y_{i,\widebar{\mathcal{M}}}^T - \hat{\alpha}_{\widebar{\mathcal{M}}} \right)
\end{align}
with $\hat{\Gamma}_{m,\widebar{\mathcal{M}}}:=(\hat{\Gamma}_{m,k})_{k\in \widebar{\mathcal{M}}}^T$, $\hat{\Gamma}_{\widebar{\mathcal{M}},\widebar{\mathcal{M}}} :=(\hat{\Gamma}_{k,k'})_{k,k'\in \widebar{\mathcal{M}}}$, $Y_{. \widebar{\mathcal{M}}} := (Y_{.k})_{k\in \widebar{\mathcal{M}}}$ and $\hat{\alpha}_{\widebar{\mathcal{M}}} := (\hat{\alpha}_{k})_{k\in \widebar{\mathcal{M}}}$. 
\end{Def}

\subsection{Algorithm}\label{sec:algo}


The proposed imputation method is described in Algorithm \ref{alg:imputation} and can handle different MNAR mechanisms: self-masked MNAR case but also cases where the probability to have missing values on variables depends on both the underlying values and values of other variables (observed or missing). 

\setlength{\columnseprule}{1pt}

\begin{algorithm}[H]
    \begin{algorithmic}[1]
    \REQUIRE $r$ (number of latent variables), $\sigma^2$ (noise level), $\mathcal{J}$ (pivot variables indices), $\Omega$ (mask).
    \end{algorithmic}
    \vspace{-0.3cm}
    \begin{multicols}{2}
	\begin{algorithmic}[1]
		\FOR{each MNAR variable $(Y_{.m})_{m\in \mathcal{M}}$}
		\STATE Evaluate $\hat{\alpha}_m$ the estimator of its mean given in \eqref{eq:expectation_main} using the $r$ pivot variables indexed by $\mathcal{J}$.
 		\STATE Evaluate $\widehat{\mathrm{Var}}(Y_{.m})$, and $\widehat{\mathrm{Cov}}(Y_{.m},Y_{.\ell})$ for $\ell \in \mathcal{J}$, using \eqref{eq:estimcov_main}. 
 		\STATE Evaluate $\widehat{\mathrm{Cov}}(Y_{.m},Y_{.\ell})$ for $\ell\in\widebar{\mathcal{J}}\setminus \{m\}$ using  Proposition \ref{prop:covmiss}. 
		\ENDFOR
		\columnbreak
		\STATE Form $\hat{\Sigma}$, covariance matrix estimator in \eqref{eq:cov_matrix_general}.\STATE Compute the loading matrix estimator $\hat{B}$ given in \eqref{eq:Bhat}.
		\STATE Compute 
		$\hat{\Gamma}=\hat{B}^T\hat{B}+\sigma^2\mathrm{Id}_{p\times p}.$
		\FOR{each missing variable $(Y_{.j})$}
		    \FOR{$i$ such that $\Omega_{ij}=0$}
		    \STATE $\hat{Y}_{ij} \leftarrow$ Impute $Y_{ij}$ as in \eqref{eq:imputation_MNAR_var}.
		    \ENDFOR
		 \ENDFOR
	\end{algorithmic}
	\end{multicols}
	\vspace{-0.3cm}
	\caption{\label{alg:imputation} PPCA with MNAR variables.}
\end{algorithm}


Algorithm \ref{alg:imputation} requires the set $\mathcal{J}$, \textit{i.e.} the selection of $r$ pivot variables on which the regressions in Propositions \ref{prop:mean_formula_general}, \ref{prop:var_formula_general} and \ref{prop:covmiss} will be performed.  If there are more than $r$ variables that can be pivot, the final estimator is provided by computing the median of the estimators over all possible combinations of $r$ pivot variables. 
In addition, in order to estimate the coefficients in Definition \ref{def:coeff_complete_Case}, we use  ordinary least squares despite that the exogeneity assumption, \textit{i.e.} the noise term is independent of the covariates, does not hold. It still leads to accurate estimation in numerical experiments as shown in Section \ref{sec:simu}. 


\section{Numerical experiments}\label{sec:simu}
\subsection{Synthetic data}
\label{sec:simu_synthdata}
We empirically compare Algorithm \ref{alg:imputation} (\textbf{MNAR}) to the state-of-the-art methods, including 
\begin{enumerate}[label={(\roman*)},topsep=0pt]
\item \textbf{MAR}\label{methMAR} our method which has been adapted to handle MAR data (inspired by \cite[Theorems 1, 2, 3]{mohan2018estimation} in linear models),  see Appendix \ref{sec:MARcase} for details;
\item \textbf{EMMAR}\label{methEMMAR}:  EM algorithm to perform PPCA with MAR values \citep{ilin2010practical};
\item \textbf{SoftMAR}\label{methSoft}: matrix completion  using iterative soft-thresholding singular value decomposition algorithm  \citep{mazumder2010spectral} relevant only for M(C)AR values;
\item \textbf{MNARparam}\label{methparam}: matrix completion technique modeling the MNAR mechanism with a  parametric logistic model  \citep{sportisse2018imputation}. 
\end{enumerate}
Note that method \ref{methEMMAR} is specially designed to estimate the PPCA loading matrix  and not to perform imputation, but this is possible combining Method \ref{methEMMAR} with steps 8 and 9 in Algorithm \ref{alg:imputation}. This is the other way around for completion methods \ref{methSoft} and \ref{methparam}, but the loading matrix can be computed as in \eqref{eq:Bhat}. Note also that methods \ref{methSoft} and \ref{methparam} are developed in a context of fixed effects low-rank models. 
They require tuning a regularization parameter $\lambda$; we consider an oracle value minimizing the true prediction error. We also use oracle values for the noise level and the rank in Algorithm \ref{alg:imputation}. These methods are compared with the imputation by the mean (\textbf{Mean}), 
which serves as a benchmark, and the naive listwise deletion method (\textbf{Del}) which consists in estimating the parameters empirically with the fully-observed data only. 

\paragraph{Measuring the performance} For the loading matrix, the RV coefficient \citep{josse2008testing}, which is a measure of relationship between two random vectors, between the estimate $\hat{B}$ and the true $B$ is computed.
An RV coefficient close to one means high correlation between the image spaces of  $\hat{B}$ and $B$. Denoted the Frobenius norm as $\|.\|_F$, the quality of imputation  is measured with the normalized prediction error given by $\| (\hat{Y}-Y) \odot (1-\Omega)\|^2_F/\left\| Y \odot (1-\Omega) \right\|_F^2$.
A discussion on computational times can be found in Appendix \ref{sec:comp_time}.


\subparagraph{Setting}
\setlength{\columnseprule}{0pt}
\begin{multicols}{2}
 We generate a data matrix of size $n=1000$ and $p=10$  from a  PPCA model \eqref{eq:model} with two latent variables ($r=2$) and with a noise level $\sigma=0.1$. Missing values are introduced on seven variables $(Y_{.k})_{k \in [1:7]}$ 
 according to a logistic self-masked MNAR mechanism, leading to $35\%$ of missing values in total. Results are presented for one missing variable (same results hold  for other missing variables). All the observed variables $(Y_{.k})_{k \in [8:10]}$ 
 are considered to be pivot. Figure \ref{fig:MeanVariancesY1} shows that Algorithms \ref{alg:imputation} is the only one which always gives unbiased estimators of the mean, variance and associated covariances of $Y_{.1}$. As expected, the listwise deletion method provides biased estimates inasmuch as the observed sample is not representative of the population with MNAR data. Method \ref{methEMMAR}, specifically designed for PPCA models but assuming MAR missing values, provides biased estimators. Method \ref{methparam} improves on the benchmark mean imputation and on Method \ref{methSoft} as well, as it explicitly takes into account the MNAR  mechanism, but it still leads to biased estimates probably because of the fixed effect model assumption.
  \columnbreak
\begin{figure}[H]
\vspace{-0.8cm}
\hspace{-0.3cm}
\includegraphics[width=0.5\textwidth]{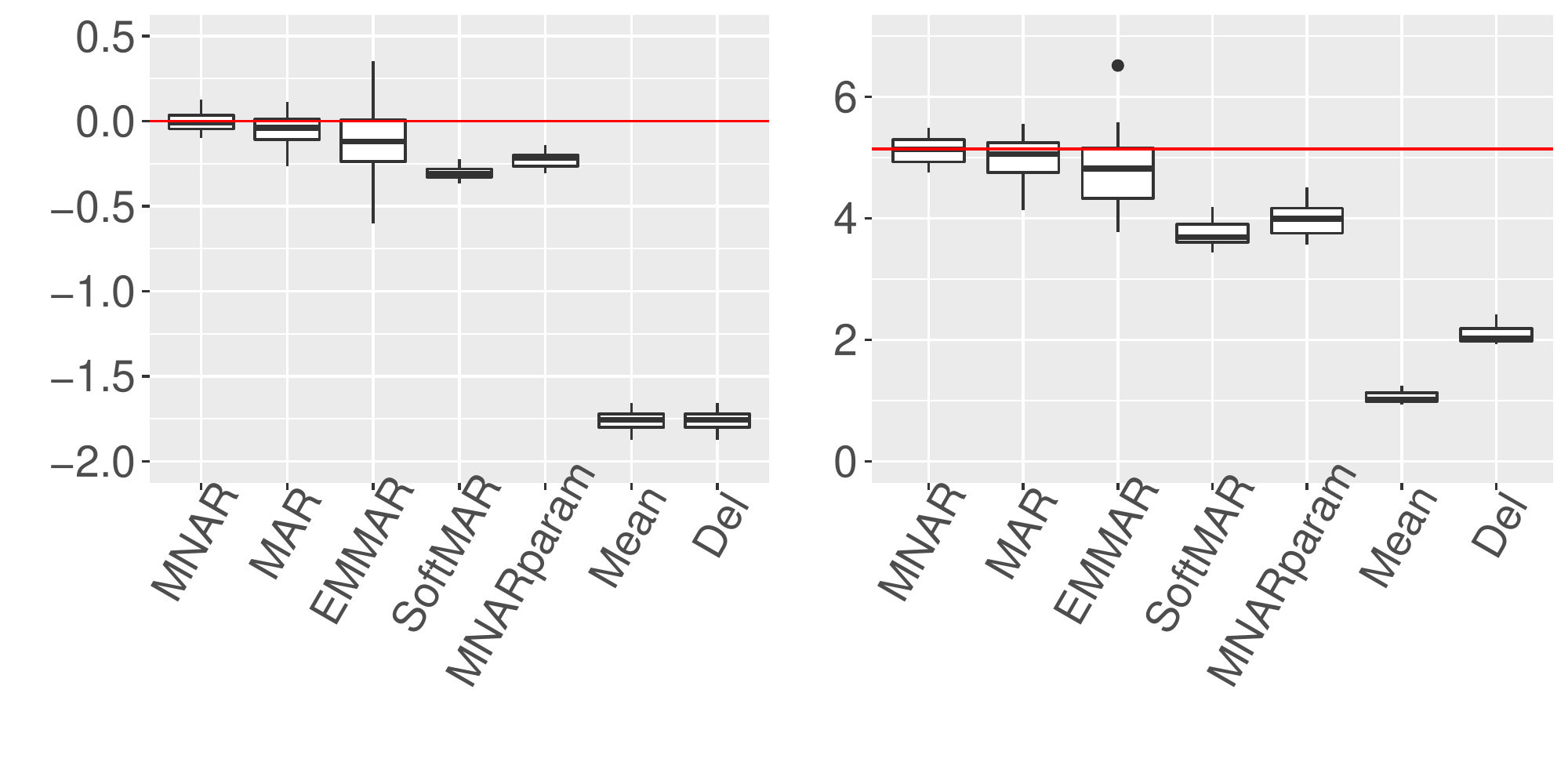}
\par\vspace{-0.25cm}
\hspace{-0.3cm}
\includegraphics[width=0.5\textwidth]{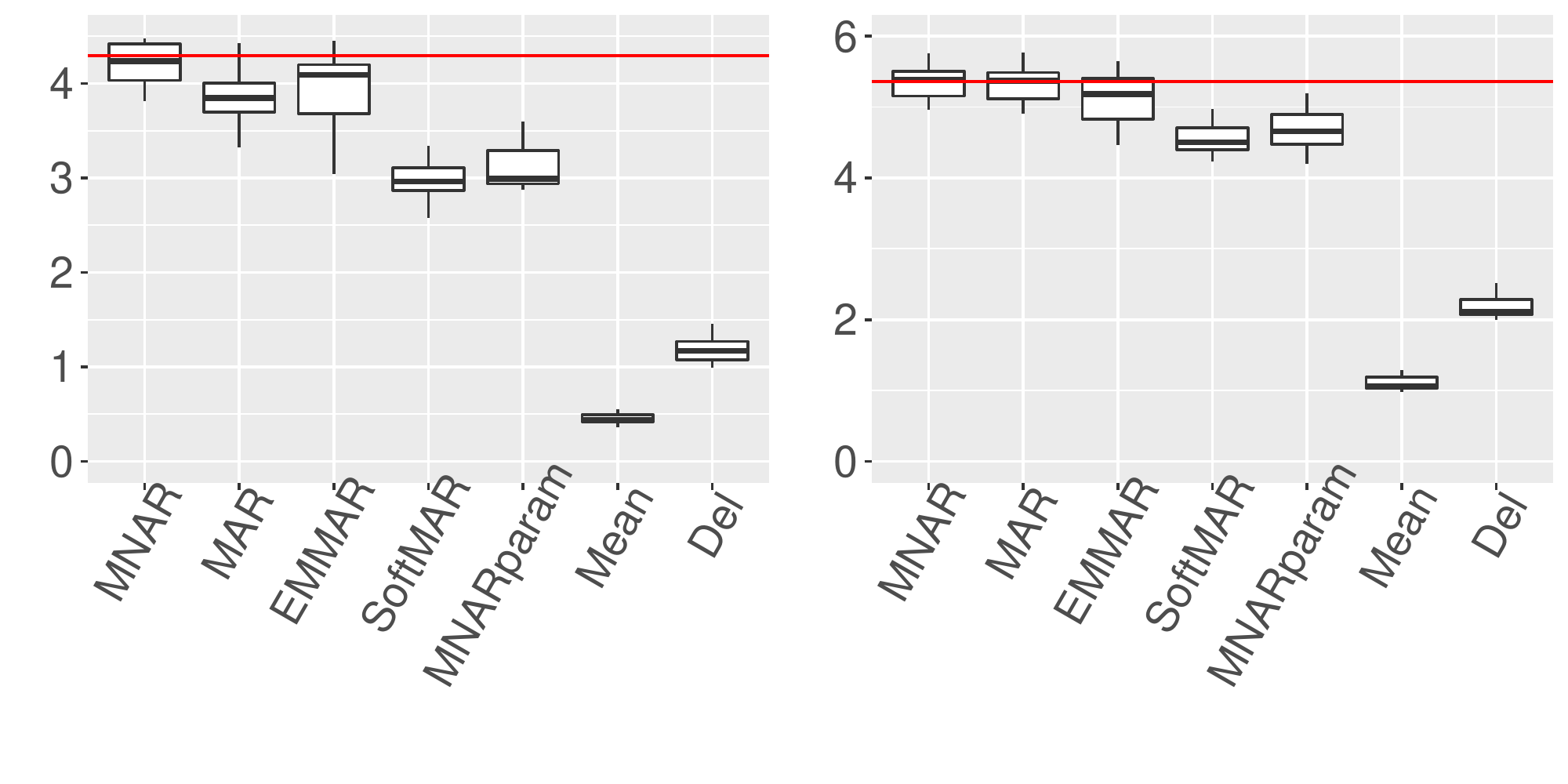}
\vspace{-0.8cm}
\caption{\label{fig:MeanVariancesY1} Mean and variance estimations of the missing variable $Y_{.1}$ (top left and right graphics) and  covariances estimations (bottom graphics) of $\textrm{Cov}(Y_{.1},Y_{.2})$ (\textit{i.e.} covariance between two missing variables) and of $\textrm{Cov}(Y_{.1},Y_{.8})$ (\textit{i.e.} between one missing variable and one pivot variable). True values  are indicated by red lines.}
\end{figure}
\end{multicols}

\setlength{\columnseprule}{0pt}
\begin{multicols}{2}
Figure \ref{fig:MSECorrY1} shows that Algorithm \ref{alg:imputation} gives the best estimate of the loading matrix and the smallest imputation error.  
Biases in estimation results being lower, Method \ref{methMAR}, based on same arguments as Algorithm \ref{alg:imputation} but considering MAR data, may be considered as a second choice for this low-dimensional example (yet not in higher dimension, see Appendix \ref{sec:othernumexp}). 

\begin{figure}[H]
    \captionsetup[subfigure]{labelformat=empty}
	\begin{subfigure}[H]{0.3\textwidth}
		\includegraphics[width=1\textwidth]{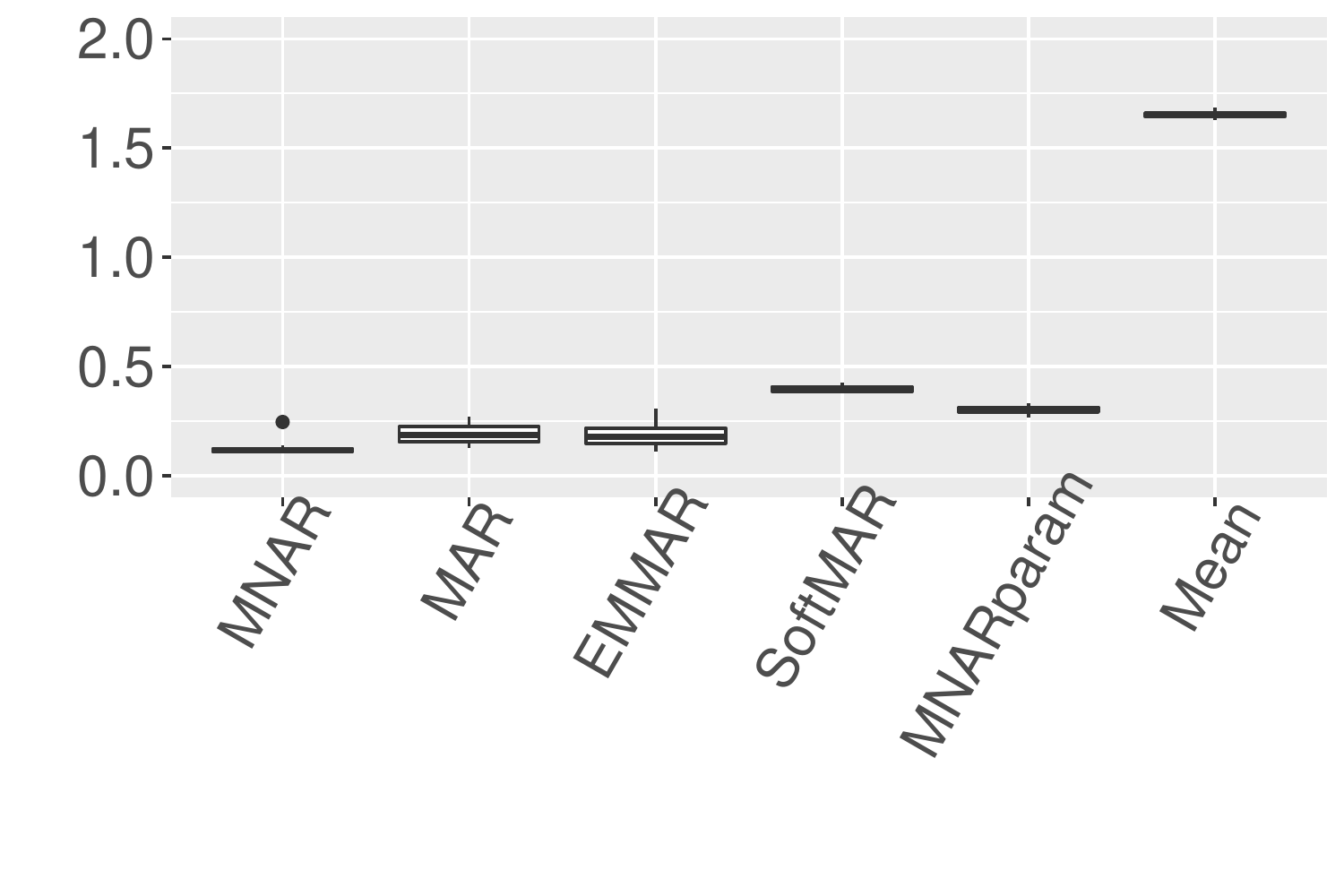}
		\caption{}
	\end{subfigure}
	\hspace{0.05cm}
	\begin{subfigure}[H]{0.165\textwidth}
	\vspace{-0.8cm}
	{\scriptsize
	\begin{tabular}{c|c}
		& \textrm{RV} \\
		\textrm{MNAR} & {$0.997$} \\
		\textrm{MAR} & $0.993$ \\
		\textrm{EMMAR} &  $ 0.988 $\\
		\textrm{SoftMAR} &   $0.986$ \\
		\textrm{Mean} &  $0.593$
	\end{tabular}
	}
		\caption{}
	\end{subfigure}
	\vspace{-1.2cm}
	\caption{\label{fig:MSECorrY1}Prediction error (left) and median of the RV coefficients for the loading matrix (right).}
\end{figure}
\end{multicols}
\vspace{-0.4cm}
In Appendix \ref{sec:othernumexp}, we report further simulation results, where we vary the features dimension ($p=50$), the rank ($r=5$),  the missing values mechanism using probit self-masking and also multivariate MNAR (when the probability to be missing for a variable depends on its underlying values and on values of other variables that can be missing) and the percentage of missing values (10\%, 50\%). 
The results obtained on the simulations presented before are representative of other results obtained with different number of variables, ranks and mechanisms. Besides, as expected, all the methods deteriorate with an increasing percentage of missing values but our method is stable. 

We also  assess the robustness of the methods in terms of noise,  model misspecification (assuming a fixed effect model) and we evaluated the impact of underestimating or overestimating the  number $r$ of latent variables.
When we increase the level of noise, our method is very robust in terms of mean and variance estimations, and despite a bias for some covariances estimation for large noise it outperforms competitors regarding  the prediction error.  When data are simulated according to a low-rank fixed effect model, Algorithm \ref{alg:imputation} provides results close to the ones of Method \ref{methparam} which is specifically developed for this model and MNAR data. Moreover, it turns out that the procedure remains stable at a wrong specification of the number of latent variables $r$. These extensive simulations highlight that our approach is stable to model misspecifications of the PPCA assumption. 

\subsection{Application to clinical data}
\label{sec:simu_realdata}
We illustrate our method on the TraumaBase$^{\mbox{\normalsize{\textregistered}}}$ dataset containing the clinical measurements of 3159 patients with brain trauma injury.
Nine quantitative variables, selected by doctors, contain from 1 to 30\% missing values, leading to 11\% in the whole dataset. After discussion with doctors, some variables can be considered to have informative missing values, such as the variable \textit{HR.ph}, which denotes the heart rate. 
Indeed, when the patient’s condition is too critical and therefore his heart rate is either high or low, the heart rate may not be measured, as doctors prefer to provide emergency care. 
Both percentage and nature of missing data demonstrate the importance of taking appropriate account of missing data. More information on the data can be found in Appendix \ref{sec:appdata}.

\vspace{-0.15cm}
\paragraph{Imputation performances}

To assess the quality of our method, we introduce additional MNAR values in the variable \textit{HR.ph} (which has an initial missing rate of 1\%) using a logistic self-masked mechanism 
leading to 50\% missing values.
The other variables are considered M(C)AR. 
\vspace{-0.15cm}
\begin{multicols}{2}
The noise level is estimated using the mean of the last eigenvalues \citep{josse2016denoiser} and the rank of $Y$ is estimated using cross-validation \citep{josse2012selecting}. Both quantities are estimated using the complete-case analysis (1862 observations).  In Figure \ref{fig:MSERealData}, 
the prediction error is relative to the error of the benchmark imputation by the mean.
Algorithm \ref{alg:imputation} gives significantly smaller prediction error than other methods. 
\begin{figure}[H]
\vspace{0.05cm}
\centering
\includegraphics[width=0.4\textwidth]{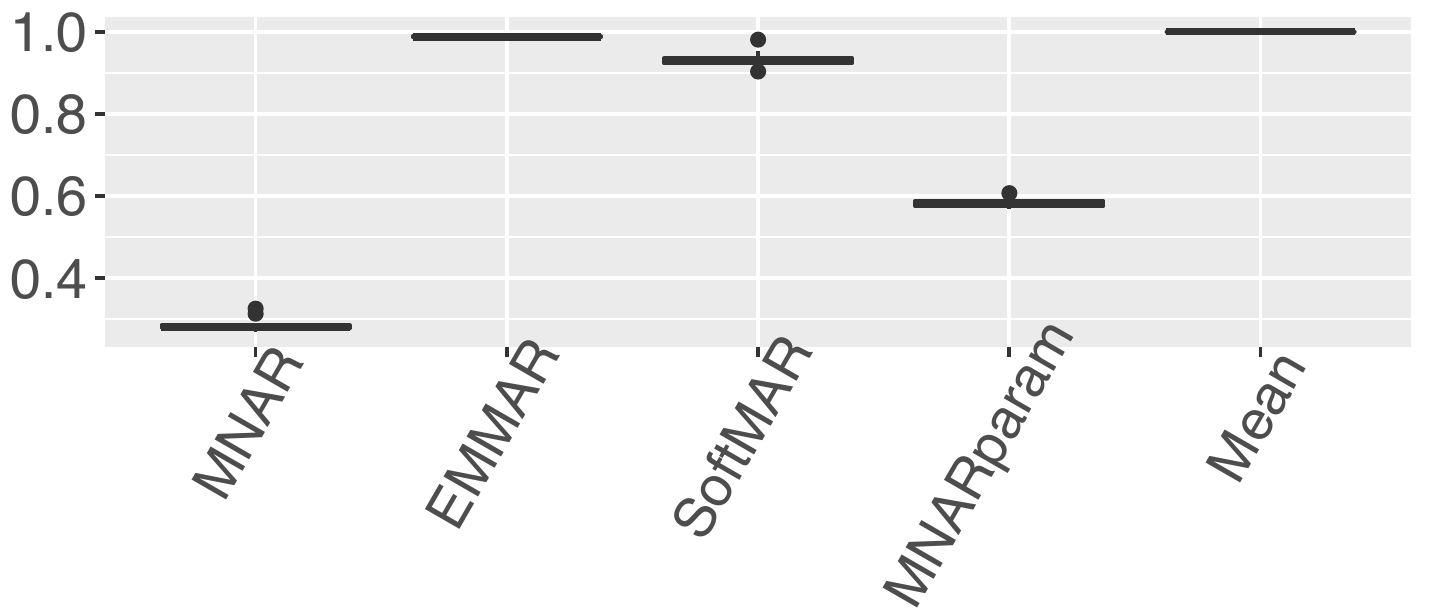}
\vspace{-0.3cm}
\caption{\label{fig:MSERealData} Comparison of the prediction error over 10 replications for the TraumaBase data.}
\end{figure}
\end{multicols}


\vspace{-0.5cm}

\section*{Conclusion}
\addcontentsline{toc}{section}{Conclusion}
\vspace{-0.1cm}
In this work, we propose a new estimation and imputation method to perform PPCA with MNAR data (possibly coupled with M(C)AR data), without any need of modeling the missing mechanism. 
This comes with strong theoretical guarantees as identifiability and consistency, but also with an efficient algorithm.
Estimating the rank in the PPCA setting with MNAR data remains non trivial. Once the number of latent variables is estimated, the noise variance can be estimated. A cross-validation strategy by additionally adding some MNAR values is a first solution, but this definitely requires further research.
Another ambitious prospect would be to extend work to the exponential family to process count data, for example, which is prevalent in many application fields such as genomics.

\bibliographystyle{plainnat}
\bibliography{biblio}

\appendix
\section{Proof of Proposition \ref{prop:identifiability}}
\label{sec:proofidentif}

For the sake of readability, we first present the proof of Proposition \ref{prop:identifiability} in the case of the toy example presented in Section \ref{sec:toyexample} with  $p=3$ and $r=2$. The proof in the general setting follows.

\subsection{Proof of Proposition \ref{prop:identifiability} in the case of the toy example presented in Section \ref{sec:toyexample}}

Consider the setting of the toy example presented in Section \ref{sec:toyexample} with  $p=3$ and $r=2$. The PPCA model in \eqref{eq:model} reads
$$
\left\{
\begin{array}{ll}
Y&=\begin{pmatrix} 
Y_{1} & Y_{2} & Y_{3}
\end{pmatrix}=\begin{pmatrix} \alpha_1 & \alpha_2 & \alpha_3 \end{pmatrix} + \begin{pmatrix}W_{1} & W_{2}\end{pmatrix}B + \epsilon, \\
Y &\sim \mathcal{N}(\alpha,\Sigma), \: \Sigma=B^TB+\sigma^2 I.
\end{array}
\right.
$$
$Y_2$ and $Y_3$ are assumed to be observed and $Y_{1}$ is self-masked MNAR, \textit{i.e.}
\begin{equation}\label{eq:omega1}
\mathbb{P}(\Omega_{1}=1|Y_{1},Y_{2},Y_{3};\phi_1)=\mathbb{P}(\Omega_{1}=1|Y_{1};\phi_1)=F_1(\phi_1^0+\phi_1^1 y_{1}),
\end{equation}
where $F_1$ is strictly monotone with a positive finite support.

\begin{proof}
Assume that $(Y,\Omega)$ and $(Y',\Omega')$ have distributions respectively parameterized by $(\alpha,\Sigma, \phi_1)$ and $(\alpha',\Sigma', \phi_1')$.
Assume that $Y$ and $Y'$ have the same observed distribution, \textit{i.e.}
\begin{align}
\label{eq:id_lawequality_toy}
 \mathcal{L}(Y_1,\Omega_{1}=1;\alpha_1,\Sigma_{11},\phi_1)&=\mathcal{L}(Y'_{1},\Omega'_{1}=1;\alpha'_1,\Sigma'_{11},\phi'_1) \\
\label{eq:id_lawequality_twovar_toy}
  \mathcal{L}(Y_1,Y_j,\Omega_{1}=1;\alpha_1,\alpha_j,\Sigma_{(1j)},\phi_1)&=\mathcal{L}(Y'_{1},Y'_{j},\Omega'_{1}=1;\alpha'_1,\alpha'_j,\Sigma'_{(1j)},\phi'_1) \qquad j \in \{2,3\},
\end{align}
where $\Sigma_{(1j)}$ is the covariance matrix $\begin{pmatrix}
\Sigma_{11} & \Sigma_{1j} \\
\Sigma_{1j} & \Sigma_{jj}
\end{pmatrix}$. 
In order to show that parameters identifiability holds, we need to show that \eqref{eq:id_lawequality_toy} and \eqref{eq:id_lawequality_twovar_toy} imply that $\alpha=\alpha'$, $\Sigma=\Sigma'$ and $\phi_1=\phi_1'$. Then, under a known noise level $\sigma^2$, we prove that $B$ and $B'$ are equal up to a row permutation. 

As $(Y_{2},Y_{3})$ and $(Y'_{2},Y'_{3})$ are fully observed, the parameters of the  distributions $\mathcal{L}(Y_{2})$, $\mathcal{L}(Y'_{2})$,  $\mathcal{L}(Y_{3})$, $\mathcal{L}(Y'_{3})$, $\mathcal{L}(Y_{2},Y_{3})$ and $\mathcal{L}(Y'_{2},Y'_{3})$ are identifiable. 
It trivially implies that $\alpha_2=\alpha'_2$, $\Sigma_{22}=\Sigma'_{22}$, $\alpha_3=\alpha'_3$, $\Sigma_{33}=\Sigma'_{33}$ and $\Sigma_{23}=\Sigma'_{23}$.

\subparagraph{Identifiability of the MNAR variable variance}

Equation \eqref{eq:id_lawequality_toy} can be rewritten in terms of density function as follows
$$f_{Y_{1},\Omega_{1}=1}(y_1;\alpha_1,\Sigma_{11},\phi_1)=f_{Y'_{1},\Omega'_{1}=1}(y_1;\alpha'_1,\Sigma'_{11},\phi'_1) \qquad \forall y_1 \in \mathbb{R}.$$
Given the missing mechanism in \eqref{eq:omega1} and that $Y_{.1} \sim \mathcal{N}(\alpha_1,\Sigma_{11})$,  \cite[Theorem 1 a)]{miao2016identifiability} ensures that $\Sigma_{11}=\Sigma'_{11}$.

\subparagraph{Identifiability of the Mean and the MNAR mechanism parameter}

Using \eqref{eq:id_lawequality_toy} and \eqref{eq:id_lawequality_twovar_toy}, the previous computations entail that
$$\mathcal{L}(Y_{2}|Y_{1},\Omega_{1}=1;\alpha_1,\alpha_2,\Sigma_{(12)},\phi_1)
=\mathcal{L}(Y'_2|Y'_{1},\Omega'_{1}=1;\alpha'_1,\alpha'_2,\Sigma'_{(12)},\phi'_1),$$
noting that $$f_{Y_{2}|Y_{1}=y_1,\Omega_{1}=1}(y_2;\alpha_1,\alpha_2,\Sigma_{(12)},\phi_1)=\frac{f_{Y_{1},Y_{2},\Omega_{1}=1}(y_1,y_2;\alpha_1,\alpha_2,\Sigma_{(12)},\phi_1)}{f_{Y_1,\Omega_{1}=1}(y_1;\alpha_1,\Sigma_{11},\phi_1)} \qquad \forall (y_1,y_2) \in \mathbb{R}^2$$

One obtains
\begin{multline*}
\frac{\mathbb{P}(\Omega_{1}=1|Y_1 =y_1,Y_2=y_2;\phi_1)f_{Y_{2}|Y_{1}=y_1}(y_2;\alpha_1,\alpha_2,\Sigma_{(12)})}{\mathbb{P}(\Omega_{1}=1|Y_{1}=y_{1};\phi_1)} \\
=\frac{\mathbb{P}(\Omega'_{.1}=1|Y'_1 =y_1,Y'_2=y_2;\phi'_1)f_{Y'_{2}|Y'_{1}=y_1}(y_2;\alpha'_1,\alpha'_2,\Sigma'_{(12)})}{\mathbb{P}(\Omega'_{1}=1|Y_{1}=y_{1};\phi'_1)} \qquad \forall (y_1,y_2) \in \mathbb{R}^2
\end{multline*}

Yet, 
\begin{align}
\nonumber
\mathbb{P}(\Omega_{1}=1|Y_{1} =y_1,Y_{2} =y_2;\phi_1)&=\mathbb{E}[\mathbb{E}[\mathrm{1}_{\Omega_{1}=1}|Y_{1} =y_1,Y_{2} =y_2,Y_{3}=y_3;\phi_1]|Y_{1} =y_1,Y_{2} =y_2] \\
\nonumber
&=\mathbb{E}[\mathbb{P}(\Omega_{1}=1|Y =y;\phi_1)|Y_{1} =y_1,Y_{2} =y_2] \\
\nonumber
&=\mathbb{E}[\mathbb{P}(\Omega_{1}=1|Y_1 =y_1;\phi_1)|Y_{1} =y_1,Y_{2} =y_2] \\
\label{eq:tiptoy}
&=\mathbb{P}(\Omega_{1}=1|Y =y_1;\phi_1)
\end{align}
by measurability. It implies for all $y_1\in \mathbb{R}$ and $y_2\in \mathbb{R}$
$$f_{Y_{2}|Y_{1}=y_1}(y_2;\alpha_1,\alpha_2,\Sigma_{(12)})
=f_{Y'_{2}|Y'_{1}=y_1}(y_2;\alpha'_1,\alpha'_2,\Sigma'_{(12)})$$
which leads to the equality of the conditional expectations and variances associated to the above densities:
\begin{align*}
\alpha_{2}+\Sigma_{12} \Sigma_{11}^{-1}(\alpha_1-y_1)&=\alpha_{2}+\Sigma'_{12} \Sigma_{11}^{-1}(\alpha'_1-y_1) \qquad \forall y_1 \in \mathbb{R}\\
\Sigma_{22}-\Sigma_{12}^2\Sigma_{11}^{-1}&=\Sigma_{22}-(\Sigma'_{12})^2\Sigma_{11}^{-1}.
\end{align*}

It implies that
\begin{align}
\label{eq:sigma12toy}
\Sigma_{12}^2=(\Sigma'_{12})^2 &\Longrightarrow |\Sigma_{12}|=|\Sigma'_{12}|  \\
\label{eq:sigma12alphatoy}
\frac{\Sigma_{21}}{\Sigma'_{21}}=\frac{(\alpha'_1-y_1)}{(\alpha_1-y_1)} &\Longrightarrow |\alpha_1-y_1|=|\alpha'_1-y_1| \qquad \forall y_1 \in \mathbb{R}
\end{align}

Equation \eqref{eq:sigma12alphatoy} implies that $\alpha_1=\alpha'_1$, since for $y_1=\alpha_1'$, one has $\alpha_1-\alpha'_1=0$.


Using \eqref{eq:id_lawequality_twovar_toy}, one has
\begin{multline}\label{eq:Y1Y2toy}
\mathbb{P}(\Omega_{1}=1|Y_{1} =y_1,Y_{2} =y_2;\phi_1)f_{(Y_{1},Y_{2})}(y_1,y_2;\alpha_1,\alpha_2,\Sigma_{(12)}) \\
=\mathbb{P}(\Omega'_{1}=1|Y'_{1} =y_1,Y'_{2}=y_2;\phi'_1)f_{(Y'_{1},Y'_{2})}(y_1,y_2;\alpha'_1,\alpha'_2,\Sigma'_{(12)}) \qquad \forall (y_1,y_2) \in \mathbb{R}^2
\end{multline}
Using \eqref{eq:tiptoy}, 
\begin{align*}
\frac{\exp{\left(-\frac{1}{2}\begin{pmatrix}y_1-\alpha_1 & y_2-\alpha_2\end{pmatrix}\Sigma^{-1}_{(12)}\begin{pmatrix}y_1-\alpha_1 \\ y_2-\alpha_2\end{pmatrix}\right)}}{\exp{\left(-\frac{1}{2}\begin{pmatrix}y_1-\alpha_1 & y_2-\alpha_2\end{pmatrix}(\Sigma'_{(12)})^{-1}\begin{pmatrix}y_1-\alpha_1 \\ y_2-\alpha_2\end{pmatrix}\right)}}\frac{\mathbb{P}(\Omega_{1}=1|Y_{1}=y_1;\phi_1)}{\mathbb{P}(\Omega'_{1}=1|Y'_{1}=y_1;\phi_1')}=\frac{\sqrt{\mathrm{det}(\Sigma_{(12)})}}{\sqrt{\mathrm{det}(\Sigma'_{(12)}})},
\end{align*}
where $\mathrm{det}(\Sigma_{(12)})$ denotes the determinant of the matrix $\Sigma_{(12)}$.

With \eqref{eq:sigma12toy}, one has $\Sigma_{11}\Sigma_{22}-\Sigma_{12}^2=\Sigma_{11}\Sigma_{22}-(\Sigma'_{12})^2$ and $\frac{\sqrt{\mathrm{det}(\Sigma_{(12)})}}{\sqrt{\mathrm{det}(\Sigma'_{(12)}})}=1$. 

It leads to $\forall (y_1,y_2) \in \mathbb{R}^2$,
$$K \cdot \frac{\mathbb{P}(\Omega_{1}=1|Y_{1}=y_1;\phi_1)}{\mathbb{P}(\Omega'_{1}=1|Y'_{1}=y_1;\phi_1')}=1,$$
with
$$K:=\frac{\exp{\left(-\frac{1}{2\mathrm{det}(\Sigma_{(12)})}\left((y_1-\alpha_1)^2\Sigma_{11}+(y_2-\alpha_2)^2\Sigma_{22}-2(y_1-\alpha_1)(y_2-\alpha_2)\Sigma_{12}\right)\right)}}{\exp{\left(-\frac{1}{2\mathrm{det}(\Sigma_{(12)})}\left((y_1-\alpha_1)^2\Sigma_{11}+(y_2-\alpha_2)^2\Sigma_{22}-2(y_1-\alpha'_1)(y_2-\alpha_2)\Sigma'_{12}\right)\right)}}.$$
The quantity $K$ is equal to one, because 
$$(y_2-\alpha_2)\big((y_1-\alpha_1)\Sigma_{12}-(y_1-\alpha'_1)\Sigma'_{12}\big)=0$$ 

using \eqref{eq:sigma12alphatoy}. Thus, 

$$
\frac{\mathbb{P}(\Omega_{1}=1|Y_{1}=y_1;\phi_1)}{\mathbb{P}(\Omega'_{1}=1|Y'_{1}=y_1;\phi_1')}=1 
\quad \Longleftrightarrow \quad  F_1(\phi^0_1+\phi^1_1y_{1})=F_1((\phi')^0_1+(\phi')^1_1y_{1}) \qquad \forall y_1 \in \mathbb{R}
$$

As $F_1$ is strictly monotone, it is an injective function. Thus,

$$\phi^0_1+\phi^1_1y_{1}=(\phi')^0_1+(\phi')^1_1y_{1} \qquad \forall y_1 \in \mathbb{R}
\qquad \Longleftrightarrow \qquad  (\phi^0_1-(\phi')^0_1)+((\phi')^1_1-\phi^1_1)y_1=0 \qquad \forall y_1 \in \mathbb{R}$$

It implies $\phi_1=\phi'_1$. 

\subparagraph{Identifiability of the Covariances of the MNAR variable}

Equation \eqref{eq:Y1Y2toy} thus leads to
$$f_{(Y_{1},Y_{2})}(y_1,y_2;\alpha_1,\alpha_2,\Sigma_{(12)})=f_{(Y'_{1},Y'_{2})}(y_1,y_2;\alpha'_1,\alpha'_2,\Sigma'_{(12)}) \qquad \forall (y_1,y_2) \in \mathbb{R}^2$$
One can conclude that $\Sigma_{12}=\Sigma'_{12}$. 
The same reasoning may be done for the covariance between $Y_1$ and $Y_3$. 

\subparagraph{Identifiability of the loading matrix}

One wants to prove $B=B'$ up to row permutation. 
One has
\begin{align}
    \nonumber
    \Sigma=\Sigma'
    &\Leftrightarrow \Sigma-\sigma^2I_{p\times p}=\Sigma'-\sigma^2I_{p\times p} \\
    \label{eq:btb_equal}
    &\Leftrightarrow B^TB = (B')^TB'
\end{align}

As $B^TB$ is a positive symetric matrix of rank $2$, one has the following singular value decomposition, 
$$B^TB = (B')^TB'=UDU^T,$$
where $U= (u_1 | u_2 | u_3) \in \mathbb{R}^{3\times 3}$ the orthogonal matrix of singular vector and 
$$D= \begin{pmatrix}
			\sqrt{{d}_1} & 0 &0  \\
			0 &  \sqrt{{d}_2} & 0 \\
			0 & 0 & 0 \\
		\end{pmatrix}  \in \mathbb{R}^{3\times3}$$ 
with $d_1 \geq d_2 \geq 0$.
One can choose
$$B=\begin{pmatrix}
			& \sqrt{d_1}{u}_1^T & \\
			\hline
			& \sqrt{d_2} {u}_2^T & \end{pmatrix}$$
noting that a row permutation of B would not change the product $B^TB$. Therefore, $B=B'$ up to a row permutation.

\end{proof}

\subsection{Proof of Proposition \ref{prop:identifiability} in the general case}

We present the proof of Proposition \ref{prop:identifiability} in the general case where $d$ variables are self-masked MNAR and $p-d$ variables are MCAR.

\begin{proof}
Assume that $(Y,\Omega)$ and $(Y',\Omega')$ have distributions respectively parameterized by $(\alpha,\Sigma, \phi)$ and $(\alpha',\Sigma', \phi')$.
Assume that $Y$ and $Y'$ have the same following observed distributions
\begin{equation}
\label{eq:id_lawequality}
\mathcal{L}(Y_j,\Omega_j=1;\alpha_j,\Sigma_{jj},\phi_j)=\mathcal{L}(Y'_j,\Omega'_j=1;\alpha'_j,\Sigma'_{jj},\phi'_j) \qquad \forall j \in \{1, \dots,p\},
\end{equation}
\begin{multline}
\label{eq:id_lawequality_twovar}
\mathcal{L}(Y_j,Y_k, \Omega_j=1, \Omega_k=1;\alpha_j,\alpha_k,\Sigma_{(jk)},\phi_j,\phi_k)\\ 
=\mathcal{L}(Y'_j,Y'_k,\Omega'_j=1, \Omega'_k=1;\alpha'_j,\alpha'_k,\Sigma'_{(jk)},\phi'_j,\phi'_k) \qquad \forall j\neq k \in \{1, \dots,p\},
\end{multline}
where $\Sigma_{(jk)}$ denotes the covariance matrix $\begin{pmatrix} \Sigma_{jj} & \Sigma_{jk} \\
\Sigma_{jk} & \Sigma_{kk}\end{pmatrix}$.

In order to show that parameters identifiability holds, we need to show that \eqref{eq:id_lawequality} and \eqref{eq:id_lawequality_twovar} implies that $\alpha=\alpha'$, $\Sigma=\Sigma'$ and $\phi=\phi'$. Then, under a known noise level $\sigma^2$, we will prove that $B$ and $B'$ are equal up to row permutations. 

In what follows, $f_{Y_{.j}}$ or $f_{(Y_{.j},Y_{.k})}$ respectively denote the density function of $Y_{.j}$, and of $(Y_{.j},Y_{.k})$.

In the following, we will use the following tip, for any $l \in \{1,\dots,p\}$ and $\mathcal{K} \subset \{1,\dots,p\}\setminus\{l\}$ such that $0\leq|\mathcal{K}|\leq p-1$,
\begin{align*}
\mathbb{P}(\Omega_{l}=1|Y_{l} =y_l,Y_{\mathcal{K}} =y_{\mathcal{K}};\phi_l)&=\mathbb{E}[\mathbb{E}[\mathbf{1}_{\Omega_{l}=1}|Y;\phi_l]|Y_{l} =y_l,Y_{\mathcal{K}} =y_{\mathcal{K}}] \\
&=\mathbb{E}[\mathbb{P}(\Omega_{l}=1|Y =y;\phi_l)|Y_{l} =y_l,Y_{\mathcal{K}} =y_{\mathcal{K}}]
\end{align*}
Thus, 
\begin{multline*}
\mathbb{P}(\Omega_{l}=1|Y_{l} =y_l,Y_{\mathcal{K}} =y_{\mathcal{K}};\phi_l)\\
=\left\{\begin{array}{lr} 
 \mathbb{E}[\mathbb{P}(\Omega_{l}=1|Y_l =y_l;\phi_l)|Y_{l} =y_l,Y_{\mathcal{K}} =y_{\mathcal{K}}]  &\qquad \textrm{if $Y_l$ is self-masked MNAR} \\
  \mathbb{E}[\mathbb{P}(\Omega_{l}=1;\phi_l)|Y_{l} =y_l,Y_{\mathcal{K}} =y_{\mathcal{K}}]  & \qquad \textrm{if $Y_l$ is MCAR} \\
  \end{array}\right.
 \end{multline*}
by measurability if $Y_{l}$ is self-masked MNAR and by independence if $Y_{l}$ is MCAR. Thus, using the mechanisms in \eqref{eq:selfmaskedMNARmechanism} and \eqref{eq:MCARmechanism},
\begin{empheq}[left={\mathbb{P}(\Omega_{l}=1|Y_{l} =y_l,Y_{\mathcal{K}} =y_{\mathcal{K}};\phi_l)=\empheqlbrace}]{align}
\label{eq:tipMNAR}
    &\mathbb{P}(\Omega_{l}=1|Y_l =y_l;\phi_l) & \textrm{if $Y_l$ is self-masked MNAR} \\
\label{eq:tipMCAR}
    & \mathbb{P}(\Omega_{l}=1;\phi_l) & \textrm{if $Y_l$ is MCAR}
  \end{empheq}

\paragraph{Identifiability of the parameters for the not-MNAR variables $(Y_j)_{j \in    \widebar{\mathcal{M}}}$.}

\subparagraph{Mechanism parameter, Mean and Variance of $Y_j, j \in \widebar{\mathcal{M}}$.}
Equation \eqref{eq:id_lawequality} trivially gives that
$$P(\Omega_j=1)=P(\Omega'_j=1).$$
Using \eqref{eq:tipMCAR}, $P(\Omega_j=1)=\mathbb{P}(\Omega_{j}=1|Y_j =y_j;\phi_j)=F_j(\phi_j)$. As $F_j$ is strictly monotone, it implies that
$$F_j(\phi_j)=F_j(\phi'_j) 
\Longleftrightarrow 
\phi_j=\phi'_j.$$
Equation \eqref{eq:id_lawequality} also leads to
$$\mathbb{P}(\Omega_{j}=1|Y_j =y_j;\phi_j)f_{Y_{j}}(y_j;\alpha_j,\Sigma_{jj})=\mathbb{P}(\Omega'_{j}=1|Y'_j =y_j;\phi'_j)f_{Y'_{j}}(y_j;\alpha'_j,\Sigma'_{jj}) \qquad \forall y_j \in \mathbb{R}.$$
As $\phi_j=\phi'_j$, one obtains
$$f_{Y_{j}}(y_j;\alpha_j,\Sigma_{jj})=f_{Y'_{j}}(y_j;\alpha'_j,\Sigma'_{jj}) \qquad \forall y_j \in \mathbb{R}$$
which directly implies that $\alpha_j=\alpha'_j$ and $\Sigma_{jj}=\Sigma'_{jj}$, since $Y_{j}$ and $Y'_{j}$ are Gaussian variables.

\subparagraph{Covariance between two not MNAR variables $Y_j$ and $Y_k, \: j\neq k \in \widebar{\mathcal{M}}$.}
Equation \eqref{eq:id_lawequality_twovar} gives that for all $(y_j,y_k) \in \mathbb{R}^2$
\begin{multline}\label{eq:MCARid}
\mathbb{P}(\Omega_{j}=1,\Omega_k=1|Y_j=y_j,Y_k=y_k;\phi_{j},\phi_k)f_{(Y_{j},Y_{k})}(y_j,y_k;\alpha_{j},\alpha_k,\Sigma_{(j,k)}) \\
=\mathbb{P}(\Omega'_{j}=1,\Omega'_k=1|Y'_j=y_j,Y'_k=y_k;\phi'_{j},\phi'_k)f_{(Y'_{j},Y'_{k})}(y_j,y_k;\alpha'_{j},\alpha'_k,\Sigma'_{(j,k)}),
\end{multline}
and one has as well that
$$\mathbb{P}(\Omega_{j}=1,\Omega_k=1|Y_j=y_j,Y_k=y_k;\phi_{j},\phi_k)=\mathbb{P}(\Omega_{j}=1|Y_j=y_j;\phi_j)\mathbb{P}(\Omega_{k}=1|Y_k=y_k;\phi_k),$$
because $\Omega_{j}\independent\Omega_{k}|Y$. Likewise, 
$$\mathbb{P}(\Omega'_{j}=1,\Omega'_k=1|Y'_j=y_j,Y'_k=y_k;\phi'_{j},\phi'_k)=\mathbb{P}(\Omega'_{j}=1|Y'_j=y_j;\phi'_j)\mathbb{P}(\Omega'_{k}=1|Y'_k=y_k;\phi'_k).$$ 
Given that $\phi_j=\phi'_j$ and $\phi_k=\phi'_k$, one obtains
$$\mathbb{P}(\Omega_{j}=1,\Omega_k=1|Y_j=y_j,Y_k=y_k;\phi_{j},\phi_k)=\mathbb{P}(\Omega'_{j}=1,\Omega'_k=1|Y'_j=y_j,Y'_k=y_k;\phi_{j},\phi_k).$$
Thus, Equation \eqref{eq:MCARid} leads to, for all $(y_j,y_k) \in \mathbb{R}^2$,
$$f_{(Y_{j},Y_{k})}(y_j,y_k;\alpha_{j},\alpha_k,\Sigma_{(j,k)})=f_{(Y'_{j},Y'_{k})}(y_j,y_k;\alpha'_{j},\alpha'_k,\Sigma'_{(j,k)}),$$
and $\Sigma_{jk}=\Sigma'_{jk}$. 

\paragraph{Identifiability of the parameters for the MNAR variables} 

\subparagraph{Variance of $Y_m, m \in \mathcal{M}$}
Equation \eqref{eq:id_lawequality} gives that 
$$f_{(Y_{m},\Omega_{m}=1)}(y_m;\alpha_m,\Sigma_{mm},\phi_m)=f_{(Y'_{m},\Omega'_{m}=1)}(y_m;\alpha'_m,\Sigma'_{mm},\phi'_m) \qquad \forall y_m \in \mathbb{R}.$$
Given the missing mechanism in \eqref{eq:selfmaskedMNARmechanism} and that $Y_{.m} \sim \mathcal{N}(\alpha_m,\Sigma_{mm})$,  \cite[Theorem 1 a)]{miao2016identifiability} ensures that $\Sigma_{mm}=\Sigma'_{mm}$.

\subparagraph{Mean and mechanism parameter of $Y_m, m \in \mathcal{M}$}

Let $j$ be the index of a not MNAR variable.
One has
\begin{multline}\label{eq:id_YjgivenYm}
\mathcal{L}(Y_{j},\Omega_{j}=1|Y_{m},\Omega_{m}=1;\alpha_j,\alpha_m,\Sigma_{(jm)},\phi_j,\phi_m) \\
=\mathcal{L}(Y'_{j},\Omega'_{j}=1|Y'_{m},\Omega'_{m}=1;\alpha'_j,\alpha'_m,\Sigma'_{(jm)},\phi'_j,\phi'_m)
\end{multline}
using \eqref{eq:id_lawequality} and \eqref{eq:id_lawequality_twovar} and noting that
\begin{multline*}
f_{(Y_{j},\Omega_{j}=1)|Y_{m}=y_m,\Omega_{m}=1}(y_j;\alpha_j,\alpha_m,\Sigma_{(jm)},\phi_j,\phi_m) \\
=\frac{f_{(Y_{j},\Omega_{j}=1,Y_{m},\Omega_{m}=1)}(y_j,y_m;\alpha_j,\alpha_m,\Sigma_{(jm)},\phi_j,\phi_m)}{f_{(Y_{m},\Omega_{m}=1)}(y_m;\alpha_m,\Sigma_{mm},\phi_m)} \qquad \forall (y_j,y_m) \in \mathbb{R}^2.
\end{multline*}
Equation \eqref{eq:id_YjgivenYm} implies that $\forall (y_j,y_m) \in \mathbb{R}^2$,
\begin{multline}\label{eq:mechanismYm}
\mathbb{P}(\Omega_j=1|Y_j=y_j,Y_m=y_m,\Omega_m=1;\phi_j)\frac{\mathbb{P}(\Omega_m=1|Y_j=y_j,Y_m=y_m;\phi_m)f_{Y_j|Y_m=y_m}(y_j;\alpha_j,\alpha_m,\Sigma_{(jm)})}{\mathbb{P}(\Omega_m=1|Y_m=y_m;\phi_m)}\\
=\mathbb{P}(\Omega'_j=1|Y'_j=y_j,Y'_m=y_m,\Omega'_m=1;\phi'_j)\frac{\mathbb{P}(\Omega'_m=1|Y'_j=y_j,Y'_m=y_m;\phi'_m)f_{Y'_j|Y'_m=y_m}(y_j;\alpha'_j,\alpha'_m,\Sigma'_{(jm)})}{\mathbb{P}(\Omega'_m=1|Y'_m=y_m;\phi'_m)}
\end{multline}
\normalsize
One can note that 
$$
\mathbb{P}(\Omega_j=1|Y_j=y_j,Y_m=y_m,\Omega_m=1;\phi_j)=\mathbb{P}(\Omega_j=1|Y_j=y_j;\phi_j).
$$ 
Indeed,
\begin{align*}
\mathbb{P}(\Omega_j=1|Y_j=y_j,Y_m=y_m,\Omega_m=1;\phi_j)&=\frac{\mathbb{P}(\Omega_j=1\cap \Omega_m=1|Y_j=y_j,Y_m=y_m;\phi_j,\phi_m)}{\mathbb{P}(\Omega_m=1|Y_j=y_j,Y_m=y_m;\phi_m)} \\
&=\frac{\mathbb{P}(\Omega_j=1|Y_j=y_j;\phi_j)\mathbb{P}(\Omega_m=1|Y_m=y_m;\phi_m)}{\mathbb{P}(\Omega_m=1|Y_j=y_j,Y_m=y_m;\phi_m)} \\
&=\mathbb{P}(\Omega_j=1|Y_j=y_j;\phi_j),
\end{align*}

using \eqref{eq:hypID} in the second step. Likewise,  
$$\mathbb{P}(\Omega'_j=1|Y'_j=y_j,Y'_m=y_m,\Omega'_m=1;\phi'_j)=\mathbb{P}(\Omega'_j=1|Y'_j=y_j;\phi'_j).$$
Given that $\phi_j=\phi'_j$, 
\begin{equation*}
\mathbb{P}(\Omega_j=1|Y_j=y_j,Y_m=y_m,\Omega_m=1;\phi_j)=\mathbb{P}(\Omega'_j=1|Y'_j=y_j,Y'_m=y_m,\Omega'_m=1;\phi'_j)
\end{equation*}

Thus, Equation \eqref{eq:mechanismYm} leads to 
\begin{multline*}
\frac{\mathbb{P}(\Omega_m=1|Y_j=y_j,Y_m=y_m;\phi_m)f_{Y_j|Y_m=y_m}(y_j;\alpha_j,\alpha_m,\Sigma_{(jm)})}{\mathbb{P}(\Omega_m=1|Y_m=y_m;\phi_m)}\\
=\frac{\mathbb{P}(\Omega'_m=1|Y'_j=y_j,Y'_m=y_m;\phi'_m)f_{Y'_j|Y'_m=y_m}(y_j;\alpha'_j,\alpha'_m,\Sigma'_{(jm)})}{\mathbb{P}(\Omega'_m=1|Y'_m=y_m;\phi'_m)} \qquad \forall (y_j,y_m) \in \mathbb{R}^2.
\end{multline*}
As $\mathbb{P}(\Omega_m=1|Y_j=y_j,Y_m=y_m;\phi_m)=\mathbb{P}(\Omega_m=1|Y_j=y_j;\phi_m)$ by using \eqref{eq:tipMNAR}, one obtains 
$$f_{Y_j|Y_m=y_m}(y_j;\alpha_j,\alpha_m,\Sigma_{(jm)})
=f_{Y'_j|Y'_m=y_m}(y_j;\alpha'_j,\alpha'_m,\Sigma'_{(jm)}) \qquad \forall (y_j,y_m) \in \mathbb{R}^2, $$
which leads to the equality of the conditional expectation and variance, as follows:
\begin{align*}
\alpha_{j}+\Sigma_{mj}\Sigma_{mm}^{-1}(\alpha_m-y_m)&=\alpha'_{j}+\Sigma'_{mj}(\Sigma'_{mm})^{-1}(\alpha'_m-y_m) \qquad \forall (y_j,y_m) \in \mathbb{R}^2 \\
\Sigma_{jj}-\Sigma_{mj}^2\Sigma_{mm}^{-1}&=\Sigma'_{jj}-(\Sigma'_{mj})^2(\Sigma'_{mm})^{-1}
\end{align*}

As $\alpha_j=\alpha'_j$ and $\Sigma_{mm}=\Sigma'_{mm}$, 

\begin{align}
\label{eq:sigma12}
\Sigma_{mj}^2=(\Sigma'_{mj})^2 &\Longrightarrow |\Sigma_{mj}|=|\Sigma'_{mj}|  \\
\label{eq:sigma12alpha}
\frac{\Sigma_{mj}}{\Sigma'_{mj}}=\frac{(\alpha'_m-y_m)}{(\alpha_m-y_m)} &\Longrightarrow |\alpha_m-y_m|=|\alpha'_m-y_m| \qquad \forall y_m \in \mathbb{R}
\end{align}

Equation \eqref{eq:sigma12alpha} implies that $\alpha_m=\alpha'_m$, since for $y_m=\alpha_m'$, one has $\alpha_m-\alpha'_m=0$.


In addition, using \eqref{eq:id_lawequality}, one has for all  $(y_j,y_m) \in \mathbb{R}^2$,
\begin{multline}\label{eq:mechanismYmbis}
\mathbb{P}(\Omega_j=1,\Omega_m=1|Y_j=y_j,Y_m=y_m;\phi_j,\phi_m)f_{(Y_{j},Y_{m})}(y_j,y_m;\alpha_j,\alpha_m,\Sigma_{(jm)}) \\
= \mathbb{P}(\Omega'_j=1,\Omega'_m=1|Y'_j=y_j,Y'_m=y_m;\phi'_j,\phi'_m)f_{(Y'_{j},Y'_{m})}(y_j,y_m;\alpha'_j,\alpha'_m,\Sigma'_{(jm)}) 
\end{multline}
One can note that
\begin{multline*}
\mathbb{P}(\Omega_j=1,\Omega_m=1|Y_j=y_j,Y_m=y_m;\phi_j,\phi_m)\\
=\mathbb{P}(\Omega_j=1|Y_j=y_j;\phi_j)\mathbb{P}(\Omega_m=1|Y_m=y_m;\phi_m),
\end{multline*}
using \eqref{eq:hypID}. The same equation holds for $(Y'_j,Y'_m,\Omega'_j,\Omega'_m)$ with the parameters $(\phi'_j,\phi'_m)$. Using $\phi_j=\phi'_j$, Equation \eqref{eq:mechanismYmbis} leads to 
\begin{multline}\label{eq:mechanismYmter}
\mathbb{P}(\Omega_m=1|Y_m=y_m;\phi_m)f_{(Y_{j},Y_{m})}(y_j,y_m;\alpha_j,\alpha_m,\Sigma_{(jm)})= \\
\mathbb{P}(\Omega'_m=1|Y'_m=y_m;\phi'_m)f_{(Y'_{j},Y'_{m})}(y_j,y_m;\alpha'_j,\alpha'_m,\Sigma'_{(jm)}) \qquad \forall (y_j,y_m) \in \mathbb{R}^2.
\end{multline}
It implies that, $\forall (y_j,y_m) \in \mathbb{R}^2$,
\footnotesize
\begin{align*}
\frac{\exp{\left(-\frac{1}{2}\begin{pmatrix}y_j-\alpha_j & y_m-\alpha_m\end{pmatrix}\Sigma^{-1}_{(jm)}\begin{pmatrix}y_j-\alpha_j \\ y_m-\alpha_m\end{pmatrix}\right)}}{\exp{\left(-\frac{1}{2}\begin{pmatrix}y_j-\alpha'_j & y_m-\alpha'_m\end{pmatrix}(\Sigma'_{(jm)})^{-1}\begin{pmatrix}y_j-\alpha'_j \\ y_m-\alpha'_m\end{pmatrix}\right)}}
\frac{\mathbb{P}(\Omega_{m}=1|Y_{m}=y_m;\phi_m)}{\mathbb{P}(\Omega'_{m}=1|Y'_{m}=y_m;\phi'_m)}
=\frac{\sqrt{\mathrm{det}(\Sigma_{(jm)})}}{\sqrt{\mathrm{det}(\Sigma'_{(jm)}})},
\end{align*}
\normalsize
where $\mathrm{det}(\Sigma_{(jm)})$ denotes the determinant of the covariance matrix $\Sigma_{(jm)}$. 

With $\Sigma_{jj}=\Sigma'_{jj}$, $\Sigma_{mm}=\Sigma'_{mm}$ and Equation \eqref{eq:sigma12}, one has $$\Sigma_{jj}\Sigma_{mm}-\Sigma_{mj}^2=\Sigma_{jj}\Sigma_{mm}-(\Sigma'_{mj})^2 
\qquad \Longrightarrow \qquad 
\frac{\sqrt{\mathrm{det}(\Sigma_{(jm)})}}{\sqrt{\mathrm{det}(\Sigma'_{(jm)}})}=1.
$$

Besides, using $\alpha_j=\alpha'_j$, $\Sigma_{jj}=\Sigma'_{jj}$ and  $\Sigma_{mm}=\Sigma'_{mm}$, one obtains that for all $ (y_j,y_m) \in \mathbb{R}^2$,
$$K \cdot \frac{\mathbb{P}(\Omega_{m}=1|Y_{m}=y_m;\phi_m)}{\mathbb{P}(\Omega'_{m}=1|Y'_{m}=y_m;\phi'_m)}=1,$$
with
$$K:=\frac{\exp{\left(-\frac{1}{2\mathrm{det}(\Sigma_{(jm)})}\left((y_j-\alpha_j)^2\Sigma_{jj}+(y_m-\alpha_m)^2\Sigma_{mm}-2(y_j-\alpha_j)(y_m-\alpha_m)\Sigma_{mj}\right)\right)}}{\exp{\left(-\frac{1}{2\mathrm{det}(\Sigma_{(jm)})}\left((y_j-\alpha_j)^2\Sigma_{jj}+(y_m-\alpha_m)^2\Sigma_{mm}-2(y_j-\alpha_j)(y_m-\alpha'_m)\Sigma'_{mj}\right)\right)}}.$$
The quantity $K$ is equal to one, because $$(y_j-\alpha_j)((y_m-\alpha_m)\Sigma_{mj}-(y_m-\alpha'_m)\Sigma'_{mj})=0$$ using \eqref{eq:sigma12alpha}. Thus, for all $y_m \in \mathbb{R}$,
$$
\frac{\mathbb{P}(\Omega_{m}=1|Y_{m}=y_m;\phi_m)}{\mathbb{P}(\Omega'_{m}=1|Y'_{m}=y_m;\phi'_m)}=1 
\quad \Longleftrightarrow \quad  F_m(\phi^0_m+\phi^1_my_{m})=F_m((\phi')^0_m+(\phi')^1_my_{m}).
$$

As F is strictly monotone, it is an injective function. Thus,
$$\phi^0_m+\phi^1_my_{m}=(\phi')^0_m+(\phi')^1_my_{m}\Leftrightarrow ((\phi')^0_m-\phi^0_m)+((\phi')^1_m-\phi^1_m)y_m=0 \qquad \forall y_1 \in \mathbb{R}$$

It implies that $\phi_m=\phi'_m$.

\subparagraph{Covariance between $Y_j$ and $Y_m$ with $j \in \widebar{\mathcal{M}}, m \in \mathcal{M}$}

Using \eqref{eq:mechanismYmter} and $\phi_m=\phi'_m$, one has 
\begin{align*}
f_{(Y_{j},Y_{m})}(y_j,y_m;\alpha_j,\alpha_m,\Sigma_{(jm)})=f_{(Y'_{j},Y'_{m})}(y_j,y_m;\alpha'_j,\alpha'_m,\Sigma'_{(jm)}) \qquad \forall (y_j,y_m) \in \mathbb{R}^2
\end{align*}
One can conclude that $\Sigma_{mj}=\Sigma'_{mj}$. 

\subparagraph{Covariance between $Y_\ell$ and $Y_m$ with $\ell\neq m \in \mathcal{M}$}

Using \eqref{eq:id_lawequality}, one has for all $(y_\ell,y_m) \in \mathbb{R}^2$,
\begin{multline}\label{eq:covYmbis}
\mathbb{P}(\Omega_\ell=1,\Omega_m=1|Y_j=y_j,Y_m=y_m;\phi_\ell,\phi_m)f_{(Y_{\ell},Y_{m})}(y_\ell,y_m;\alpha_\ell,\alpha_m,\Sigma_{(\ell m)}) \\
= \mathbb{P}(\Omega'_\ell=1,\Omega'_m=1|Y'_\ell=y_\ell,Y'_m=y_m;\phi'_\ell,\phi'_m)f_{(Y'_{\ell},Y'_{m})}(y_\ell,y_m;\alpha'_\ell,\alpha'_m,\Sigma'_{(\ell m)}) 
\end{multline}
One can note that
\begin{multline*}
\mathbb{P}(\Omega_\ell=1,\Omega_m=1|Y_\ell=y_\ell,Y_m=y_m;\phi_\ell,\phi_m) \\ =\mathbb{P}(\Omega_\ell=1|Y_\ell=y_\ell;\phi_\ell)\mathbb{P}(\Omega_m=1|Y_m=y_m;\phi_m),
\end{multline*}

using \eqref{eq:hypID}. The same equation holds for $(Y'_\ell,Y'_m,\Omega'_\ell,\Omega'_m)$ with the parameters $(\phi'_\ell,\phi'_m)$. Yet $\phi_\ell=\phi'_\ell$ and $\phi_m=\phi'_m$, which gives, for all $(y_j,y_m) \in \mathbb{R}^2$,
$$\mathbb{P}(\Omega_\ell=1,\Omega_m=1|Y_\ell=y_\ell,Y_m=y_m;\phi_\ell,\phi_m)=\mathbb{P}(\Omega'_\ell=1,\Omega'_m=1|Y'_\ell=y_\ell,Y'_m=y_m;\phi_\ell,\phi'_m).
$$

Equation \eqref{eq:covYmbis} leads to
$$f_{(Y_{\ell},Y_{m})}(y_\ell,y_m;\alpha_\ell,\alpha_m,\Sigma_{(\ell m)})=f_{(Y'_{\ell},Y'_{m})}(y_\ell,y_m;\alpha'_\ell,\alpha'_m,\Sigma'_{(\ell m)}) \qquad \forall (y_\ell,y_m) \in \mathbb{R}^2,$$
which implies that $\Sigma_{\ell m}=\Sigma'_{\ell m}$.

\paragraph{Identifiability of the loading matrix}

One wants to prove that $B=B'$ up to a row permutation. 
One has
\begin{align}
    \nonumber
    \Sigma=\Sigma'
    &\Longleftrightarrow \Sigma-\sigma^2I_{p\times p}=\Sigma'-\sigma^2I_{p\times p} \\
    \label{eq:btb_equal}
    &\Longleftrightarrow B^TB = (B')^TB'
\end{align}

As $B^TB$ is a positive symetric matrix of rank $r$, its singular value decomposition reads 
$$B^TB = (B')^TB'=UDU^T,$$
where $U= (u_1 | \hdots | u_p) \in \mathbb{R}^{p\times p}$ is an orthogonal matrix containing the singular vectors and 
$$D= \begin{pmatrix}
			\sqrt{{d}_1} & & & & &  \\
			&  \ddots & & 0 & & \\
			& & \sqrt{{d}_r} & & & \\
			&  0 & & 0 & &  \\
			&  & & & \ddots &  \\
			&  & & & &  0 \\
		\end{pmatrix}  \in \mathbb{R}^{p\times p}$$ 
with $d_1\geq \dots \geq d_r \geq 0$.
One can choose
$$B=\begin{pmatrix}
			& \sqrt{d_1}{u}_1^T & \\
			\hline
			& \vdots & \\
			\hline
			& \sqrt{d_r} {u}_r^T & \end{pmatrix}$$
A row permutation of B does not change the product $B^TB$. Therefore, $B=B'$ up to a row permutation.

\end{proof}

\section{Proof for Section \ref{sec:generalcase}}
\label{sec:proofresults}

\subsection{Proof of Lemma \ref{lem:PCAlinear}}
\label{sec:proofPCAlinar}

\newtheorem*{lem:PCAlinear}{Lemma \ref{lem:PCAlinear}}
\begin{lem:PCAlinear}
	Under the PPCA model \eqref{eq:model} and Assumption \ref{hyp1}, choose $j\in \mathcal{J}$. Denote $B^{-1} \in \mathbb{R}^{r\times r}$ the inverse of $\begin{pmatrix} B_{.m} & (B_{.j'})_{j'\in \mathcal{J}_{-j}} \end{pmatrix}$. 
	One has
		\begin{equation*}
		Y_{.j}=\mathcal{B}_{j\rightarrow m,\mathcal{J}_{-j}[0]}+\sum_{j'\in \mathcal{J}_{-j}} \mathcal{B}_{j\rightarrow m,\mathcal{J}_{-j}[j']} Y_{.j'} + \mathcal{B}_{j\rightarrow m,\mathcal{J}_{-j}[m]} Y_{.m} + \zeta
    \end{equation*}
    	with:
	\begin{align*}
		\mathcal{B}_{j\rightarrow m,\mathcal{J}_{-j}[j']} &:=\sum_{k \in \{m\}\cup\mathcal{J}_{-j}}B^{-1}_{kj'}B_{jk}, \forall j' \in \mathcal{J}_{-j} \\
		\mathcal{B}_{j\rightarrow m,\mathcal{J}_{-j}[m]}&:=\sum_{k \in \{m\}\cup\mathcal{J}_{-j}}B^{-1}_{km}B_{jk},  \\
		\mathcal{B}_{j\rightarrow m,\mathcal{J}_{-j}[0]}&:=\mathbf{1}\alpha_{j} - \sum_{j' \in \mathcal{J}_{-j}}	\mathcal{B}_{j\rightarrow m,\mathcal{J}_{-j}[j']}\mathbf{1}\alpha_{j'} - \mathcal{B}_{j\rightarrow m,\mathcal{J}_{-j}[m]}\mathbf{1}\alpha_m \\ 
		\zeta&=-\sum_{j'\in\mathcal{J}_{-j}} \mathcal{B}_{j\rightarrow m,\mathcal{J}_{-j}[j']} \epsilon_{.j'} - \mathcal{B}_{j\rightarrow m,\mathcal{J}_{-j}[m]} \epsilon_{.m}+ \epsilon_{.j}
	\end{align*}
	
\end{lem:PCAlinear}

\begin{proof}
	
	Starting from the PPCA model written in \eqref{eq:model} and recalled here
	$$Y=\mathbf{1}\alpha + WB + \epsilon$$
	and the matrix $B \in \mathbb{R}^{r\times p}$ being of full rank $r$, solving this linear system is the same as solving the following reduced system
	
	\begin{equation*}
		\begin{pmatrix}
			Y_{.m} &
			(Y_{.j'})_{j' \in \mathcal{J}_{-j}}
		\end{pmatrix}= \mathbf{1}\alpha_{|r} + \begin{pmatrix}
			W_{.1} &
			\dots &
			W_{.r}
		\end{pmatrix} B_{|r}  + \epsilon_{|r}, 
	\end{equation*}
	where $B_{|r} \in \mathbb{R}^{r \times r}$ denotes the reduced matrix $\begin{pmatrix}
			B_{.m} &
			(B_{.j'})_{j' \in \mathcal{J}_{-j}}
		\end{pmatrix}$ of $B$. Similarly, $\alpha_{|r} \in \mathbb{R}^{r}$ and $\epsilon_{|r} \in \mathbb{R}^{n \times r}$ denote the reduced matrices of $\alpha$ and $\epsilon$. 
	With a slight abuse of notation, $B^{-1}$ denotes the inverse of the reduced matrix $\begin{pmatrix} B_{.m} & (B_{.j'})_{j' \in \mathcal{J}_{-j}}
	\end{pmatrix}$ which exists using \ref{hyp1}. 
	
	Then, one can derive that
	\begin{equation*}
		\begin{pmatrix}
			W_{.1} &
			\dots &
			W_{.r}
		\end{pmatrix} =  \left(	\begin{pmatrix}
			Y_{.m} &
			(Y_{.j'})_{j' \in \mathcal{J}_{-j}}
		\end{pmatrix}  - \mathbf{1}\alpha_{|r} - \epsilon_{|r} \right) B^{-1}.
	\end{equation*}
	
	The expression of $Y_{.j}$ as a function of the latent variables is 
	\begin{align*}
		Y_{.j}&=\mathbf{1}\alpha _{j}+ \begin{pmatrix}
			W_{.1} &
			\dots &
			W_{.r}
		\end{pmatrix}B_{j.}+\epsilon_{.j}  \\
		&= \mathbf{1}\alpha_{j} +  \left( \begin{pmatrix}
			Y_{.m}  &
			(Y_{.j'})_{j' \in \mathcal{J}_{-j}}
		\end{pmatrix}   - \mathbf{1}\alpha_{|r} - \epsilon_{|r} \right)  B^{-1}B_{j.} +\epsilon_{.j},
	\end{align*}
	so that
	\begin{multline*}
		Y_{.j}=\sum_{k \in \{m\}\cup\mathcal{J}_{-j}} \left(\sum_{\ell \in \{m,\}\cup\mathcal{J}_{-j}} B^{-1}_{lk}B_{jl}\right)Y_{.k} \\
		- \sum_{k \in \{m\}\cup\mathcal{J}_{-j}} \left(\sum_{\ell \in \{m\}\cup\mathcal{J}_{-j}} B^{-1}_{lk}B_{jl}\right)(\mathbf{1}\alpha_{k}+\epsilon_{.k}) + \epsilon_{.j} + \mathbf{1}\alpha_{j}.
    \end{multline*}
	which leads to the desired solution.
	
\end{proof}

\subsection{Proof of Proposition \ref{prop:mean_formula_general}}
\label{sec:proofmean}

\newtheorem*{prop:mean_formula_general}{Proposition \ref{prop:mean_formula_general}}

\begin{prop:mean_formula_general}[Mean estimator]
Consider the PPCA model \eqref{eq:model}. Under Assumptions \ref{hyp1} and \ref{hyp2}, an estimator of the mean of a MNAR variable $Y_{.m}$, for $m\in\mathcal{M}$, can be constructed as follows: choose $j\in\mathcal{J}$, and compute
	\begin{equation*}
		\hat{\alpha}_m :=\frac{\hat{\alpha}_{j}-\hat{\mathcal{B}}_{j\rightarrow m, \mathcal{J}_{-j}[0]}^c-\sum_{j'  \in \mathcal{J}_{-j}}\hat{\mathcal{B}}_{j\rightarrow m, \mathcal{J}_{-j}[j']}^c\hat{\alpha}_{j'}}{ \hat{\mathcal{B}}_{j\rightarrow m, \mathcal{J}_{-j}[m]}^c}, 
	\end{equation*}
	with the $(\hat{\mathcal{B}}_{j\rightarrow m, \mathcal{J}_{-j}[k]})$'s estimators of the coefficients given in Definition \ref{def:coeff_complete_Case} and assuming that the coefficient $\mathcal{B}_{j\rightarrow m, \mathcal{J}_{-j}[m]}^c$ estimated by $\hat{\mathcal{B}}_{j\rightarrow m, \mathcal{J}_{-j}[m]}^c$ is non zero. 
	
	Under the additional Assumptions \ref{hyp3} and \ref{hyp4}, this estimator is consistent. 
\end{prop:mean_formula_general}

\begin{proof}
    The main goal is to obtain a formula for $\alpha_{.m}$, \textit{i.e.}
    \begin{equation}\label{eq:meanproof}
    {\alpha}_m=\frac{{\alpha}_{j}-{\mathcal{B}}_{j\rightarrow m, \mathcal{J}_{-j}[0]}^c-\sum_{j'  \in \mathcal{J}_{-j}}{\mathcal{B}}_{j\rightarrow m, \mathcal{J}_{-j}[j']}^c{\alpha}_{j'}}{ {\mathcal{B}}_{j\rightarrow m, \mathcal{J}_{-j}[m]}^c},
    \end{equation}
    from which an estimator can be deduced.
    The idea is to express $\alpha_j$ from $\alpha_m$ and $(\alpha_j')_{j' \in \mathcal{J}_{-j}}$.
    Note that  $\mathbb{E}[Y_{.j}]=\mathbb{E}[\mathbb{E}[Y_{.j}|(Y_{.k})_{k \in \widebar{\{j\}}}]]$. Assumption \ref{hyp2} leads to 
    $$\mathbb{E}[Y_{.j}|(Y_{.k})_{k \in \widebar{\{j\}}}]=\mathbb{E}[Y_{.j}|(Y_{.k})_{k \in \widebar{\{j\}}},\Omega_{.m}=1].$$
	Then, by Definition \ref{def:coeff_complete_Case} which gives
	$(Y_{.j})_{|\Omega_{.m}=1}$,
	\footnotesize
	\begin{align*}
		&\mathbb{E}[Y_{.j}|(Y_{.k})_{k \in \widebar{\{j\}}},\Omega_{.m}=1] \\
		&=\mathbb{E}\left[\mathcal{B}^c_{j\rightarrow m,\mathcal{J}_{-j}[0]}
    +\sum_{k\in \{m\}\cup\mathcal{J}_{-j}} \mathcal{B}^c_{j\rightarrow m,\mathcal{J}_{-j}[k]} Y_{.k}  + \zeta^c\bigg|(Y_{.k})_{k \in \widebar{\{j\}}}\right] \\
		&= \mathcal{B}^c_{j\rightarrow m,\mathcal{J}_{-j}[0]}
    +\sum_{k\in \{m\}\cup\mathcal{J}_{-j}} \mathcal{B}^c_{j\rightarrow m,\mathcal{J}_{-j}[k]} Y_{.k}+\mathbb{E}\left[\zeta^c\bigg|(Y_{.k})_{k \in \widebar{\{j\}}}\right]
	\end{align*}
	\normalsize
	
	Thus, by taking the mean and given that $\mathbb{E}[\epsilon_{.k}]=0, \forall k \in \{m\}\cup\mathcal{J}_{-j}$, one has
	$$\alpha_j=\mathcal{B}^c_{j\rightarrow m,\mathcal{J}_{-j}[0]}+\sum_{j'\in \mathcal{J}_{-j}} \mathcal{B}^c_{j\rightarrow m,\mathcal{J}_{-j}[j']} \alpha_{j'} + \mathcal{B}^c_{j\rightarrow m,\mathcal{J}_{-j}[m]} \alpha_{m} ,$$
	implying Equation \eqref{eq:meanproof}, provided that $\mathcal{B}_{j\rightarrow m, \mathcal{J}_{-j}[m]}^c\neq 0$. 
    
    From this formula for the mean $\alpha_m$, one define its estimator $\hat{\alpha}_m$ as in \eqref{eq:expectation_main}. It is trivially consistent as the linear combination of consistent quantities under \ref{hyp3} and \ref{hyp4}
\end{proof}

\subsection{Proof of Proposition \ref{prop:var_formula_general}}
\label{sec:proofvar}

\newtheorem*{prop:var_formula_general}{Proposition \ref{prop:var_formula_general}}

\begin{prop:var_formula_general}[Variance and covariances estimators] 
	Consider the PPCA model \eqref{eq:model}. Under Assumptions \ref{hyp1} and \ref{hyp2}, an estimator of the variance of a MNAR variable $Y_{.m}$ for $m\in\mathcal{M}$ and its covariances with the pivot variables, can be constructed as follows: choose $j \in \mathcal{J}$ and compute
	\begin{equation*}
		\begin{pmatrix}
			\widehat{\mathrm{Var}}(Y_{.m}) &
			\widehat{\mathrm{Cov}}(Y_{.m},(Y_{.k})_{k\in \mathcal{J}})
		\end{pmatrix}^T:=(\widehat{M}_j)^{-1}\widehat{P}_j,
	\end{equation*}
	assuming that $\sigma^2$ tends to zero and the inverse of the matrix ${M}_j$ estimated by $(\widehat{M}_j)^{-1}$ exists, with 
	\begin{equation*}
    \widehat{M}_j = 
    \begin{tikzpicture}[baseline={-0.5ex},mymatrixenv,large/.style={font=\large}]
        \matrix [mymatrix,inner sep=4pt] (m)  
        {
        (\hat{\mathcal{B}}_{j\rightarrow m,\mathcal{J}_{-j}[m]}^c)^2 &0  & 2\hat{\mathcal{B}}_{j\rightarrow m,\mathcal{J}_{-j}[m]}^c\left(\hat{\mathcal{B}}_{j\rightarrow m,\mathcal{J}_{-j}[\mathcal{J}_{-j}]}^c\right)^T  \\
        -(\hat{\mathcal{B}}_{k\rightarrow m,\mathcal{J}_{-k}[m]}^c)_{k \in \mathcal{J}} &  & \\
        };

        \node[large] at (m-2-3){$(\widehat{M}^k)_{k \in \mathcal{J}}$};
        \mymatrixbraceright{1}{1}{$\in \mathbb{R}$}
        \mymatrixbraceright{2}{2}{$\in \mathbb{R}^{r}$}
        \mymatrixbracetop{2}{3}{$\in \mathbb{R}^{r}$}
        \mymatrixbracetop{1}{1}{$\in \mathbb{R}$}
    \end{tikzpicture}
\end{equation*}
    Let us precise that $\widehat{M}_j \in \mathbb{R}^{(r+1)\times(r+1)}$. One has $(\hat{\mathcal{B}}_{k\rightarrow m,\mathcal{J}_{-k}[m]}^c)_{k \in \mathcal{J}}=\begin{pmatrix}
    \hat{\mathcal{B}}_{j_1\rightarrow m,\mathcal{J}_{-j_1}[m]}^c \\
    \vdots \\
    \hat{\mathcal{B}}_{j_r\rightarrow m,\mathcal{J}_{-j_r}[m]}^c
    \end{pmatrix}$. 
    
    One details $\widehat{M^k}$ for $k=j_1$ and the same definition is valid for all $k \in \mathcal{J}$. 
    $$\widehat{M^{j_1}}=\begin{pmatrix}1 & -\hat{\mathcal{B}}_{j_1\rightarrow m,\mathcal{J}_{-j_1}[j_2]}^c & \dots &  -\hat{\mathcal{B}}_{j_1\rightarrow m,\mathcal{J}_{-j_1}[j_r]}^c
    \end{pmatrix} \in \mathbb{R}^r
    $$
    
    \begin{equation*}
    \widehat{P}_j = 
    \begin{tikzpicture}[baseline={-0.5ex},mymatrixenv,large/.style={font=\large}]
        \matrix [mymatrix,inner sep=4pt] (m)  
        {
        (\widehat{\mathrm{Var}}(Y_{.j})-Q^{ c}-(\hat{\mathcal{B}}^c_{j\rightarrow m,\mathcal{J}_{-j}[\mathcal{J}_{-j}]})^T \widehat{\mathrm{Var}}(Y_{\mathcal{J}_{-j}})\hat{\mathcal{B}}^c_{j\rightarrow m,\mathcal{J}_{-j}[\mathcal{J}_{-j}]} \\
        \left(((\hat{\mathcal{B}}^c_{k\rightarrow m,\mathcal{J}_{-k}})^T\begin{pmatrix}
	    1 & \hat{\alpha}_{m} & (\hat{\alpha}_{\ell})_{\ell \in \mathcal{J}_{-k}}
	    \end{pmatrix}^T -\hat{\alpha}_k)\hat{\alpha}_{m}]\right)_{k \in \mathcal{J}} \\
        };

        \mymatrixbraceleft{1}{1}{$\in \mathbb{R}$}
        \mymatrixbraceleft{2}{2}{$\in \mathbb{R}^{r}$}
        \mymatrixbracetop{1}{1}{$\in \mathbb{R}$}
    \end{tikzpicture}
    \end{equation*}
    

	\begin{multline*}
	\hat{Q}^{c}=\left(\widehat{\mathrm{Var}}(Y_{.j})\big| \Omega_{.m}=1\right)
	\\
	-\left(\widehat{\mathrm{Cov}}((Y_{.k})_{k \in \widebar{\{j\}}},Y_{.j}) \widehat{\mathrm{Var}}((Y_{.k})_{k \in \widebar{\{j\}}})^{-1} \widehat{\mathrm{Cov}}((Y_{.k})_{k \in \widebar{\{j\}}},Y_{.j})^T \big| \Omega_{.m}=1\right).
	\end{multline*}

	Under the additional Assumptions \ref{hyp3} and \ref{hyp4}, 
	the estimators for the variance of $Y_{.m}$ and its covariances with the pivot variables given in \eqref{eq:cov_matrix_general} are consistent. 
\end{prop:var_formula_general}

\begin{proof}
    As for the mean, to derive some estimator of the variance and the covariances, we want to obtain a formula as 
    \begin{equation}\label{eq:varcovformula}
		M_j\begin{pmatrix}
			{\mathrm{Var}}(Y_{.m}) &
			{\mathrm{Cov}}(Y_{.m},(Y_{.k})_{k\in \mathcal{J}})
		\end{pmatrix}^T=\left({P}_j-\mathcal{O}(\sigma^2)\right),
	\end{equation}
	with 

    \begin{equation*}
    M_j = 
    \begin{tikzpicture}[baseline={-0.5ex},mymatrixenv,large/.style={font=\large}]
        \matrix [mymatrix,inner sep=4pt] (m)  
        {
        (\mathcal{B}_{j\rightarrow m,\mathcal{J}_{-j}[m]}^c)^2 &0  & 2\mathcal{B}_{j\rightarrow m,\mathcal{J}_{-j}[m]}^c\left(\mathcal{B}_{j\rightarrow m,\mathcal{J}_{-j}[\mathcal{J}_{-j}]}^c\right)^T  \\
        -(\mathcal{B}_{k\rightarrow m,\mathcal{J}_{-k}[m]}^c)_{k \in \mathcal{J}} &  & \\
        };

        \node[large] at (m-2-3){$(M^k)_{k \in \mathcal{J}}$};
        \mymatrixbraceright{1}{1}{$\in \mathbb{R}$}
        \mymatrixbraceright{2}{2}{$\in \mathbb{R}^{r}$}
        \mymatrixbracetop{2}{3}{$\in \mathbb{R}^{r}$}
        \mymatrixbracetop{1}{1}{$\in \mathbb{R}$}
    \end{tikzpicture}
\end{equation*}
    Let us precise that $M_j \in \mathbb{R}^{(r+1)\times(r+1)}$. One has $(\mathcal{B}_{k\rightarrow m,\mathcal{J}_{-k}[m]}^c)_{k \in \mathcal{J}}=\begin{pmatrix}
    \mathcal{B}_{j_1\rightarrow m,\mathcal{J}_{-j_1}[m]}^c \\
    \vdots \\
    \mathcal{B}_{j_r\rightarrow m,\mathcal{J}_{-j_r}[m]}^c
    \end{pmatrix}$. 
    
    One details $M^k$ for $k=j_1$ and the same definition is valid for all $k \in \mathcal{J}$. 
    $$M^{j_1}=\begin{pmatrix}1 & -\mathcal{B}_{j_1\rightarrow m,\mathcal{J}_{-j_1}[j_2]}^c & \dots &  -\mathcal{B}_{j_1\rightarrow m,\mathcal{J}_{-j_1}[j_r]}^c
    \end{pmatrix} \in \mathbb{R}^r
    $$
    
    \begin{equation*}
    P_j = 
    \begin{tikzpicture}[baseline={-0.5ex},mymatrixenv,large/.style={font=\large}]
        \matrix [mymatrix,inner sep=4pt] (m)  
        {
        (\mathrm{Var}(Y_{.j})-Q^{ c}-(\mathcal{B}^c_{j\rightarrow m,\mathcal{J}_{-j}[\mathcal{J}_{-j}]})^T \mathrm{Var}(Y_{\mathcal{J}_{-j}})\mathcal{B}^c_{j\rightarrow m,\mathcal{J}_{-j}[\mathcal{J}_{-j}]} \\
        \left(((\mathcal{B}^c_{k\rightarrow m,\mathcal{J}_{-k}})^T\begin{pmatrix}
	    1 & \mathbb{E}[Y_{.m}] & (\mathbb{E}[Y_{.\ell}])_{\ell \in \mathcal{J}_{-k}}
	    \end{pmatrix}^T -\mathbb{E}[Y_{.k}])\mathbb{E}[Y_{.m}]\right)_{k \in \mathcal{J}} \\
        };

        \mymatrixbraceleft{1}{1}{$\in \mathbb{R}$}
        \mymatrixbraceleft{2}{2}{$\in \mathbb{R}^{r}$}
        \mymatrixbracetop{1}{1}{$\in \mathbb{R}$}
    \end{tikzpicture}
    \end{equation*}
    
    \begin{equation*}
    \mathcal{O}(\sigma^2) = 
    \begin{tikzpicture}[baseline={-0.5ex},mymatrixenv,large/.style={font=\large}]
        \matrix [mymatrix,inner sep=4pt] (m)  
        {
         o_{\mathrm{var}}(\sigma^2) \\
        -\left(o_{\mathrm{cov},k}(\sigma^2)\right)_{k \in \mathcal{J}} \\
        };

        \mymatrixbraceleft{1}{1}{$\in \mathbb{R}$}
        \mymatrixbraceleft{2}{2}{$\in \mathbb{R}^{r}$}
        \mymatrixbracetop{1}{1}{$\in \mathbb{R}$}
    \end{tikzpicture},
    \end{equation*}
    with $o_{\mathrm{var}}(\sigma^2)$ and $o_{\mathrm{cov},k}(\sigma^2)$ detailed in \eqref{eq:ovar} and \eqref{eq:ocov} respectively. 
	\begin{multline}\label{eq:Qcdef}
	Q^{c}=\left(\mathrm{Var}(Y_{.j})\big| \Omega_{.m}=1\right)
	\\
	-\left(\mathrm{Cov}((Y_{.k})_{k \in \widebar{\{j\}}},Y_{.j}) \mathrm{Var}((Y_{.k})_{k \in \widebar{\{j\}}})^{-1} \mathrm{Cov}((Y_{.k})_{k \in \widebar{\{j\}}},Y_{.j})^T \big| \Omega_{.m}=1\right).
	\end{multline}

    The strategy is to prove each equality of the linear system in \eqref{eq:varcovformula}.
	
	\paragraph{Deriving an equation for the variance.}	The idea is first to express $\mathrm{Var}(Y_{.j})$ from $\mathrm{Var}(Y_{.m})$, $(\mathrm{Var}(Y_{.j'}))_{j' \in \mathcal{J}_{-j}}$ and $(\mathrm{Cov}(Y_{.k},Y_{.\ell}))_{k \neq \ell \in \{m\}\cup\mathcal{J}_{-j}}$.
	The law of total variance reads as 
	\begin{equation}\label{eq:varY2toy}
	\mathrm{Var}(Y_{.j})=\mathbb{E}[\mathrm{Var}(Y_{.j}|Z)]+\mathrm{Var}(\mathbb{E}[Y_{.j}|Z]),
	\end{equation}
	with $Z=(Y_{.k})_{k \in \widebar{\{j\}}}$.
	
	For the first term in \eqref{eq:varY2toy}, using Assumption \ref{hyp2}, one has
	$$Y_{.j}\independent (\Omega_{.m}=1) | Z$$
	which leads to
	$$\mathrm{Var}(Y_{.j}|Z)=\mathrm{Var}(Y_{.j}|Z,\Omega_{.m}=1).$$
	The conditional variance for a Gaussian vector gives
	$$\mathrm{Var}(Y_{.j}|Z)=\mathrm{Var}(Y_{.j})-\mathrm{Cov}(Z,Y_{.j}) \mathrm{Var}(Z)^{-1} \mathrm{Cov}(Z,Y_{.j})^T,$$
	implying that $$\mathrm{Var}(Y_{.j}|Z,\Omega_{.m}=1)=\left(\mathrm{Var}(Y_{.j})-\mathrm{Cov}(Z,Y_{.j}) \mathrm{Var}(Z)^{-1} \mathrm{Cov}(Z,Y_{.j})^T\big| \Omega_{.m}=1 \right)$$ and then, as deterministic quantity,
	$$\mathbb{E}[\mathrm{Var}(Y_{.j}|Z)]=\left(\mathrm{Var}(Y_{.j})-\mathrm{Cov}(Z,Y_{.j}) \mathrm{Var}(Z)^{-1} \mathrm{Cov}(Z,Y_{.j})^T\big| \Omega_{.m}=1 \right).$$
	
	One has
	\begin{multline*}
	\mathrm{Cov}(Z,Y_{.j}) \mathrm{Var}(Z)^{-1} \mathrm{Cov}(Z,Y_{.j})^T= \\
	\mathrm{Cov}((Y_{.k})_{k \in \widebar{\{j\}}},Y_{.j}) \mathrm{Var}((Y_{.k})_{k \in \widebar{\{j\}}})^{-1} \mathrm{Cov}((Y_{.k})_{k \in \widebar{\{j\}}},Y_{.j})^T
	\end{multline*}
	
	leading to 
	\begin{equation}
	\label{eq:espvarcondtoy}
	\mathbb{E}[\mathrm{Var}(Y_{.j}|Z)]=Q^c,
	\end{equation}
	where $Q^c$ is defined in \eqref{eq:Qcdef}.
	
	For the second term of \eqref{eq:varY2toy}, remark that \ref{hyp2} implies that
	$$\mathrm{Var}(\mathbb{E}[Y_{.j}|Z])=\mathrm{Var}(\mathbb{E}[Y_{.j}|Z,\Omega_{.m}=1]),$$ 
	and
	$$
	\mathrm{Var}(\mathbb{E}[Y_{.j}|Z,\Omega_{.m}=1])=\mathrm{Var}\left(\mathbb{E}\left[\mathcal{B}^c_{j\rightarrow m,\mathcal{J}_{-j}[0]}
    +\sum_{k\in \{m\}\cup\mathcal{J}_{-j}} \mathcal{B}^c_{j\rightarrow m,\mathcal{J}_{-j}[k]} Y_{.k}  + \zeta^c\bigg|Z\right]\right),$$
    \textit{i.e.}
    \begin{multline*}
	\mathrm{Var}(\mathbb{E}[Y_{.j}|Z,\Omega_{.m}=1])
	\\=\mathrm{Var}\left(\sum_{k\in \{m\}\cup\mathcal{J}_{-j}} \mathcal{B}^c_{j\rightarrow m,\mathcal{J}_{-j}[k]} Y_{.k}-\sum_{k\in \{m\}\cup\mathcal{J}_{-j}} \mathcal{B}^c_{j\rightarrow m,\mathcal{J}_{-j}[k]} \mathbb{E}[\epsilon_{.k}|Z] +\mathcal{B}^c_{j\rightarrow m,\mathcal{J}_{-j}[0]}+\mathbb{E}[\epsilon_{.j}]\right)
    \end{multline*}
	
	In the variance, the first term is obtained using that the variables $(Y_{.k})_{k \in \{m\} \cup \mathcal{J}_{-j}}$ are $Z-$measurable. The two last terms use that $\mathcal{B}^c_{j\rightarrow m,\mathcal{J}_{-j}[k]}$ is a constant and $\epsilon_{.j}$ is independent of Z.
	To calculate the second term, involving $\mathbb{E}[\epsilon_{.k}|Z]$, one first shows that the vector
	$\begin{pmatrix} (Y_{.k})_{k \in \{m\} \cup \mathcal{J}_{-j}} & (\epsilon_{.k})_{k \in \{m\} \cup \mathcal{J}_{-j}} \end{pmatrix}^T$ is gaussian. Indeed,
	\begin{itemize}
	    \item  $(Y_{.k})_{k \in \{m\} \cup \mathcal{J}_{-j}}$ is a gaussian vector, using the model \eqref{eq:model}. 
	    \item $(\epsilon_{.k})_{k \in \{m\} \cup \mathcal{J}_{-j}}$ is a gaussian vector, because its components are independent gaussian variables. 
	    \item for $k\neq \ell \in  \{m\} \cup \mathcal{J}_{-j}$, $\begin{pmatrix}WB_{k.} & \epsilon_{.\ell}\end{pmatrix}^T$ is a gaussian vector, because $Y_{.k} \independent \epsilon_{.\ell}$. 
	    
	    \item for $k \in \{m\} \cup \mathcal{J}_{-j}$, $\begin{pmatrix} Y_{.k} & \epsilon_{.k}\end{pmatrix}^T$ is a gaussian vector, given that $Y_{.k}$ is a linear combination of $\begin{pmatrix}WB_{k.} & \epsilon_{.k}\end{pmatrix}^T$ which is gaussian, as $WB_{k.}$ and $\epsilon_{.k}$ are independent gaussian variables.
	\end{itemize}

	Thus, 
	\begin{align*}
	\mathbb{E}[\epsilon_{.k}|Z]&=\mathbb{E}[\epsilon_{.k}]+\mathrm{Cov}(\epsilon_{.k}, Z) \mathrm{Var}(Z)^{-1} (Z - \mathbb{E}[Z]) \\
	&=\mathrm{Cov}(\epsilon_{.k}, Y_{.k}) (\mathrm{Var}(Z)^{-1})_{k.} (Z - \mathbb{E}[Z]),
	\end{align*}
	using $\mathrm{Cov}(\epsilon_{.k}, Y_{.l})=0$, for $k\neq l$. $\Gamma_Z=\mathrm{Var}(Z)^{-1}$ denotes the inverse of the covariance matrix of $Z$ and $(\Gamma_Z)_{k.}$ is its k-th row. It leads to
	\begin{equation}\label{eq:epsilonkcond}
	    \mathbb{E}[\epsilon_{.k}|Z]=\sigma^2 (\Gamma_Z)_{k.} (Z - \mathbb{E}[Z]).
	\end{equation}
	given that $\mathrm{Cov}(\epsilon_{.k}, Y_{.k})=\mathrm{Cov}(\epsilon_{.k}, WB_{k.}+\epsilon_{.k})=\mathrm{Var}(\epsilon_{.k})$. 
	
	
	Therefore, 
	\begin{multline}\label{eq:condvarToy}
	\mathrm{Var}(\mathbb{E}[Y_{.j}|Z,\Omega_{.m}=1])=\sum_{k\in \{m\}\cup\mathcal{J}_{-j}} (\mathcal{B}^c_{j\rightarrow m,\mathcal{J}_{-j}[k]})^2 \mathrm{Var}(Y_{.k})\\
	+\sum_{(k<\ell)\in \{m\}\cup\mathcal{J}_{-j}}2\mathcal{B}^c_{j\rightarrow m,\mathcal{J}_{-j}[k]}\mathcal{B}^c_{j\rightarrow m,\mathcal{J}_{-j}[\ell]}\mathrm{Cov}(Y_{.k},Y_{.\ell}) + o_{\mathrm{var}}(\sigma^2),
	\end{multline}
	where 
	\small
	\begin{multline}\label{eq:ovar}
	o_{\mathrm{var}}(\sigma^2)=-2\sigma^2\sum_{(k,\ell) \in \{m\}\cup\mathcal{J}_{-j}}\mathcal{B}_{j\rightarrow m,\mathcal{J}_{-j}[k]}^c\mathcal{B}_{j\rightarrow m,\mathcal{J}_{-j}[\ell]}^c\sum_{\ell' \in \{m\}\cup\mathcal{J}_{-j}}(\Gamma_Z)_{\ell\ell'}\mathrm{Cov}(Y_{.k},Y_{.\ell'}) \\
	+\sigma^4 \sum_{k \in \{m\}\cup\mathcal{J}_{-j}}(\mathcal{B}_{j\rightarrow m,\mathcal{J}_{-j}[k]}^c)^2 \left(\sum_{(\ell<\ell') \in \{m\}\cup\mathcal{J}_{-j}}(\Gamma_Z)_{k\ell}^2\mathrm{Var}(Y_{.\ell})
	-2 (\Gamma_Z)_{k\ell}(\Gamma_Z)_{k\ell'}\mathrm{Cov}(Y_{.\ell},Y_{.\ell'})\right) \\
	-2\sigma^4\sum_{(k<\ell) \in \{m\}\cup\mathcal{J}_{-j}}\mathcal{B}_{j\rightarrow m,\mathcal{J}_{-j}[k]}^c\mathcal{B}_{j\rightarrow m,\mathcal{J}_{-j}[\ell]}^c\sum_{(k',\ell') \in \{m\}\cup \mathcal{J}_{-j}}(\Gamma_Z)_{kk'}(\Gamma_Z)_{\ell\ell'}\mathrm{Cov}(Y_{.k'},Y_{.\ell'})
    \end{multline}
	\normalsize
	Combining \eqref{eq:espvarcondtoy} with \eqref{eq:condvarToy}, one get the following expression for the first line of the linear system
	\begin{multline}\label{eq:vartoy}
	(\mathcal{B}^c_{j\rightarrow m,\mathcal{J}_{-j}[m]})^2\mathrm{Var}(Y_{.m})+\sum_{j'\in \mathcal{J}_{-j}}2\mathcal{B}^c_{j\rightarrow m,\mathcal{J}_{-j}[j']}\mathcal{B}^c_{j\rightarrow m,\mathcal{J}_{-j}[m]}\mathrm{Cov}(Y_{.j'},Y_{.m})
	\\
	=\mathrm{Var}(Y_{.j})-Q^c-(\mathcal{B}^c_{j\rightarrow m,\mathcal{J}_{-j}[\mathcal{J}_{-j}]})^T \mathrm{Var}(Y_{\mathcal{J}_{-j}})\mathcal{B}^c_{j\rightarrow m,\mathcal{J}_{-j}[\mathcal{J}_{-j}]}-o_{\mathrm{var}}(\sigma^2)
	\end{multline}
	
	\textbf{Deriving equations for the covariances.}	
	Let $k$ be an element of $\mathcal{J}$, our objective is to express $\mathrm{Cov}(Y_{.m},Y_{.k})$ from $\mathrm{Var}(Y_{.m})$, $\alpha_m$, $(\alpha_k)_{k \in \mathcal{J}}$ and $(\mathrm{Cov}(Y_{.m},Y_{.k}))_{k\in \{m\}\cup\mathcal{J}}$. 
	\begin{align}
	\nonumber \notag
	\mathrm{Cov}(Y_{.m},Y_{.k})&=\mathbb{E}[Y_{.m}Y_{.k}]-\mathbb{E}[Y_{.m}]\mathbb{E}[Y_{k}] \\
	\notag
	&=\mathbb{E}[\mathbb{E}[Y_{.m}Y_{.k}|Z]]-\mathbb{E}[Y_{.m}]\mathbb{E}[Y_{.k}] \\
	\label{eq:covbegintoy}
	&=\mathbb{E}[Y_{.m}\mathbb{E}[Y_{.k}|Z]]-\mathbb{E}[Y_{.m}]\mathbb{E}[Y_{.k}],
	\end{align}
	with $Z=(Y_{.\ell})_{\ell \in \widebar{\{k\}}}$.
	
	For the first term in \eqref{eq:covbegintoy}, one has
	\begin{align*}
	\mathbb{E}[Y_{.m}\mathbb{E}[Y_{.k}|Z]]\overset{(i)}{=}& \mathbb{E}[Y_{.m}\mathbb{E}[Y_{.k}|Z,\Omega_{.m}=1]] \\
	\overset{(ii)}{=}&\mathbb{E}\left[Y_{.m}\left(\mathcal{B}^c_{k\rightarrow m,\mathcal{J}_{-k}[0]}
    +\sum_{\ell\in \{m\}\cup\mathcal{J}_{-k}} \mathcal{B}^c_{k\rightarrow m,\mathcal{J}_{-k}[\ell]} Y_{.\ell}  + \mathbb{E}[\zeta_k^c|Z]\right)\right] \\
	\overset{(iii)}{=}&\mathcal{B}^c_{k\rightarrow m,\mathcal{J}_{-k}[0]}\mathbb{E}[Y_{.m}]+\mathcal{B}^c_{k\rightarrow m,\mathcal{J}_{-k}[m]}\mathbb{E}[Y_{.m}^2] \\
	&+\sum_{\ell\in \mathcal{J}_{-k}} \mathcal{B}^c_{k\rightarrow m,\mathcal{J}_{-k}[\ell]}\mathbb{E}[Y_{.m}Y_{.\ell}]+o_{\mathrm{cov},k}(\sigma^2)
	\end{align*}
	with $\zeta_k^c=-\sum_{\ell\in\mathcal{J}_{-k}} \mathcal{B}^c_{k\rightarrow m,\mathcal{J}_{-k}[\ell]} \epsilon_{.\ell} - \mathcal{B}^c_{k\rightarrow m,\mathcal{J}_{-k}[m]} \epsilon_{.m}+ \epsilon_{.k}.$
	
	Assumption \ref{hyp2} and Definition \ref{def:coeff_complete_Case}  are used for (i) and (ii) respectively. For (iii), using \eqref{eq:epsilonkcond}, one has 
	$$\mathbb{E}[Y_{.m}\mathbb{E}[\zeta_k^c|Z]]=\mathbb{E}\left[Y_{.m}\left(-\sum_{\ell \in \mathcal{J}_{-k}} \mathcal{B}^c_{k\rightarrow m,\mathcal{J}_{-k}[\ell]}\sigma^2 (\Gamma_Z)_{\ell.} (Z - \mathbb{E}[Z])-\mathcal{B}^c_{k\rightarrow m,\mathcal{J}_{-k}[m]} \epsilon_{.m}\right)\right],$$
	given that $\mathbb{E}[\epsilon_{.k}|Z]=\mathbb{E}[\epsilon_{.k}]=0$ by independence. 
	\begin{multline*}
	\mathbb{E}[Y_{.m}\mathbb{E}[\zeta_k^c|Z]]\\
	=-\sigma^2\left(\mathbb{E}\left[\sum_{\ell \in \mathcal{J}_{-k}}\mathcal{B}^c_{k\rightarrow m,\mathcal{J}_{-k}[\ell]}Y_{.m}\sum_{\ell'\in \mathcal{J}_{-k}}(\Gamma_Z)_{\ell\ell'} (Y_{.\ell'} - \mathbb{E}[Y_{.\ell'}])\right]+\mathcal{B}^c_{k\rightarrow m,\mathcal{J}_{-k}[m]}\right),
	\end{multline*}
	because $\mathbb{E}[Y_{.m}\epsilon_{.m}]=\mathrm{Cov}(Y_{.m},\epsilon_{.m})+\mathbb{E}[Y_{.m}]\mathbb{E}[\epsilon_{.m}]=\mathrm{Cov}(Y_{.m},\epsilon_{.m})=\sigma^2$. In addition, 
	\begin{align*}
	&\mathbb{E}\left[\sum_{\ell \in \mathcal{J}_{-k}}\mathcal{B}^c_{k\rightarrow m,\mathcal{J}_{-k}[\ell]}Y_{.m}\sum_{\ell'\in \mathcal{J}_{-k}}(\Gamma_Z)_{\ell\ell'} (Y_{.\ell'} - \mathbb{E}[Y_{.\ell'}])\right]\\
	&=\sum_{\ell \in \mathcal{J}_{-k}}\sum_{\ell'\in \mathcal{J}_{-k}}(\Gamma_Z)_{\ell\ell'}\mathcal{B}^c_{k\rightarrow m,\mathcal{J}_{-k}[\ell]}\left(\mathrm{Cov}\left(Y_{.m},Y_{.\ell'}\right) + \mathbb{E}[Y_{.m}]\mathbb{E}[(Y_{.\ell'} - \mathbb{E}[Y_{.\ell'}])]\right) \\
	&=\sum_{\ell \in \mathcal{J}_{-k}}\sum_{\ell'\in \mathcal{J}_{-k}}(\Gamma_Z)_{\ell\ell'}\mathcal{B}^c_{k\rightarrow m,\mathcal{J}_{-k}[\ell]}\mathrm{Cov}\left(Y_{.m},Y_{.\ell'}\right)
	\end{align*}
	It implies that, in (iii), 
	\begin{equation}\label{eq:ocov}
	o_{\mathrm{cov},k}(\sigma^2)=-\sigma^2\left(\sum_{\ell \in \mathcal{J}_{-k}}\sum_{\ell'\in \mathcal{J}_{-k}}(\Gamma_Z)_{\ell\ell'}\mathcal{B}^c_{k\rightarrow m,\mathcal{J}_{-k}[\ell]}\mathrm{Cov}\left(Y_{.m},Y_{.\ell'}\right)+\mathcal{B}^c_{k\rightarrow m,\mathcal{J}_{-k}[m]}\right)
	\end{equation}
	
	
	Equation \eqref{eq:covbegintoy} leads thus to
	\begin{multline}
	\label{eq:cov2toyproof}
	\mathrm{Cov}(Y_{.m},Y_{.k})=\mathcal{B}^c_{k\rightarrow m,\mathcal{J}_{-k}[0]}\mathbb{E}[Y_{.m}]+\mathcal{B}^c_{k\rightarrow m,\mathcal{J}_{-k}[m]}(\mathrm{Var}(Y_{.m})+\mathbb{E}[Y_{.m}]^2) \\
	+\sum_{\ell\in \mathcal{J}_{-k}} \mathcal{B}^c_{k\rightarrow m,\mathcal{J}_{-k}[\ell]}(\mathrm{Cov}(Y_{.m},Y_{.\ell})+\mathbb{E}[Y_{.m}]\mathbb{E}[Y_{.\ell}]) - \mathbb{E}[Y_{.m}]\mathbb{E}[Y_{.k}]+o_{\mathrm{cov},k}(\sigma^2),
	\end{multline}
	which can be rewritten as
	\begin{multline}
	\label{eq:cov2toyproof}
	\mathrm{Cov}(Y_{.m},Y_{.k})-\mathcal{B}^c_{k\rightarrow m,\mathcal{J}_{-k}[m]}\mathrm{Var}(Y_{.m}) -\sum_{\ell\in \mathcal{J}_{-k}} \mathcal{B}^c_{k\rightarrow m,\mathcal{J}_{-k}[\ell]}\mathrm{Cov}(Y_{.m},Y_{.\ell}) \\
	=((\mathcal{B}^c_{k\rightarrow m,\mathcal{J}_{-k}})^T\begin{pmatrix}
	1 & \mathbb{E}[Y_{.m}] & (\mathbb{E}[Y_{.\ell}])_{\ell \in \mathcal{J}_{-k}}
	\end{pmatrix}^T -\mathbb{E}[Y_{.k}])\mathbb{E}[Y_{.m}]
	+o_{\mathrm{cov},k}(\sigma^2),
	\end{multline}
	Combining Equations \eqref{eq:vartoy} and \eqref{eq:cov2toyproof} forms the desired matrix system \eqref{eq:varcovformula}.
	
	From these formulae for $\begin{pmatrix}
			{\mathrm{Var}}(Y_{.m}) &
			{\mathrm{Cov}}(Y_{.m},(Y_{.k})_{k\in \mathcal{J}})
		\end{pmatrix}^T$, assuming that ${M}_j$ is invertible and that $\sigma^2$ tends to zero, one get their estimators $\begin{pmatrix}
			\widehat{\mathrm{Var}}(Y_{.m}) &
			\widehat{\mathrm{Cov}}(Y_{.m},(Y_{.k})_{k\in \mathcal{J}})
		\end{pmatrix}^T$ defined in \eqref{eq:estimcov_main}.
	
	As for the consistency, $\hat{\alpha}_m$ is a consistent estimator for $\alpha_m$ by using Proposition \ref{prop:mean_formula_general}. The estimators in \eqref{eq:estimcov_main} are consistent, under Assumption \ref{hyp3} and \ref{hyp4}.
\end{proof}

\subsection{Proof of Proposition \ref{prop:covmiss}} 
\label{sec:proofcov2}

For deriving the covariance between a MNAR variable and a MNAR or not pivot variable, we assume the following
\begin{enumerate}[label=\textbf{A\arabic*.}]
    \setcounter{enumi}{4}
    \item \label{hyp5} $\forall m \in \mathcal{M}$, $\forall \ell \in \widebar{\mathcal{J}}$, for all set $\mathcal{H} \subset \mathcal{J}_{-j}$ such that $|\mathcal{H}|=r-2$,  $\begin{pmatrix} B_{.m} & B_{.\ell} & (B_{.j'})_{j' \in \mathcal{H}} \end{pmatrix}$ is invertible,
    \item \label{hyp6} 
    $\forall k \in \widebar{\mathcal{J}}\setminus\mathcal{M}$, $\forall j \in \mathcal{J}$,  for all set $\mathcal{H} \subset \mathcal{J}_{-j}$ such that $|\mathcal{H}|=r-2$,
    \, $Y_{.j} \independent \Omega_{.k}| (Y_{.\ell})_{\ell \in \widebar{\{j\}}}.$
    \item \label{hyp6bis} $\forall k,\ell \in \widebar{J}, \quad  k\neq l, \: \Omega_{.k} \independent \Omega_{.\ell} | Y$
    \item \label{hyp8} $\forall j \in \mathcal{J}, \forall m \in \mathcal{M}, \forall \ell \in \widebar{\mathcal{J}}$, for all set $\mathcal{H} \subset \mathcal{J}_{-j}$ such that $|\mathcal{H}|=r-2$, the complete-case coefficients $\mathcal{B}_{j\rightarrow m,\ell,\mathcal{H}[0]}^c$ and $\mathcal{B}_{j\rightarrow m,\ell,\mathcal{H}[k]}^c, k\neq j, k \in \{m,\ell\}\cup\mathcal{H}$ can be consistently estimated. (Here, note that the complete case is when $\Omega_{.m}=1$ and $\Omega_{.\ell}=1$.)
    \item  \label{hyp7} For the variables neither MNAR nor pivot, their means $(\alpha_{k})_{k\in\widebar{\mathcal{J}}\setminus\mathcal{M}}$, variances $(\mathrm{Var}(Y_{.k}))_{k\in\widebar{\mathcal{J}}\setminus\mathcal{M}}$ and covariances $(\mathrm{Cov}(Y_{.k},Y_{.k'}))_{k \neq k'\in\widebar{\mathcal{J}}\setminus\mathcal{M}}$ can be consistently estimated. The covariances between these variables and the pivot variables $(\mathrm{Cov}(Y_{.j},Y_{.k}))_{j \in \mathcal{J}, k\in\widebar{\mathcal{J}}\setminus\mathcal{M}}$ are also consistent. 
\end{enumerate}

\begin{prop}[Covariance between a MNAR variable and a MNAR or not pivot variable]\label{prop:covmiss}
    Consider the PPCA model \eqref{eq:model}. Under Assumptions \ref{hyp2}, \ref{hyp5}, \ref{hyp6} and \ref{hyp6bis}, an estimator of the covariance between a MNAR variable $Y_{.m}$, for $m \in \mathcal{M}$, and a variable $Y_{.\ell}$, for $\ell \in \widebar{\mathcal{J}}\setminus \{ m\}$, can be constructed as follows: choose $j \in \mathcal{J}$ and $r-2$ variable indexes in $\mathcal{J}_{-j}$ and compute:
	\begin{multline}\label{eq:covmissvar}
	\widehat{\mathrm{Cov}}(Y_{.m},Y_{.\ell})=\frac{1}{\hat{K}}\widehat{\mathrm{Var}}(Y_{.j})-\hat{q}^c
	-\sum_{k\in \{m,\ell\}\cup\mathcal{H}} (\hat{\mathcal{B}}^c_{j\rightarrow m,\ell,\mathcal{H}[k]})^2 \widehat{\mathrm{Var}}(Y_{.k})\\
	-\sum_{k<k',k \in \{m,\ell\}\cup\mathcal{H}, k'\in \mathcal{H}}2\hat{\mathcal{B}}^c_{j\rightarrow m,\ell,\mathcal{H}[k]}\hat{\mathcal{B}}^c_{j\rightarrow m,\mathcal{H}[k']}\widehat{\mathrm{Cov}}(Y_{.k},Y_{.k'}),
	\end{multline}
    assuming that $\sigma^2$ tends to zero and with $\hat{K}=2\hat{\mathcal{B}}^c_{j\rightarrow m,\ell,\mathcal{H}[m]}\hat{\mathcal{B}}^c_{j\rightarrow m,\mathcal{H}[\ell]}$ and 
    \begin{multline*}
	\hat{q}^{c}=\left(\widehat{\mathrm{Var}}(Y_{.j})\big| \Omega_{.m}=1,\Omega_{.\ell}=1\right)
	\\
	-\left(\widehat{\mathrm{Cov}}((Y_{.k})_{k \in \widebar{\{j\}}},Y_{.j}) \widehat{\mathrm{Var}}((Y_{.k})_{k \in \widebar{\{j\}}})^{-1} \widehat{\mathrm{Cov}}((Y_{.k})_{k \in \widebar{\{j\}}},Y_{.j})^T \big| \Omega_{.m}=1,\Omega_{.\ell}=1\right),
	\end{multline*}
	given that $K$ estimated by $\hat{K}$ is non zero.
	
    Under the additional Assumptions \ref{hyp3}, \ref{hyp8} and \ref{hyp7}.
	this estimator given in \eqref{eq:covmissvar} is consistent. 
\end{prop}

\begin{proof}
    Let $\mathcal{H}$ bet the set of the $r-2$ variable indexes. One has $\mathcal{H}\subset\mathcal{J}_{-j}$. 
    We use the same strategy as the proof for Proposition \ref{prop:var_formula_general} (paragraph for deriving an equation for the variance). 
    
    To derive a formula for $\mathrm{Cov}(Y_{.m},Y_{.\ell})$, the idea is to express $\mathrm{Var}(Y_{.j})$ from  $(\mathrm{Var}(Y_{.k}))_{k \in \{m,l\}\cup\mathcal{H}}$ and $(\mathrm{Cov}(Y_{.k},Y_{.k'}))_{k \neq k' \in \{m,\ell\}\cup\mathcal{H}}$. 
    
    The law of total variance reads as 
	\begin{equation}\label{eq:var2miss}
	\mathrm{Var}(Y_{.j})=\mathbb{E}[\mathrm{Var}(Y_{.j}|Z)]+\mathrm{Var}(\mathbb{E}[Y_{.j}|Z]),
	\end{equation}
	with $Z=(Y_{.k})_{k \in \widebar{\{j\}}}$.
	
	For the first term in \eqref{eq:var2miss}, one uses
	$$Y_{.j}\independent \Omega_{.m}, \Omega_{.l} | Z.$$
	If $Y_{.m}$ and $Y_{.\ell}$ are both MNAR variables, this conditional independance is obtained using Assumption \ref{hyp2} and \ref{hyp6bis}. Otherwise, if $Y_{.\ell}$ is not a MNAR variable, Assumption \ref{hyp6} and \ref{hyp6bis} lead to the desired result. It implies
	$$\mathrm{Var}(Y_{.j}|Z)=\mathrm{Var}(Y_{.j}|Z,\Omega_{.m}=1,\Omega_{.\ell}=1).$$
	The conditional variance for a Gaussian vector gives
	$$\mathrm{Var}(Y_{.j}|Z)=\mathrm{Var}(Y_{.j})-\mathrm{Cov}(Z,Y_{.j}) \mathrm{Var}(Z)^{-1} \mathrm{Cov}(Z,Y_{.j})^T,$$
	implying that $$\mathrm{Var}(Y_{.j}|Z,\Omega_{.m}=1,\Omega_{.\ell}=1)=\left(\mathrm{Var}(Y_{.j})-\mathrm{Cov}(Z,Y_{.j}) \mathrm{Var}(Z)^{-1} \mathrm{Cov}(Z,Y_{.j})^T\big| \Omega_{.m}=1,\Omega_{.\ell}=1 \right)$$ and then, as deterministic quantity,
	\begin{equation}\label{eq:espvarcondtoy2}
	\mathbb{E}[\mathrm{Var}(Y_{.j}|Z)]=q^c
	\end{equation}

	with
    \begin{multline*}
	q^{c}=\left(\mathrm{Var}(Y_{.j})\big| \Omega_{.m}=1,\Omega_{.\ell}=1\right)
	\\
	-\left(\mathrm{Cov}((Y_{.k})_{k \in \widebar{\{j\}}},Y_{.j}) \mathrm{Var}((Y_{.k})_{k \in \widebar{\{j\}}})^{-1} \mathrm{Cov}((Y_{.k})_{k \in \widebar{\{j\}}},Y_{.j})^T \big| \Omega_{.m}=1,\Omega_{.\ell}=1\right).
	\end{multline*}
	
	For the second term of \eqref{eq:varY2toy}, remark that \ref{hyp2}, \ref{hyp6} and \ref{hyp6bis} implies that
	$$\mathrm{Var}(\mathbb{E}[Y_{.j}|Z])=\mathrm{Var}(\mathbb{E}[Y_{.j}|Z,\Omega_{.m}=1,\Omega_{.\ell}=1]),$$ 
	and
	$$\mathrm{Var}(\mathbb{E}[Y_{.j}|Z,\Omega_{.m}=1,\Omega_{.\ell}=1])=\mathrm{Var}\left(\mathbb{E}\left[\mathcal{B}^c_{j\rightarrow m,\ell,\mathcal{H}[0]}
    +\sum_{k\in \{m,\ell\}\cup\mathcal{H}} \mathcal{B}^c_{j\rightarrow m,\ell,\mathcal{H}[k]} Y_{.k}  + \zeta_j^c\bigg|Z\right]\right),$$
    \textit{i.e.}
    \begin{multline*}
    \mathrm{Var}(\mathbb{E}[Y_{.j}|Z,\Omega_{.m}=1,\Omega_{.\ell}=1])\\=\mathrm{Var}\left(\sum_{k\in \{m,\ell\}\cup\mathcal{H}} \mathcal{B}^c_{j\rightarrow m,\ell,\mathcal{H}[k]} Y_{.k}-\sum_{k\in \{m,\ell\}\cup\mathcal{H}} \mathcal{B}^c_{j\rightarrow m,\ell,\mathcal{H}[k]} \mathbb{E}[\epsilon_{.k}|Z] +\mathcal{B}^c_{j\rightarrow m,\ell,\mathcal{H}[0]}+\mathbb{E}[\epsilon_{.j}]\right)
    \end{multline*}
	One uses the same reasoning as in the proof of Proposition \ref{prop:var_formula_general} (paragraph for deriving an equation for the variance) to get
	\begin{multline}\label{eq:condvarToy2}
	\mathrm{Var}(\mathbb{E}[Y_{.j}|Z,\Omega_{.m}=1,\Omega_{.\ell}=1])=\sum_{k\in \{m,\ell\}\cup\mathcal{H}} (\mathcal{B}^c_{j\rightarrow m,\ell,\mathcal{H}[k]})^2 \mathrm{Var}(Y_{.k})\\
	+\sum_{k<k'\in \{m,\ell\}\cup\mathcal{H}}2\mathcal{B}^c_{j\rightarrow m,\ell,\mathcal{H}[k]}\mathcal{B}^c_{j\rightarrow m,\ell,\mathcal{H}[k']}\mathrm{Cov}(Y_{.k},Y_{.k'}) + o_{\mathrm{covmiss}}(\sigma^2),
	\end{multline}
	where
	\begin{multline}\label{eq:ocovmiss}
	o_{\mathrm{covmiss}}(\sigma^2)=-2\sigma^2\sum_{(k,k') \in \{m,\ell\}\cup\mathcal{H}}\mathcal{B}_{j\rightarrow m,\ell,\mathcal{H}[k]}^c\mathcal{B}_{j\rightarrow m,\ell,\mathcal{H}[k']}^c\sum_{\ell' \in \{m,\ell\}\cup\mathcal{H}}(\Gamma_Z)_{k'\ell'}\mathrm{Cov}(Y_{.k},Y_{.\ell'}) \\
	+\sigma^4 \sum_{k \in \{m,\ell\}\cup\mathcal{H}}(\mathcal{B}_{j\rightarrow m,\ell,\mathcal{H}[k]}^c)^2 \left(\sum_{(k'<\ell') \in \{m,\ell\}\cup\mathcal{H}}(\Gamma_Z)_{kk'}^2\mathrm{Var}(Y_{.k'})
	-2 (\Gamma_Z)_{kk'}(\Gamma_Z)_{k\ell'}\mathrm{Cov}(Y_{.k'},Y_{.\ell'})\right) \\
	-2\sigma^4\sum_{(k<k') \in \{m,\ell\}\cup\mathcal{H}}\mathcal{B}_{j\rightarrow m,\ell,\mathcal{H}[k]}^c\mathcal{B}_{j\rightarrow m,\ell,\mathcal{H}[k']}^c\sum_{(k'',\ell') \in \{m,\ell\}\cup \mathcal{H}}(\Gamma_Z)_{kk''}(\Gamma_Z)_{k'\ell'}\mathrm{Cov}(Y_{.k''},Y_{.\ell'})
    \end{multline}
	
	Combining \eqref{eq:var2miss}, \eqref{eq:espvarcondtoy2} and \eqref{eq:condvarToy2}, one get the following formula for $\mathrm{Cov}(Y_{.m},Y_{.\ell})$,
	\begin{multline*}
	2\mathcal{B}^c_{j\rightarrow m,\ell,\mathcal{H}[m]}\mathcal{B}^c_{j\rightarrow m,\mathcal{H}[\ell]}\mathrm{Cov}(Y_{.m},Y_{.\ell})=\mathrm{Var}(Y_{.j})-q^c
	-\sum_{k\in \{m,\ell\}\cup\mathcal{H}} (\mathcal{B}^c_{j\rightarrow m,\ell,\mathcal{H}[k]})^2 \mathrm{Var}(Y_{.k})\\
	-\sum_{k<k',k \in \{m,\ell\}\cup\mathcal{H}, k'\in \mathcal{H}}2\mathcal{B}^c_{j\rightarrow m,\ell,\mathcal{H}[k]}\mathcal{B}^c_{j\rightarrow m,\mathcal{H}[k']}\mathrm{Cov}(Y_{.k},Y_{.k'})-o_{\mathrm{covmiss}}(\sigma^2)
	\end{multline*}
	
	An estimator of $\mathrm{Cov}(Y_{.m},Y_{.l})$ is then derived as in \eqref{eq:covmissvar}, given that $\sigma^2$ tends to zero and $K=\mathcal{B}^c_{j\rightarrow m,\ell,\mathcal{H}[m]}\mathcal{B}^c_{j\rightarrow m,\mathcal{H}[\ell]}$ is non zero. 
	
	 We use the consistent estimators defined in Proposition \ref{prop:var_formula_general} for $\mathrm{Var}(Y_{.m})$ and $\mathrm{Cov}(Y_{.m},Y_{.k})_{k \in \mathcal{H}}$. If $Y_{.l}$ is also a MNAR variable, Proposition \ref{prop:var_formula_general} is applied for estimating  $\mathrm{Var}(Y_{.l})$ and $\mathrm{Cov}(Y_{.l},Y_{.k})_{k \in \mathcal{H}}$. Otherwise, if $Y_{.l}$ is not a MNAR variable, we use \ref{hyp7}. 
	 
	 Eventually, \ref{hyp3} and \ref{hyp8} leads to the consistency of $\widehat{\mathrm{Cov}}(Y_{.m},Y_{.l})$.
\end{proof}

\subsection{Extension to more general mechanisms for the not MNAR variables}
\label{sec:extensionmecha}

The results of Proposition \ref{prop:mean_formula_general}, \ref{prop:var_formula_general} and \ref{prop:covmiss} can be extended to a more general setting than the one presented in Section \ref{sec:model}. The pivot variables may be assumed to be MCAR (or observed). The variables which are neither MNAR nor pivot may be observed or satisfying
\begin{equation}\label{eq:mechanismother}
\forall \ell \in \widebar{\mathcal{J}}\setminus\mathcal{M}, \forall i \in \{1,\dots, n\},  \quad  \mathbb{P}(\Omega_{i\ell}=1|Y_{i.})=\mathbb{P}(\Omega_{i\ell}=1|(Y_{ik})_{k \in \widebar{\mathcal{J}}\setminus\{\ell\}\cup\mathcal{M}}),
\end{equation}
\textit{i.e.} they are MCAR or MAR but their missing-data mechanisms may not depend on the pivot variables.

The proofs are similar and not presented here for the sake of brevity. 

Note that the main difference is that the complete case has to be extended.
For instance, for $j\in \mathcal{J}$ and $k \in \mathcal{J}_{-j}$, the coefficients standing respectively for the intercept and the effects of $Y_{.j}$ on $(Y_{.m}, (Y_{.j'})_{j' \in \mathcal{J}_{-j}})$ in the complete case, \textit{i.e.} when $\Omega_{.m}=1, (\Omega_{j}=1)_{j\in \mathcal{J}}
$ are in this general setting defined as follows
$$
    \left(Y_{.j}|\Omega_{.m}= 1,(\Omega_{j}=1)_{j\in \mathcal{J}}\right) := \mathcal{B}^c_{j\rightarrow m,\mathcal{J}_{-j}[0]}
    +\sum_{j'\in \mathcal{J}_{-j}} \mathcal{B}^c_{j\rightarrow m,\mathcal{J}_{-j}[j']} Y_{.j'} + \mathcal{B}^c_{j\rightarrow m,\mathcal{J}_{-j}[m]} Y_{.m} + \zeta^c,$$
with $\zeta^c=-\sum_{j'\in\mathcal{J}_{-j}} \mathcal{B}^c_{j\rightarrow m,\mathcal{J}_{-j}[j']} \epsilon_{.j'} - \mathcal{B}^c_{j\rightarrow m,\mathcal{J}_{-j}[m]} \epsilon_{.m}+ \epsilon_{.j}.$

\section{Other numerical experiments}
\label{sec:othernumexp}

\paragraph{Robustness to noise.}
Considering the same setting as in Section \ref{sec:simu_synthdata} ($n=1000$, $p=10$, $r=2$ and seven self-masked MNAR variables), the methods are tried for different noise levels $\sigma^2\in\{0.1,0.3,0.5,0.7,1\}$. The results are presented for one missing variable and for all the other ones, the results are similar. 
In Figure \ref{fig:MeanVarianceNoise}, Algorithm \ref{alg:imputation} is the only method that does not give a biased estimate of the mean and the variance regardless of the noise level. In Figure \ref{fig:CovariancesNoise}, despite a larger bias in the estimation of the covariance between a missing variable and a pivot one as the noise level increases, Algorithm \ref{alg:imputation} outperforms all the other methods, regarding the estimation of the covariance between two missing variables. Note that the formula for the estimate of the covariance between two missing variables relies on the one for the estimate of the variance, but both differ from the one used for the covarance estimation between a missing variable and a pivot one.
As expected, in Figures \ref{fig:MSECorrNoise}, estimation deteriorates as the data gets noisier and then the loading matrix estimation and the prediction error get closer to the results of mean imputation. In term of prediction error, the proposed method yet remains competitive in regards of the approaches \ref{methEMMAR} and \ref{methSoft}. Overall, when the noise level increases, the exogeneity will be worse and that ignoring it in practice can be made to the detriment of performance. 

\begin{figure}[H]
\centering
\includegraphics[width=1\textwidth]{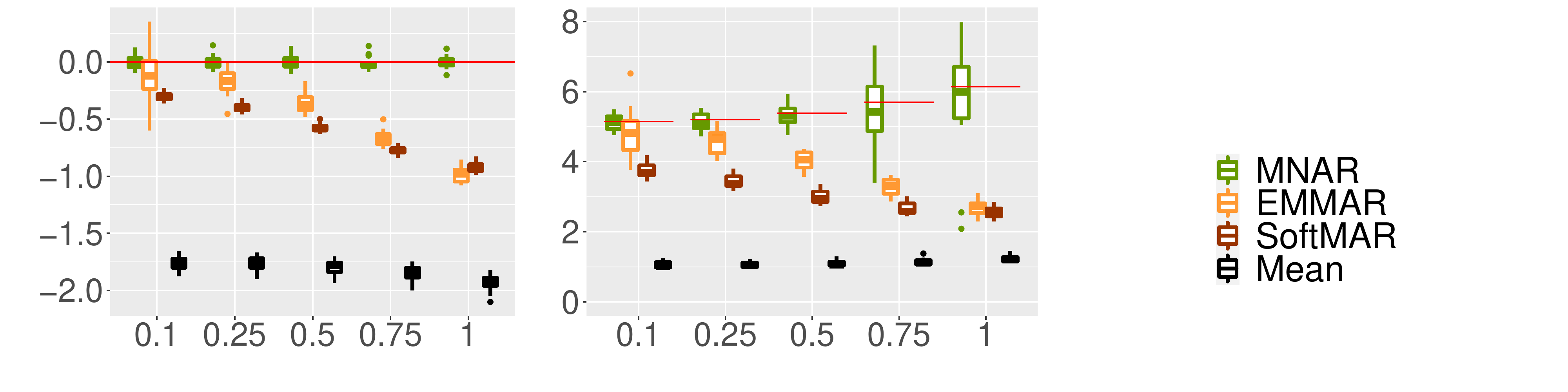}
\caption{\label{fig:MeanVarianceNoise} Mean estimation (left graphic) and variance estimation (right graphic) of one missing variable for different values of the level of noise when $r = 2$, $n = 1000$, $p = 10$ and seven variables are MNAR. True values to be estimated are indicated by red lines.}
\end{figure}

\begin{figure}[H]
\centering
\includegraphics[width=1\textwidth]{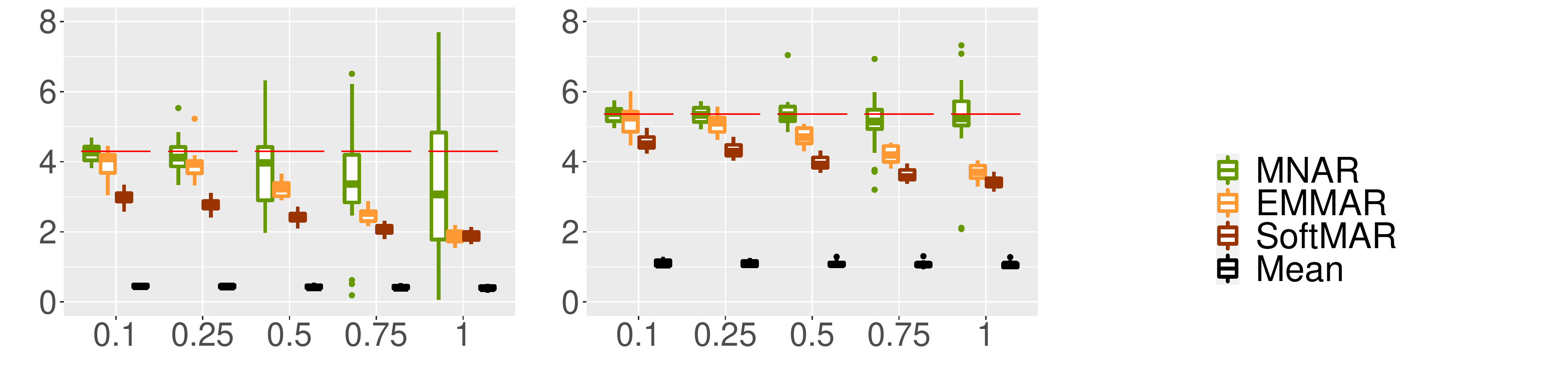}
\caption{\label{fig:CovariancesNoise} Covariance estimation beetween a missing variable and a pivot one (left graphic) and two missing variables (right graphic) for different values of the level of noise when $r = 2$, $n = 1000$, $p = 10$ and seven variables are MNAR. True values to be estimated are indicated by red lines.}
\end{figure}

\begin{figure}[H]
\centering
\includegraphics[width=1\textwidth]{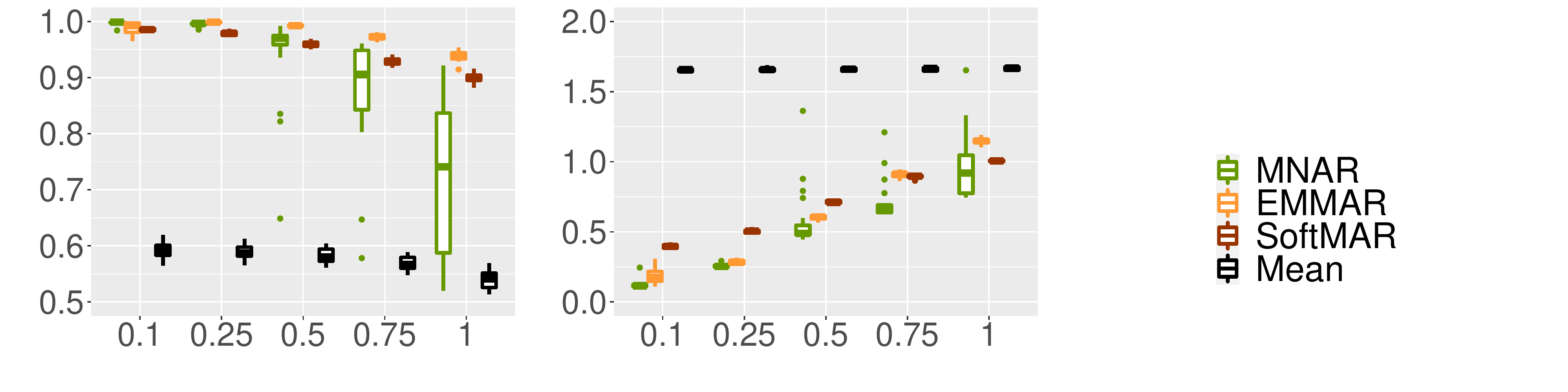}
\caption{\label{fig:MSECorrNoise} RV coefficients for the loading matrix (left graphic) and prediction error (right graphic) for different values of the level of noise when $r = 2$, $n = 1000$, $p = 10$ and seven variables are MNAR.}
\end{figure}

\paragraph{Varying the percentage of missing values}
Considering the same setting as in Section \ref{sec:simu_synthdata} ($n=1000$, $p=10$, $r=2$, $\sigma=0.1$ and seven self-masked MNAR variables), the methods are tried for different percentages of missing values (10\%, 30\%, 50\%). The results are presented in Figure \ref{fig:PercNA}. As expected, all the methods deteriorate with an increasing percentage of missing values but our method is stable. 

\begin{figure}[H]
\centering
\includegraphics[width=1\textwidth]{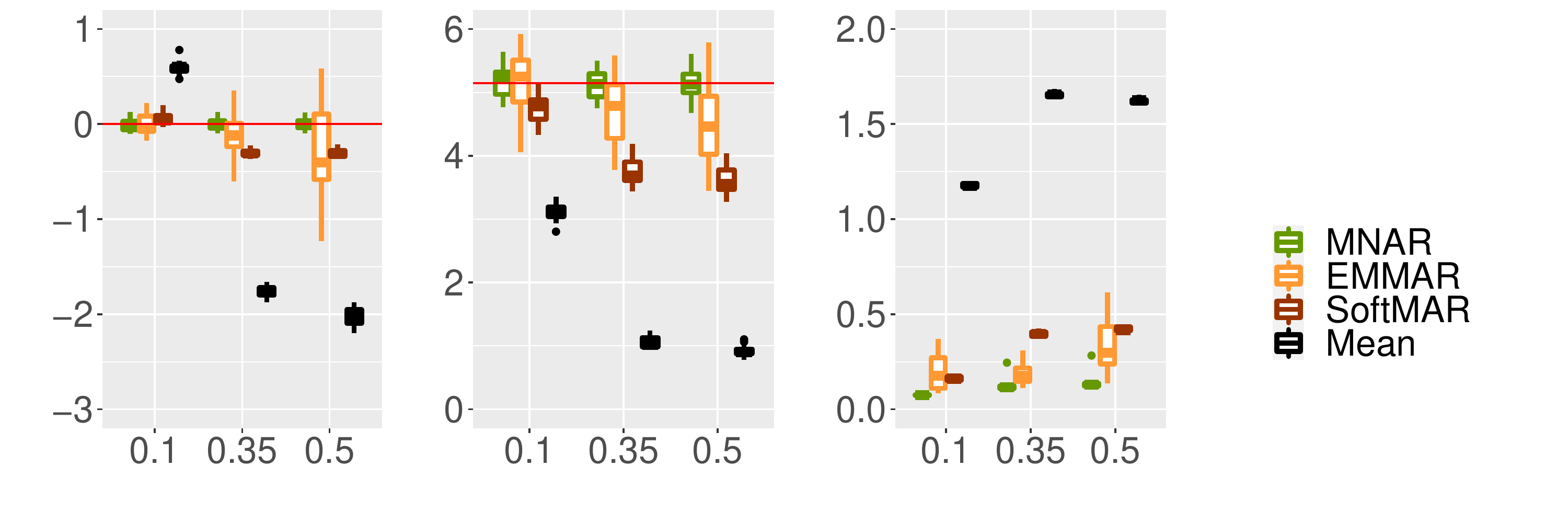}
\caption{\label{fig:PercNA} Mean estimation (left graphic), variance estimation (middle graphic) and prediction error (right graphic) for different percentages of missing values when $r = 2$, $n = 1000$, $p = 10$ and seven variables are MNAR.}
\end{figure}

\paragraph{Misspecification to the PPCA model.}
The fixed effects model is considered, \textit{i.e.} the data $Y \in \mathbb{R}^{n \times p}$ is generated as a sum of a low-rank matrix $\Theta \in \mathbb{R}^{n \times p}$ (the rank $r$ of $\Theta$ satisfies $r < \min\{n, p\}$) and a Gaussian noise matrix, \textit{i.e.}
\begin{equation}\label{eq:fixedeffect}
	Y = \Theta + \epsilon.
\end{equation}
The data matrix of size $n = 200$ and $p = 10$ is generated under the fixed effects model as \eqref{eq:fixedeffect} with a rank $r = 2$ (for $\Theta$) and a noise level $\sigma = 0.1$. Missing values are introduced on seven MNAR variables according to a self-masked MNAR mechanism, resulting in 35\% missing data in the whole matrix. Figure \ref{fig:fixedeffect} shows that estimators for the mean and the variance given by Algorithm \ref{alg:imputation} have a larger variance than those given by the parametric Method \ref{methparam}. But surprisingly, Algorithm \ref{alg:imputation} provides less biased estimates of the mean and the variance, than Method \ref{methparam}, while precisely dedicated to this specific setting. 
Note that  with Method \ref{methparam} designed for fixed effects models, the variance is slightly under-estimated, which is expected as the method imputes missing entries with $\hat \Theta$ and consequently the variability in the imputed data is smaller than the one in the observed data. 

As for the imputation performance, Figure \ref{fig:fixedeffect} also shows that Algorithm \ref{alg:imputation} gives similar results as Method \ref{methparam}, which explicitly models the MNAR mechanism. In addition, despite the model misspecification, it also remains competitive compared to  Method  \ref{methSoft}, which ignores the MNAR mechanism but is specially designed to handle fixed effect models. 

\begin{figure}[H]
\centering
\includegraphics[width=1\textwidth]{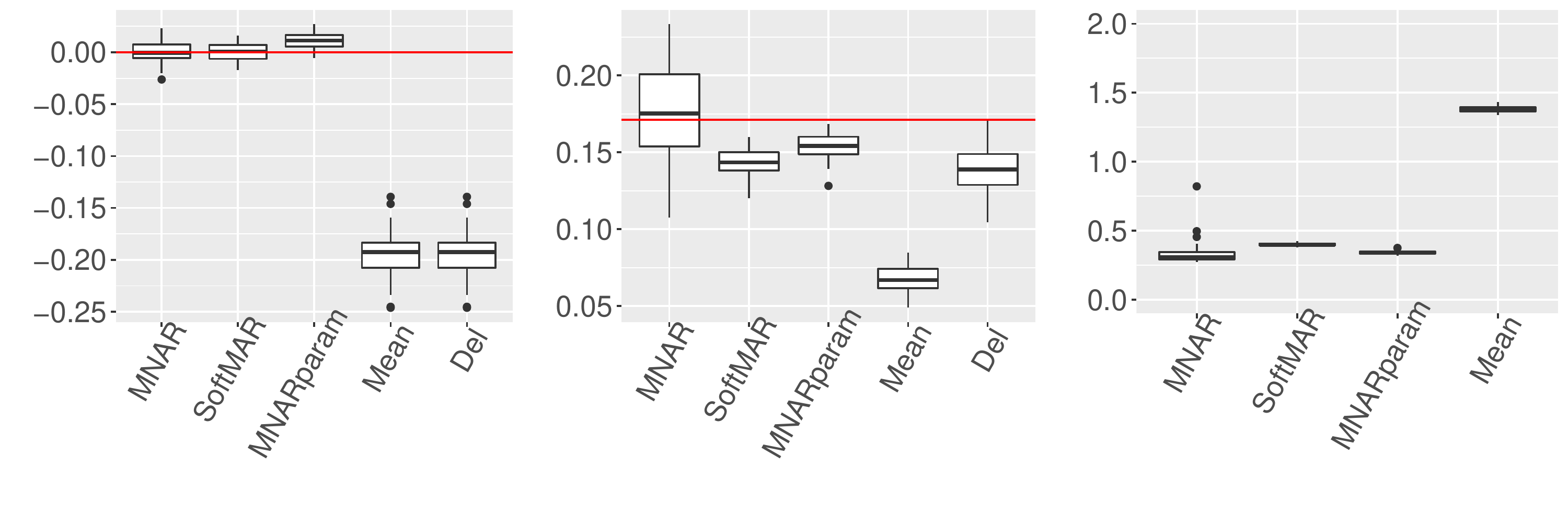}
\vspace{-1.2cm}
\caption{\label{fig:fixedeffect} Mean estimation (left graphic), variance estimation (middle graphic) of one missing variable and prediction error (right graphic) when data are generated under the fixed effects model in \eqref{eq:fixedeffect}, $r = 2$, $n = 200$, $p = 10$ and seven variables are MNAR. True values to be estimated are indicated by red lines.}
\end{figure}

\paragraph{Misspecification to the rank}
The misspecification to the parameter $r$ has been evaluated: under a model generated with $r = 3$ latent variables ($n = 1000$, $p = 20$, $\sigma=0.8$ and ten MNAR self-masked variables), the rank is either underestimated, well estimated or overestimated by giving to Algorithm \ref{alg:imputation} the information that $r=2$, $r=3$ or $r=4$. Both estimation of the loading matrix and prediction error are shown in Figure \ref{fig:missperank}. The results for an underestimated ($r=2$) or overestimated ($r=4$) rank are comparable to the case where the accurate rank is considered instead ($r=3$), showing a stability of Algorithm \ref{alg:imputation} to rank misspecification.

\begin{figure}[H]
\centering
\includegraphics[width=1\textwidth]{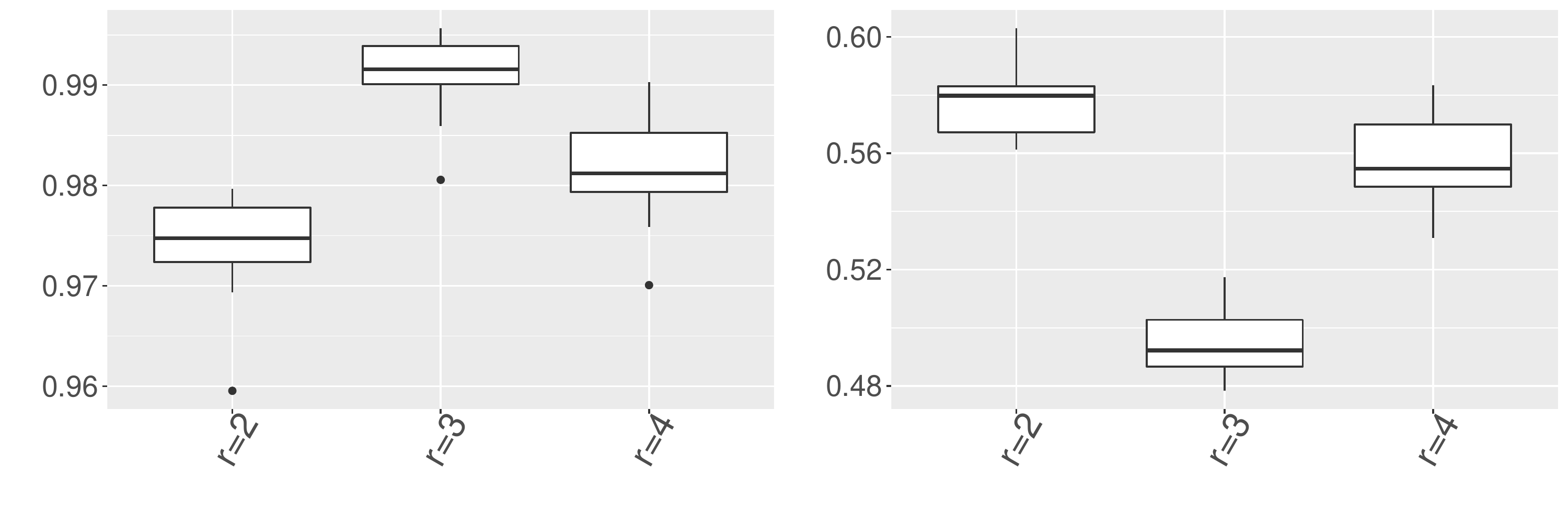}
\vspace{-0.8cm}
\caption{\label{fig:missperank} RV coefficients for the loading matrix (left graphic) and prediction error (right graphic) when $r=3$, $n=1000$, $p=20$ and ten variables are MNAR for different cases where the rank is either underestimated, well estimated or overestimated.}
\end{figure}

\paragraph{General MNAR mechanism}
We consider the setting $n=1000$, $p=20$ and $\sigma=0.8$. Here, missing values are introduced on ten variables $(Y_{.k})_{k \in [1:10]}$ using a more general MNAR mechanism (see \eqref{eq:MNARmechanism}) than the self-masked one. In particular, the MNAR mechanism we consider is defined as follows, 
\begin{equation}\label{eq:MNARsimu}
\forall m \in [1:10], \forall i \in \{1,\dots,n\}, \:  \mathbb{P}(\Omega_{im}=1|Y_{i.})=\mathbb{P}(\Omega_{im}=1|Y_{im}, Y_{ik}, Y_{i\ell}),
\end{equation}
where $k$ and $\ell$ are indexes of MNAR variables randomly chosen such that $k\neq\ell \in [1:10]\setminus\{m\}$.
In Figure \ref{fig:MNARgeneral}, Algorithm \ref{alg:imputation} provides the best estimators of the mean and the variance (in term of bias) and the smallest prediction error. 

\begin{figure}[H]
\centering
\includegraphics[width=1\textwidth]{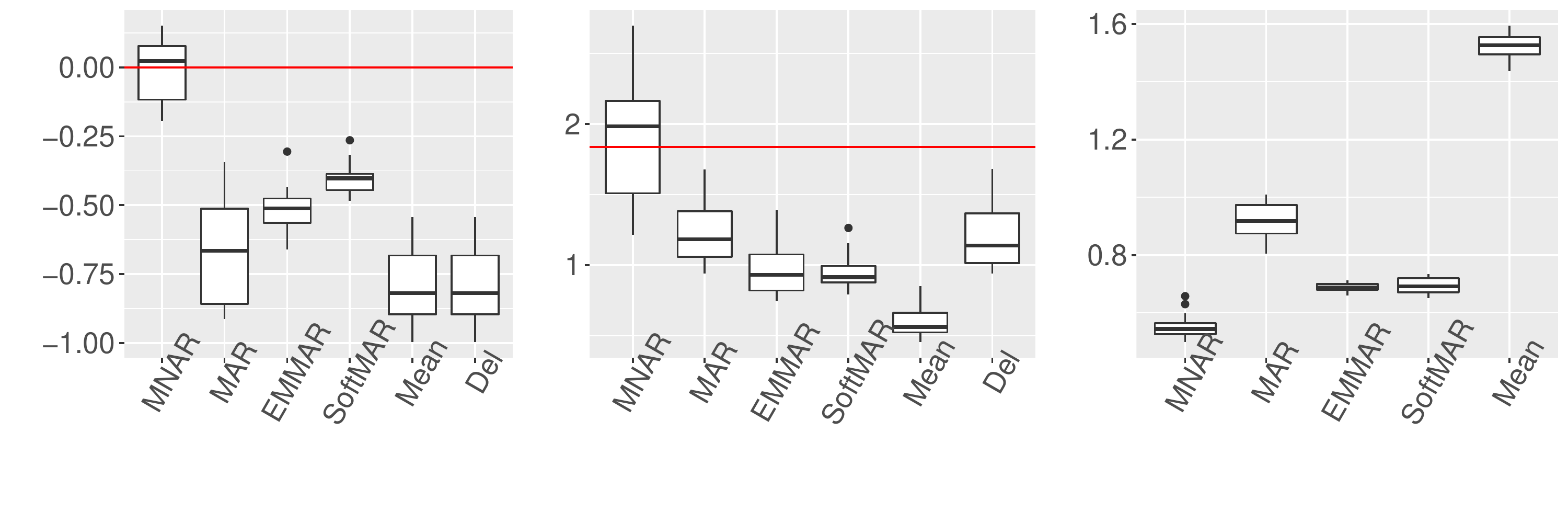}
\vspace{-0.8cm}
\caption{\label{fig:MNARgeneral} Mean estimation (left graphic), variance estimation (middle graphic) of one missing variable and prediction error (right graphic) when $r=2$, $n=1000$, $p=20$ and ten variables are MNAR as in \eqref{eq:MNARsimu}.}
\end{figure}

\paragraph{Higher dimension and variation of the rank}

The performance of the different methods for higher dimension is assessed. A data matrix of size $n = 1000$ and $p = 50$ is generated from two latent variables ($r = 2$) and with a noise level $\sigma=1$. Missing values are introduced on twenty variables according to a self-masked MNAR mechanism, leading to 20\% of missing values in total. Without loss of generality, the results are presented for one missing variable. 
Method \ref{methparam} has been discarded, as its computational time is too high for this  setting.

\begin{figure}[H]
\centering
\includegraphics[width=1\textwidth]{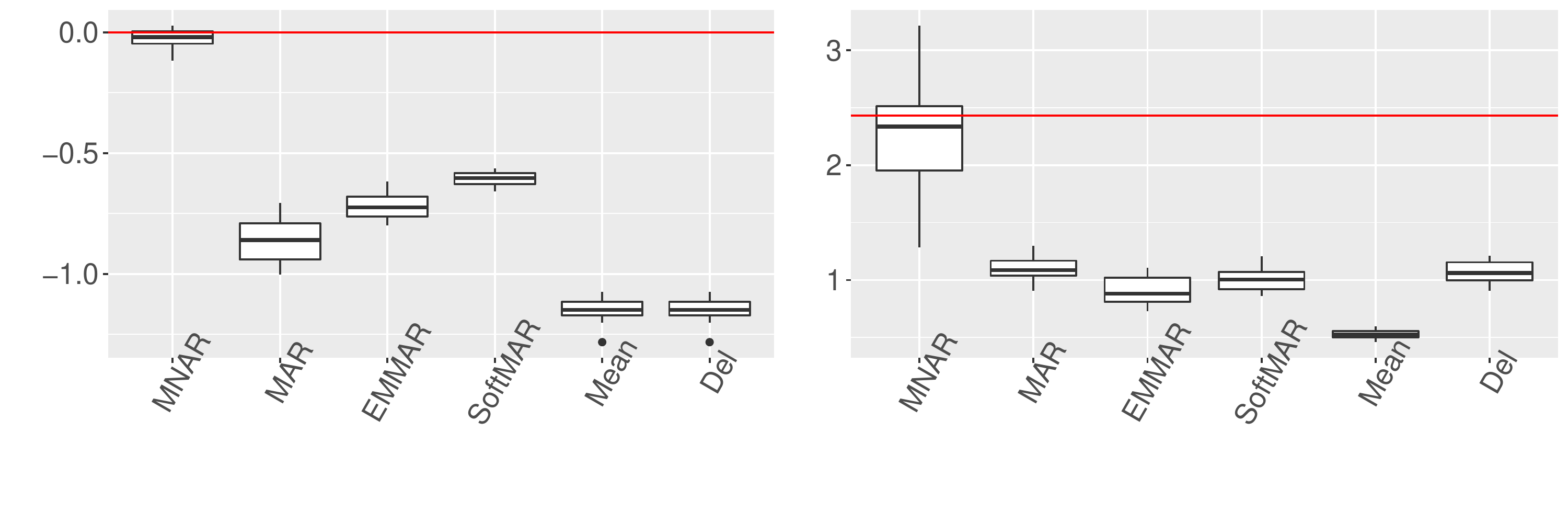}
\vspace{-1.2cm}
\caption{\label{fig:MeanVarianceHighDimr2} Mean estimation (left graphic) and variance estimation (right graphic) of one missing variable when $r = 2$, $n = 1000$, $p = 50$ and twenty variables are MNAR. True values to be estimated are indicated by red lines.}
\end{figure}

\begin{figure}[H]
\centering
\includegraphics[width=1\textwidth]{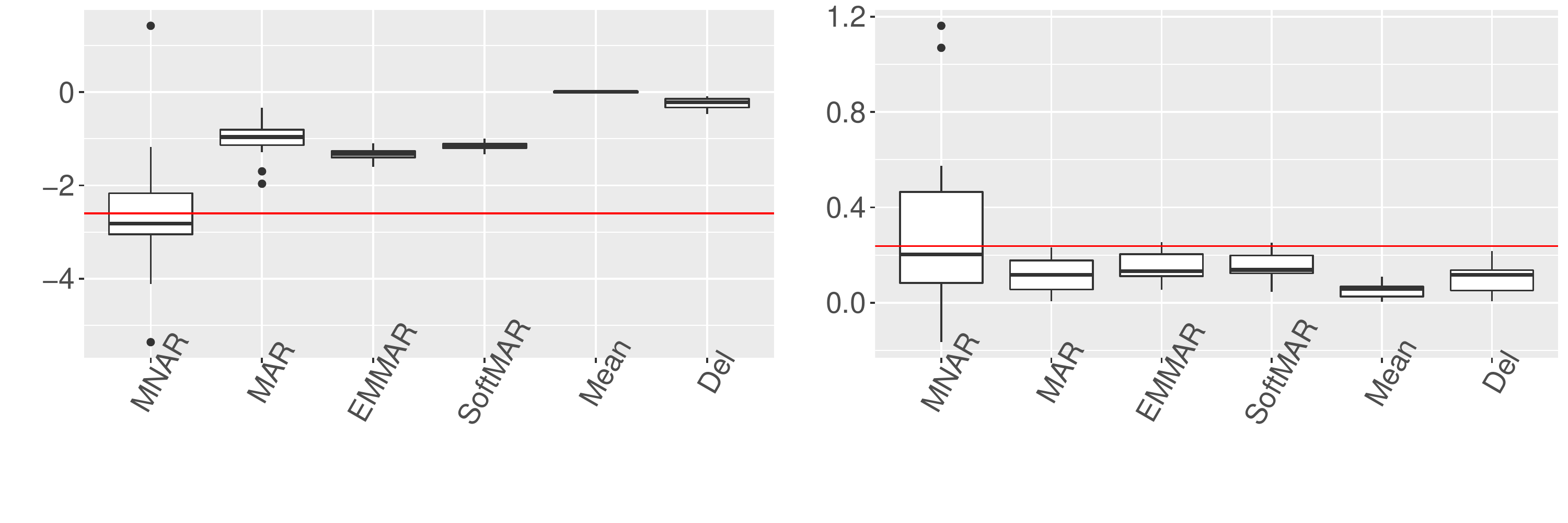}
\vspace{-1.2cm}
\caption{\label{fig:CovariancesHighDimr2} Covariance estimation beetween two missing variable (left graphic) and a missing variable and a pivot one (right graphic) when $r = 2$, $n = 200$, $p = 10$ and seven variables are MNAR. True values to be estimated are indicated by red lines.}
\end{figure}

\begin{figure}[H]
\centering
\includegraphics[width=1\textwidth]{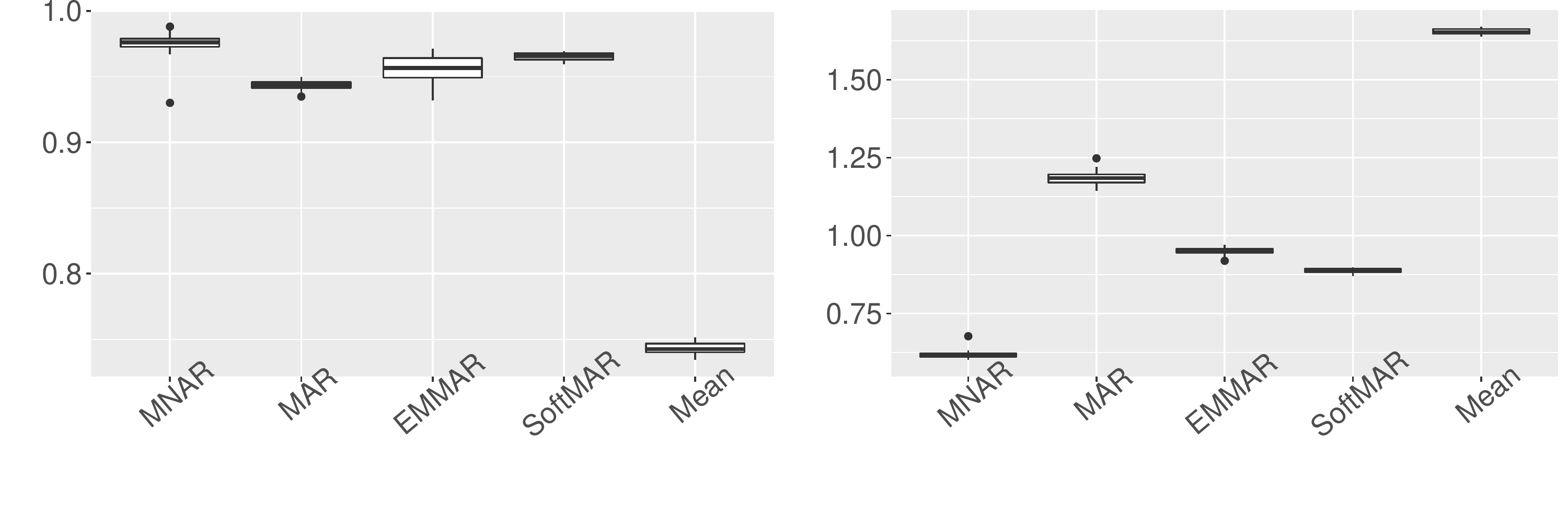}
\vspace{-1.2cm}
\caption{\label{fig:MSECorrHighDimr2} RV coefficients for the loading matrix (left graphic) and prediction error (right graphic) when $r = 2$, $n = 1000$, $p = 50$ and twenty variables are MNAR.}
\end{figure}

In Figure \ref{fig:MeanVarianceHighDimr2}, as for the estimated mean and variance, Methods  \ref{methMAR}, \ref{methEMMAR} and \ref{methSoft} suffers from a large bias, whereas Algorithm \ref{alg:imputation} gives unbiased estimators. The same comments can be done for the estimation of the covariance between two missing values in Figure \ref{fig:CovariancesHighDimr2}. As for  the covariance estimation between a missing variable and a pivot one (Figure \ref{fig:CovariancesHighDimr2}, Algorithm \ref{alg:imputation} suffers from a variability, which can be due to the fact that in this higher dimension setting, not all the possible combinations of pivot variables are considered. Indeed, instead of taking the set of pivot variables as all the not MNAR variables \textit{i.e.} $\mathcal{J}=\widebar{\mathcal{M}}$, we choose $\mathcal{J} \subset \widebar{\mathcal{M}}$ such that $|\mathcal{J}|=10$. For the mean, 270 combinations of the pivot variables are aggregated over 870 possible combinations if $\mathcal{J}=\widebar{\mathcal{M}}$.

Despite this dispersed estimator of the covariance between a MNAR variable and a pivot one, Algorithm \ref{alg:imputation} gives in Figure \ref{fig:MSECorrHighDimr2} a high RV coefficient, by improving Methods \ref{methMAR}, \ref{methSoft} and \ref{methEMMAR}. 
Concerning the imputation performance, Algorithm \ref{alg:imputation} strongly improves Methods \ref{methEMMAR} and \ref{methSoft}.

For the same dimension setting ($n=1000$, $p=50$) and the same noise level ($\sigma=1$), we vary the rank to $r=5$. Similarly as before, missing values are introduced on twenty variables according to a self-masked MNAR mechanism, leading to 20\% of missing values in total. In Figure \ref{fig:MeanVarianceHighDimr5}, for the mean and the variable estimations, Algorithm \ref{alg:imputation} gives unbiased estimators. In Figure \ref{fig:CovariancesHighDimr5}, the covariance between a missing variable and a pivot one estimated by Algorithm \ref{alg:imputation} is biased but still less than the other methods. In addition, the covariance between two missing variables is unbiased but suffers from a high variability. Note that once again we have chosen $\mathcal{J}\subset \mathcal{M}$ such that $|\mathcal{J}|=10$. For the mean, 1260 combinations of the pivot variables are aggregated over 712530 possible combinations if $\mathcal{J}=\widebar{\mathcal{M}}$. In Figure \ref{fig:MSECorrHighDimr5}, despite such results for the covariance estimators, Algorithm \ref{alg:imputation} gives a similar RV coefficient than Methods \ref{methEMMAR} and \ref{methSoft} but strongly improves all the methods in term of prediction error.

\begin{figure}[H]
\centering
\includegraphics[width=1\textwidth]{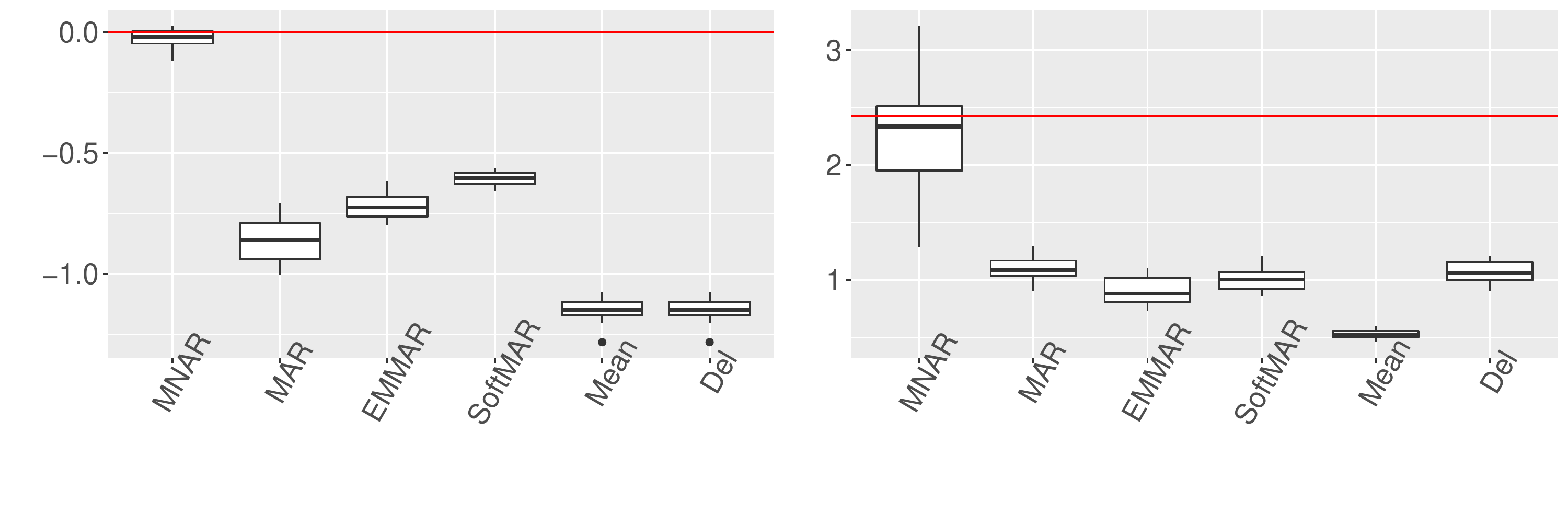}
\vspace{-1.2cm}
\caption{\label{fig:MeanVarianceHighDimr5} Mean estimation (left graphic) and variance estimation (right graphic) of one missing variable when $r = 5$, $n = 1000$, $p = 50$ and twenty variables are MNAR. True values to be estimated are indicated by red lines.}
\end{figure}

\begin{figure}[H]
\centering
\includegraphics[width=1\textwidth]{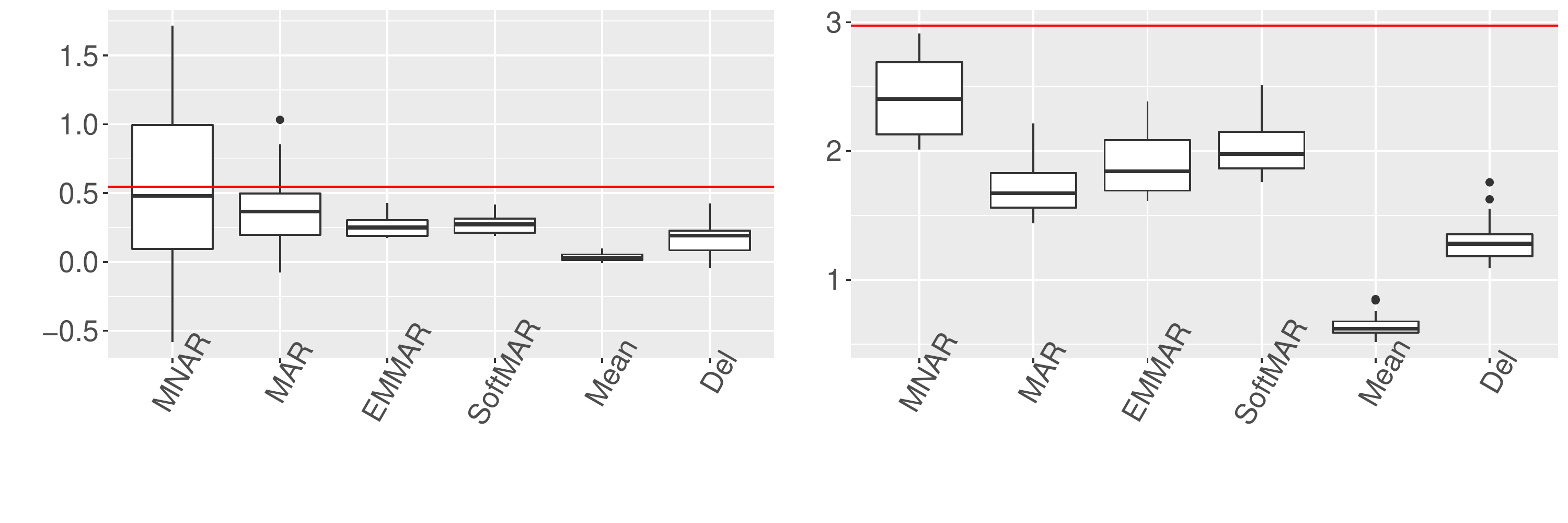}
\vspace{-1.2cm}
\caption{\label{fig:CovariancesHighDimr5} Covariance estimation beetween two missing variable (left graphic) and a missing variable and a pivot one (right graphic) when $r = 5$, $n = 1000$, $p = 50$ and twenty variables are MNAR. True values to be estimated are indicated by red lines.}
\end{figure}

\begin{figure}[H]
\centering
\includegraphics[width=1\textwidth]{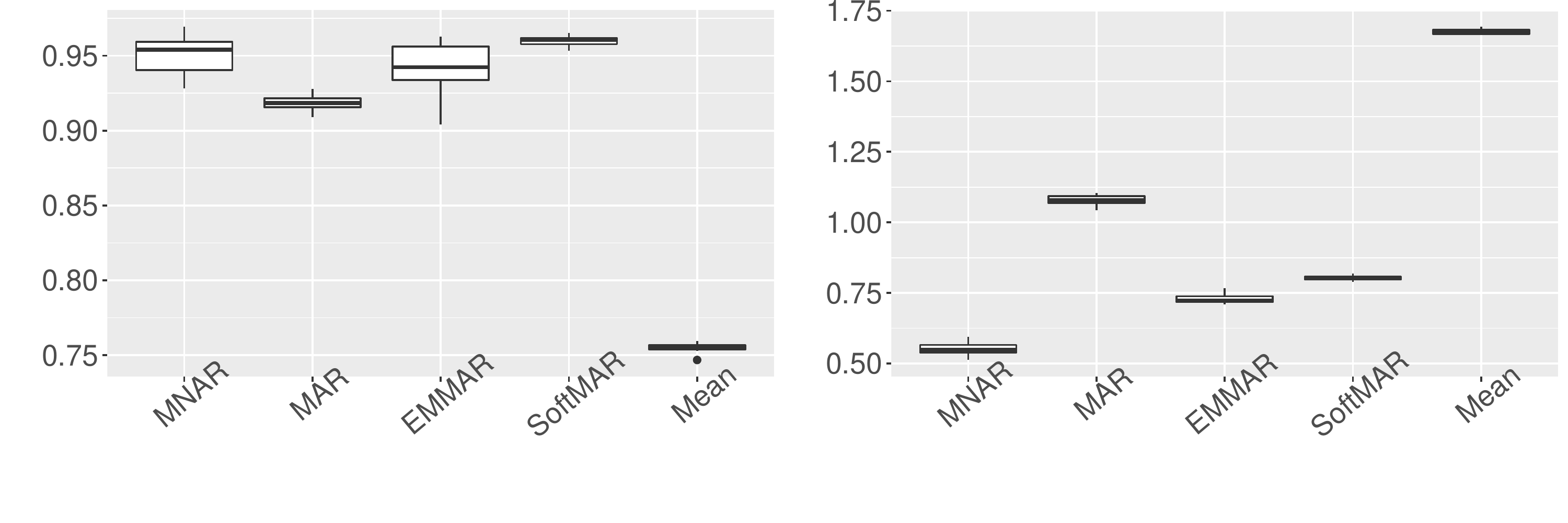}
\vspace{-1.2cm}
\caption{\label{fig:MSECorrHighDimr5} RV coefficients for the loading matrix (left graphic) and prediction error (right graphic) when $r = 5$, $n = 1000$, $p = 50$ and twenty variables are MNAR.}
\end{figure}

\section{Computation time}
\label{sec:comp_time}

Table \ref{tab:comp_time} gathers computation times of the different methods, for both settings considered in Sections \ref{sec:simu} and \ref{sec:othernumexp}. 

\begin{table}[H]
	\centering
	\begin{tabular}{|l|l|l|}
		\hline
		\begin{tabular}[c]{@{}l@{}}Method    \end{tabular} &                          
		\begin{tabular}[c]{@{}l@{}}$r=2, p=10, n=1000$\\ 35\% MNAR values \\ in 7 variables   \end{tabular} &
		\begin{tabular}[c]{@{}l@{}}$r=5, p=50, n=1000$\\ 20\% MNAR values \\  in 20 variables\end{tabular}  
		\\ \hline       
		MNAR algebraic
		&
		0,1 s  &
		11 min 48 s (1260 aggregations) \\ \hline
		SoftMAR                  & 5,5 s  &  28 s \\ \hline
		EMMAR                    &  50,8 s &  2 min 9 s  \\ \hline
		Param                    &  5 h 15 min  &  not evaluated  \\ \hline
	\end{tabular}
	\vspace{0.2cm}
	\caption{\label{tab:comp_time}Computation time for simulations in Sections \ref{sec:simu} and Appendix \ref{sec:othernumexp}. The process time is obtained for a computer with a processor Intel Core i5 of 2,3 GHz.}
\end{table}

\section{Details on the variables in TraumaBase$^{\mbox{\normalsize{\textregistered}}}$}
\label{sec:appdata}

	A description of the variables which are used in Section \ref{sec:simu_realdata} is given. The indications given in parentheses ph (pre-hospital) and h (hospital) mean that the measures have been taken before the arrival at the hospital and at the hospital.
	\begin{itemize}[label=\textbullet]
		\item \textit{SBP.ph}, \textit{DBP.ph}, \textit{HR.ph}: systolic and diastolic arterial pressure and heart rate during pre-hospital phase. (ph)
		\item \textit{HemoCue.init}: prehospital capillary hemoglobin concentration. (ph)
		\item \textit{SpO2.min}: peripheral oxygen saturation, measured by pulse oxymetry, to estimate oxygen content in the blood. (ph)
		\item \textit{Cristalloid.volume}: total amount of prehospital administered cristalloid fluid resuscitation (volume expansion). (ph)
		\item \textit{Shock.index.ph}: ratio of heart rate and systolic arterial pressure during pre-hospital phase. (ph)
		\item \textit{Delta.shock.index}: Difference of shock index between arrival at the hospital and arrival on the scene. (h)
		\item \textit{Delta.hemoCue}: Difference of hemoglobin level between arrival at the hospital and arrival on the scene. (h)
	\end{itemize}

\section{Graphical approach}
\label{sec:graph}

\subsection{Preliminaries}

Lemmas of \citet{mohan2018estimation} are used to construct some estimators of the mean, variance and covariances for a MNAR variable based on a graphical approach. 

\begin{lem}[Lemma 2 \citep{mohan2018estimation}]\label{lem:coefflinear}
	Let us consider the m-graph $G$. The coefficient of the linear regression of $Y_{.j}$ on $Y_{.k}, k \neq j$, denoted as $\beta_{j\rightarrow k, k\neq j}$ is recoverable (\textit{i.e.} they are consistent in the complete-case analysis)  if $Y_{.j} \independent \Omega | Y_{.k}, k \neq j$ and one has
	$$\beta_{j\rightarrow k, k\neq j}=\beta_{j\rightarrow k, k\neq j}^c.$$
\end{lem}

\begin{lem}[Lemma 1] \citep{mohan2018estimation}]\label{lem:graphicalcov}(Graphical approach for computing the covariance) Let $G$ be a m-graph with $k$ unblocked paths $p_1,\dots,p_k$ between two variables $Y_{.\tau}$ and $Y_{.\delta}$. Let $A_{p_i}$ be the ancestor of all notes on path $p_i$. Let the number of notes on $p_i$ be $n_{p_i}$. One can derive that
	$$\mathrm{Cov}(Y_{.\tau},Y_{.\delta})=\sum_{i=1}^{k} \mathrm{Var}(A_{p_i}) \prod_{j=1}^{n_{p_i}-1} \alpha_j^{p_i},$$
	where $\prod_{j=1}^{n_{p_i}-1} \alpha_j^{p_i}$ is the product of all causal parameters on path $p_i$.
	
\end{lem}

In addition, let us recall the basic formula,
\begin{equation}\label{eq:basicformula}
	\beta_{Y\rightarrow X}=\frac{\textrm{Cov}(X,Y)}{\textrm{Var}(X)},
\end{equation}
where $Y$ and $X$ are two variables of a linear model. 

\subsection{Estimation of the mean, variance and covariances of the MNAR variables}
\label{sec:estimationgraphical}

The graphical approach to construct an estimator of $\alpha_1$  is based on the transformation illustrated in Figure \ref{fig:proofToy} of the graphical model of PPCA as structural causal graphs, whose context is introduced in \citep{pearl2003causality}. This latter framework allows to directly apply  the results of \citet{mohan2018estimation} who consider the associated (linear) structural causal equations under the exogeneity assumption with MNAR missing values for one variable. 

For the sake of brevity, the results are presented for the toy example in Section \ref{sec:toyexample} where $p=3$, $r=2$ and $Y_{.1}$ is self-masked MNAR and the other variables are observed.

Then, one can associate to Figure \ref{fig:proofToy} (bottom right graph) the structural equation model detailled in the following lemma. 
\begin{lem}\label{lem:PCAlineargraph_toy}
	Assuming $\mathbb{E}[\epsilon_{.2}|Y_{.1},Y_{.3}]=0$, the structural equation model associated with the bottom right graph in Figure \ref{fig:proofToy} is
	\begin{equation}\label{eq:graphtoy}
	Y_{.2}=\beta_{2 \rightarrow 1,3 [0]}+\beta_{2 \rightarrow 1,3 [1]}Y_{.1}+\beta_{2 \rightarrow 1,3 [3]}Y_{.3}+\epsilon_{.2},
	\end{equation}
	where $\beta_{2 \rightarrow 1,3 [0]}$, $\beta_{2 \rightarrow 1,3 [1]}$ and $\beta_{2 \rightarrow 1,3 [3]}$ are the intercept and the coefficients of the linear regression of $Y_{.2}$ on $Y_{.1}$ and $Y_{.3}$.
\end{lem}

Using Equation \eqref{eq:graphtoy} and Lemma \ref{lem:coefflinear}, we apply the results of  \citet{mohan2018estimation} to get an estimator for the mean of the MNAR variable. 
\begin{prop}[Mean estimator for the graphical approach]\label{prop:meantoygraph}
	Under the equation  \eqref{eq:graphtoy}, assuming \ref{hyp1} and $\beta_{2\rightarrow 1.3}^{c} \neq 0$, one can construct an estimator of the mean $\alpha_1$ of the MNAR variable $Y_{.1}$ as follows
	\begin{equation}\label{eq:expectationtoyexgraph}
	\hat{\alpha}_1:=\frac{\hat{\alpha}_2-\hat{\beta}_{2 \rightarrow 1,3 [0]}^{c}-\hat{\beta}_{2 \rightarrow 1,3 [3]}^{c}\hat{\alpha}_3}{\hat{\beta}_{2 \rightarrow 1,3 [1]}^{c}},
	\end{equation}
	where $\hat{\beta}_{2 \rightarrow 1,3 [0]}^c$, $\hat{\beta}_{2 \rightarrow 1,3 [1]}^c$ and $\hat{\beta}_{2 \rightarrow 1,3 [3]}^c$ denote some estimators of $\beta_{2 \rightarrow 1,3 [0]}^c$, $\beta_{2 \rightarrow 1,3 [1]}^c$ and $\beta_{2 \rightarrow 1,3 [3]}^c$ given in Lemma \ref{lem:PCAlineargraph_toy}. 
	This estimator is consistent under additional Assumption \ref{hyp4}.
\end{prop}
\begin{proof}
To derive some estimator of the mean, we want to obtain the following formula
    \begin{equation}\label{eq:expectationtoyex2}
	\alpha_1=\frac{\alpha_2-\beta_{2\rightarrow 1,3[0]}^{c}-\beta_{2 \rightarrow 1,3[3]}^{c}\alpha_3}{\beta_{2\rightarrow 1,3[1]}^{c}}.
	\end{equation}
	Indeed, one has:
	\begin{align*}
	\mathbb{E}[Y_{.2}]&=\mathbb{E}[\mathbb{E}[Y_{.2}|Y_{.1},Y_{.3}] \\
	&=\mathbb{E}[\mathbb{E}[Y_{.2}|Y_{.1},Y_{.3},\Omega_{.1}=1]] &\textrm{(by using \ref{hyp1})} \\
	&=\mathbb{E}[\mathbb{E}[\beta_{2 \rightarrow 1,3 [0]}^c+\beta_{2 \rightarrow 1,3[1]}^cY_{.1}+\beta_{2 \rightarrow 3,1[3]}^cY_{.3}+ \epsilon_{.2}|Y_{.1},Y_{.3}]]  \\
	&=\beta_{2 \rightarrow 1,3 [0]}^c+\beta_{2 \rightarrow 1,3[1]}^c \mathbb{E}[Y_{.1}] + \beta_{2 \rightarrow 3,1[3]}^c \mathbb{E}[Y_{.3}],
	\end{align*}
	which leads to the desired Equation \eqref{eq:expectationtoyex2}, provided that $\beta_{2\rightarrow 1,3[1]}^{c} \neq 0$.
A natural estimator fo $\alpha_1$ is then given by \eqref{eq:expectationtoyexgraph}. It is consistent given that all the quantities involved are consistent, by using \ref{hyp4} (for the consistency of $\hat{\alpha}_2$ and $\hat{\alpha}_3$) and Lemma \ref{lem:coefflinear} (for the consistency of the coefficients $\hat{\beta}_{2 \rightarrow 1,3 [0]}^c$, $\hat{\beta}_{2 \rightarrow 1,3 [1]}^c$ and $\hat{\beta}_{2 \rightarrow 1,3 [3]}^c$).
\end{proof}

\begin{remark}[Mean estimation: algebraic vs. graphical approach]
In both approaches, the PPCA model is translated into a linear model. 
However, both estimators in Equations \eqref{eq:expectation_main}  and \eqref{eq:expectationtoyexgraph} theoretically differ. The exogeneity assumption and approximation is not made at the same step.
In the algebraic approach, the results are first derived without using any approximation. It gives linear models that do not comply with the standard exogeneity assumption. Consequently, an approximation is done at the estimation step since the parameters $\hat{\mathcal{B}}^c_{2 \rightarrow 1,3 [0]}$, $\hat{\mathcal{B}}^c_{2 \rightarrow 1,3 [1]}$ and  $\hat{\mathcal{B}}^c_{2 \rightarrow 1,3 [3]}$ are estimated with the standard linear regression coefficients. 
In the graphical approach, an approximation is made at the first step when a structural equation model is associated with the graphical model by assuming the exogeneity, \textit{i.e.} $\mathbb{E}[\epsilon_{.2}|Y_{.1},Y_{.3}]=0$. In practice, for both approaches, the same coefficients are naturally computed, \textit{i.e.} $\hat{\beta}^c_{j \rightarrow k,\ell}=\hat{\mathcal{B}}^c_{j \rightarrow k,\ell}$, which leads to the same computed estimators for the mean of $Y_{.1}$. 
\end{remark}

Whereas only one simplified graphical model between $Y_{.1}$, $Y_{.2}$ and $Y_{.3}$, displayed in the bottom right graph of Figure \ref{fig:proofToy}, was required to construct an estimator of the mean of $Y_{.1}$, the variance and covariance estimation relies on Equation \eqref{eq:graphtoy} and the following one (associating to the bottom left graph of Figure \ref{fig:proofToy}), 
\begin{equation}\label{eq:graphtoy2}
Y_{.3}=\beta_{3 \rightarrow 1,2 [0]}+\beta_{3 \rightarrow 1,2 [1]}Y_{.1}+\beta_{3 \rightarrow 1,2 [2]}Y_{.2}+\epsilon_{.3},
\end{equation}
assuming $\mathbb{E}[\epsilon_{.3}|Y_{.1},Y_{.2}]=0$ and where $\beta_{3 \rightarrow 1,2 [0]}$, $\beta_{3 \rightarrow 1,2 [1]}$ and $\beta_{3 \rightarrow 1,2 [2]}$ are the intercept and the coefficients of the linear regression of $Y_{.3}$ on $Y_{.1}$ and $Y_{.2}$.

Using Equations \eqref{eq:graphtoy} and \eqref{eq:graphtoy2} and Lemmas \ref{lem:coefflinear}, \ref{lem:graphicalcov}, one can derive some estimators for the variance and the covariances of $Y_1$.

\begin{prop}[Variance and covariances formulae resulting from the graphical approach when $p=3$ and $r=2$]\label{prop:vartoygraph}
	Under the two equations  \eqref{eq:graphtoy} and \eqref{eq:graphtoy2}, assuming \ref{hyp1} and also $\beta^c_{3 \rightarrow 1} \neq 0$, $\beta_{2\rightarrow 1,3[1]}^c \neq 0$ and $\mathrm{Var}(Y_{.3}) \neq 0$, one can construct an estimator of the variance of the MNAR variable $Y_{.1}$ and its covariances as follows
	\begin{align}
	\label{eq:vartoygraph_est}
	\widehat{\mathrm{Var}}(Y_{.1})&:=\frac{\widehat{\mathrm{Var}}(Y_{.3})}{\hat{\beta}_{3\rightarrow1}^c}\frac{1}{\hat{\beta}_{2 \rightarrow 1,3 [1]}^c}\left(\frac{\widehat{\mathrm{Cov}}(Y_{.2},Y_{.3})}{\widehat{\mathrm{Var}}(Y_{.3})}-\hat{\beta}_{2 \rightarrow 1,3 [3]}^c\right), \\
	\label{eq:cov1toy_est}
	\widehat{\mathrm{Cov}}(Y_{.1},Y_{.2})&:=\frac{1}{\hat{\beta}_{3 \rightarrow 1,2 [1]}^c}\left(\frac{\widehat{\mathrm{Cov}}(Y_{.2},Y_{.3})}{\widehat{\mathrm{Var}}(Y_{.2})}-\hat{\beta}_{3 \rightarrow 1,2 [2]}^c\right)\widehat{\mathrm{Var}}(Y_{.2}), \\
	\label{eq:cov2toy_est}
	\widehat{\mathrm{Cov}}(Y_{.1},Y_{.3})&:=\frac{1}{\hat{\beta}_{2 \rightarrow 1,3 [1]}^c}\left(\frac{\widehat{\mathrm{Cov}}(Y_{.2},Y_{.3})}{\widehat{\mathrm{Var}}(Y_{.3})}-\hat{\beta}_{2 \rightarrow 1,3 [3]}^c\right)\widehat{\mathrm{Var}}(Y_{.3}), 
\end{align}
    where $\hat{\beta}_{3 \rightarrow 1,2 [1]}^c$, $\hat{\beta}_{3 \rightarrow 1,2 [2]}^c$ and $\hat{\beta}_{3 \rightarrow 1}^c$ are some estimators of $\beta_{3 \rightarrow 1,2 [1]}^c$, $\beta_{3 \rightarrow 1,2 [2]}^c$ and $\beta_{3 \rightarrow 1}^c$ given in \eqref{eq:graphtoy2}. 

    These estimators are consistent under additional Assumption \ref{hyp4}.
\end{prop}

\begin{proof}
    To derive some estimators of the variance and covariances of the MNAR variable $Y_{.1}$, one want to obtain the following formulas:
    \begin{align}\label{eq:vartoygraph}
	\mathrm{Var}(Y_{.1})&=\frac{\mathrm{Var}(Y_{.3})}{\beta_{3\rightarrow 1}^c}\frac{1}{\beta_{2\rightarrow 1, 3[1]}^c}\left(\frac{\mathrm{Cov}(Y_{.2},Y_{.3})}{\mathrm{Var}(Y_{.3})}-\beta_{2 \rightarrow 1,3[3]}^c\right), \\
	\label{eq:cov1toy}
	\mathrm{Cov}(Y_{.1},Y_{.2})&=\frac{1}{\beta_{3\rightarrow 1, 2[1]}^c}\left(\frac{\mathrm{Cov}(Y_{.2},Y_{.3})}{\mathrm{Var}(Y_{.2})}-\beta_{3 \rightarrow 1,2[2]}^c\right)\mathrm{Var}(Y_{.2}), \\
	\label{eq:cov2toy}
	\mathrm{Cov}(Y_{.1},Y_{.3})&=\frac{1}{\beta_{2\rightarrow 1, 3[1]}^c}\left(\frac{\mathrm{Cov}(Y_{.2},Y_{.3})}{\mathrm{Var}(Y_{.3})}-\beta_{2 \rightarrow 1,3[3]}^c\right)\mathrm{Var}(Y_{.3}).
	\end{align}

	Using Equation \eqref{eq:basicformula}, one has 
	$$\textrm{Cov}(Y_{.1},Y_{.3})=\textrm{Var}(Y_{.1})\beta_{3\rightarrow1},$$
	$$\textrm{Cov}(Y_{.3},Y_{.1})=\textrm{Var}(Y_{.3})\beta_{1\rightarrow3},$$
	so 
	$$\textrm{Var}(Y_{.1})=\frac{\textrm{Var}(Y_{.3})\beta_{1\rightarrow3}}{\beta_{3\rightarrow1}}.$$
	
	Considering the graphical model in the bottom left graph of Figure \ref{fig:proofToy},
	\begin{align}
	\notag
	\textrm{Cov}(Y_{.2},Y_{.3})&=\beta_{2\rightarrow 1, 3[1]}\beta_{1\rightarrow 3} \textrm{Var}(Y_{.3}) + \beta_{2 \rightarrow 1,3[3]} \textrm{Var}(Y_{.3}) &\textrm{(by Lemma \ref{lem:graphicalcov})} \\
	\notag 
	\Rightarrow \beta_{1\rightarrow 3}&=\frac{1}{\beta_{2\rightarrow 1, 3[1]}}\left(\frac{\textrm{Cov}(Y_{.2},Y_{.3})}{\textrm{Var}(Y_{.3})}-\beta_{2 \rightarrow 1,3[3]}\right) & \\
	\label{eq:beta13toy}
	\Rightarrow \beta_{1\rightarrow 3}&=\frac{1}{\beta_{2\rightarrow 1, 3[1]}^c}\left(\frac{\textrm{Cov}(Y_{.2},Y_{.3})}{\textrm{Var}(Y_{.3})}-\beta_{2 \rightarrow 1,3[3]}^c\right)  &
	\end{align}
	where the last implication is given by Lemma \ref{lem:coefflinear} and Assumption \ref{hyp1}, giving also $$\beta_{3\rightarrow1}=\beta^c_{3\rightarrow1},$$
	
	which leads to Equation \eqref{eq:vartoygraph}.

	By \eqref{eq:basicformula}, the covariances can be expressed in two different ways, 
	\begin{align}
	\textrm{Cov}(Y_{.1},Y_{.2})&=\beta_{2 \rightarrow 1}\textrm{Var}(Y_{.1}) \quad \textrm{ and } \quad \textrm{Cov}(Y_{.1},Y_{.3})=\beta_{3 \rightarrow 1}\textrm{Var}(Y_{.1}),
	\label{eq:covfirstcase}
	\\
	\textrm{Cov}(Y_{.1},Y_{.2})&=\beta_{1 \rightarrow 2}\textrm{Var}(Y_{.2})\quad  \textrm{ and } \quad \textrm{Cov}(Y_{.1},Y_{.3})=\beta_{1 \rightarrow3}\textrm{Var}(Y_{.3}).
	\label{eq:covsecondcase}
	\end{align}
	In \eqref{eq:covfirstcase}, the coefficients $\beta_{2 \rightarrow 1}$ and $\beta_{3 \rightarrow 1}$ can be estimated on the complete case using Lemma \ref{lem:coefflinear}, but the variance of $Y_{.1}$ has still to be taken care of.  Instead of potentially propagate error from \eqref{eq:vartoygraph}, we propose to favor the expressions given in \eqref{eq:covsecondcase} to evaluate the covariances. 
	
	Focusing on \eqref{eq:covsecondcase},  the coefficient $\beta_{1 \rightarrow 3}$ is given in \eqref{eq:beta13toy} and $\beta_{1 \rightarrow 2}$ can be obtained using the same method, based on the reduced graphical model in the bottom right graph of Figure \ref{fig:proofToy} (by Assumption \ref{hyp1}), so that
	$$\beta_{1\rightarrow 2}=\frac{1}{\beta_{3\rightarrow 1, 2[1]}^c}\left(\frac{\textrm{Cov}(Y_{.2},Y_{.3})}{\textrm{Var}(Y_{.2})}-\beta_{3 \rightarrow 1,2[2]}^c\right).$$
	Therefore, by plugging it in \eqref{eq:covsecondcase}, Equations \eqref{eq:cov1toy} and \eqref{eq:cov2toy} are obtained.
	
	The natural estimators for $\mathrm{Var}(Y_{.1})$, $\mathrm{Cov}(Y_{.1},Y_{.2})$ and $\mathrm{Cov}(Y_{.1},Y_{.3})$ are then given by \eqref{eq:vartoygraph_est}, \eqref{eq:cov1toy_est} and \eqref{eq:cov2toy_est}. They are consistent given that all the quantites involved are consistent, by using \ref{hyp4} (for the consistency of $\widehat{\mathrm{Var}}(Y_{.2})$, $\widehat{\mathrm{Var}}(Y_{.3})$ and $\widehat{\mathrm{Cov}}(Y_{.2},Y_{.3})$) and Lemma \ref{lem:coefflinear} (for the consistency of $\hat{\beta}^c_{j \rightarrow k,\ell}$). 
\end{proof}

\begin{remark}[Var-covariance estimation: algebraic vs. graphical approach]
As for the mean, the exogeneity assumption is required in the last step of the algebraic approach to estimate coefficients and in the first step of the graphical approach to obtain structural equation models. However, contrary to the estimator suggested for the mean,  the estimators in both graphical and algebraic approaches here differ (compare \eqref{eq:estimcov_main} with \eqref{eq:vartoygraph_est}, \eqref{eq:cov1toy_est} 	and \eqref{eq:cov2toy_est}). Indeed, the algebraic approach is based on the use of conditionality, whereas the graphical one relies on graphical results standing for the linear models when exogeneity holds.
\end{remark}

\section{PPCA with MAR data}
\label{sec:MARcase}

The following proposition is an adaptation of our method to handle MAR data, called \textbf{MAR} in Section \ref{sec:simu_synthdata}, inspired by \cite[Theorems 1, 2, 3]{mohan2018estimation}.  
In the MAR case, we assume the following
\begin{enumerate}[label=\textbf{A\arabic*$_{\text{MAR}}$.}]
    \setcounter{enumi}{0}
    \item \label{hyp1MAR} $(B_{.j'})_{j' \in {\mathcal{J}}}$ is invertible.
    \item \label{hyp2MAR} $\forall m \in \mathcal{M},$ $Y_{.m} \independent \Omega_{.m} | (Y_{k})_{k \in \widebar{\{m\}}}$
    \item \label{hyp3MAR}  $\forall m \in \mathcal{M}$, the complete-case coefficients $\mathcal{B}_{m\rightarrow \mathcal{J}[0]}^c$ and $\mathcal{B}_{m\rightarrow \mathcal{J}[k]}^c, k \in \mathcal{J}$ can be consistently estimated.
\end{enumerate}
\begin{enumerate}[label=\textbf{A\arabic*$_{\text{MAR}}$.}]
    \setcounter{enumi}{4}
    \item \label{hyp5MAR} $\forall \ell \in \widebar{\mathcal{J}}$, for all set $\mathcal{H} \subset \mathcal{J}_{-j}$ such that $|\mathcal{H}|=r-1$,  $\begin{pmatrix} B_{.\ell} & (B_{.j'})_{j' \in \mathcal{H}} \end{pmatrix}$ is invertible,
    \item \label{hyp6MAR} 
    $\forall m \in \mathcal{M}, \forall \ell \in \widebar{\mathcal{J}}\setminus\mathcal{M}$, $\forall j \in \mathcal{J}$,
    \, $Y_{.m} \independent \Omega_{.\ell}| (Y_{.k})_{k \in \widebar{\{m\}}}.$
\end{enumerate}
\begin{enumerate}[label=\textbf{A\arabic*$_{\text{MAR}}$.}]
    \setcounter{enumi}{7}
    \item \label{hyp8MAR} $\forall m \in \mathcal{M}, \forall \ell \in \widebar{\{m\}}\setminus\mathcal{J}$, 
    for all set $\mathcal{H} \subset \mathcal{J}$ such that 
    $|\mathcal{H}|=r-1$, the complete-case coefficients 
    $\mathcal{B}_{m\rightarrow \ell,\mathcal{H}[0]}^c$ and 
    $\mathcal{B}_{m\rightarrow \ell,\mathcal{H}[k]}^c, k \in \{\ell\}\cup\mathcal{H}$ can be consistently estimated. 
\end{enumerate}

\begin{prop}[Expectation, variance and covariances formulae for a MAR variable when $p=3$ and $r=2$]\label{prop:MARPearl}
	Consider the PPCA model \eqref{eq:model}. Under Assumptions \ref{hyp1MAR} and \ref{hyp2MAR}, one can construct the estimators of the mean, the variance and the covariances with a pivot variable for any MAR variable $Y_{.m}, m \in \mathcal{M}$, as follows 
	\begin{itemize}[label=\textendash]
		\item the mean of the missing variable $$\hat{\alpha}_m=\hat{\mathcal{B}}^c_{m \rightarrow \mathcal{J} [0]}+\sum_{j \in \mathcal{J}}\hat{\mathcal{B}}^c_{m \rightarrow \mathcal{J} [j]}\hat{\alpha}_j,$$
		with $\mathcal{J}$ the pivot variables set,
		\item 	the variance of the missing variable 
		\begin{multline*}
		\widehat{\mathrm{Var}}(Y_{.m})=\widehat{Q}_{\mathrm{MAR}}^c
		+\sum_{j \in \mathcal{J}}(\hat{\mathcal{B}}_{m \rightarrow \mathcal{J} [j]}^c)^2\widehat{\mathrm{Var}}(Y_{.j}) \\
		+2\sum_{(j<k) \in \mathcal{J}}\hat{\mathcal{B}}_{m \rightarrow \mathcal{J} [j]}^c\hat{\mathcal{B}}_{m \rightarrow \mathcal{J} [k]}^c\widehat{\mathrm{Cov}}(Y_{.j},Y_{.k}),
		\end{multline*}
		with
	\begin{multline*}
	\widehat{Q}_{\mathrm{MAR}}^{c}=\left(\widehat{\mathrm{Var}}(Y_{.m})\big| \Omega_{.m}=1\right)
	\\
	-\left(\widehat{\mathrm{Cov}}((Y_{.j})_{j \in \widebar{\{m\}}},Y_{.m}) \widehat{\mathrm{Var}}((Y_{.j})_{j \in \widebar{\{m\}}})^{-1} \widehat{\mathrm{Cov}}((Y_{.j})_{j \in \widebar{\{m\}}},Y_{.m})^T \big| \Omega_{.m}=1\right).
	\end{multline*}
		\item the covariances between the missing variable and a pivot variable, for all $\ell \in \mathcal{J}$,
		\begin{multline*}
		 \widehat{\mathrm{Cov}}(Y_{.m},Y_{.\ell})=\hat{\mathcal{B}}^c_{m \rightarrow \mathcal{J}[0]}\hat{\alpha}_{\ell}+\hat{\mathcal{B}}^c_{m \rightarrow \mathcal{J}[\ell]}(\widehat{\mathrm{Var}}(Y_{.\ell})+\hat{\alpha}_\ell^2) \\
	+\sum_{k\in \mathcal{J}_{-\ell}} \hat{\mathcal{B}}^c_{m\rightarrow \mathcal{J}[k]}(\widehat{\mathrm{Cov}}(Y_{.\ell},Y_{.k})+\hat{\alpha}_{\ell}\hat{\alpha}_{k}) - \hat{\alpha}_{m}\hat{\alpha}_{\ell}
		\end{multline*}
	\end{itemize}
	Under Assumption \ref{hyp3MAR} and \ref{hyp4}, these estimators are consistent. 
	
	In addition, under Assumption \ref{hyp5MAR}, \ref{hyp6MAR} and \ref{hyp6bis}, one can construct the estimator of the covariance between a MAR variable $Y_{.m}$ for $m \in \mathcal{M}$ and any not pivot variable as follows
	\begin{itemize}[label=\textendash]
	\item the covariances between the missing variable and any not pivot variable, for all $\ell \in \widebar{\{m\}}\setminus\mathcal{J}$, choose $r-1$ variable indexes in $\mathcal{J}$ to form the set $\mathcal{H}\cup \mathcal{J}$ such that $|\mathcal{H}|=r-1$
		\begin{multline*}
		 \widehat{\mathrm{Cov}}(Y_{.m},Y_{.\ell})=\mathcal{B}^c_{m \rightarrow \ell,\mathcal{H}[0]}\hat{\alpha}_{\ell}+\hat{\mathcal{B}}^c_{m \rightarrow \ell,\mathcal{H}[\ell]}(\widehat{\mathrm{Var}}(Y_{.\ell})+\hat{\alpha}_\ell^2) \\
	+\sum_{k\in \mathcal{H}} \hat{\mathcal{B}}^c_{m\rightarrow \ell,\mathcal{H}[k]}(\widehat{\mathrm{Cov}}(Y_{.\ell},Y_{.k})+\hat{\alpha}_{\ell}\hat{\alpha}_{k}) - \hat{\alpha}_{m}\hat{\alpha}_{\ell}
		\end{multline*}
	\end{itemize}
	
	Under the additional Assumptions \ref{hyp8MAR} and \ref{hyp7} 
	this estimator is consistent.
\end{prop}

\begin{proof}
The proof follows exactly the same direction than in Proposition \ref{prop:mean_formula_general}, \ref{prop:var_formula_general} and \ref{prop:covmiss}. The only difference is that the regressions used are not the same. 

For the sake of clarity, consider the same toy example as in Section \ref{sec:toyexample} where $p= 3,r= 2$, in which only one variable can be missing (at random), and fix $\mathcal{M}=\{1\}$ and $\mathcal{J}=\{2,3\}$. Note that here the MAR mechanism leads to $\mathbb{P}(\Omega_{.1}=0|Y_{.1},Y_{.2},Y_{.3})=\mathbb{P}(\Omega_{.1}=0|Y_{.2},Y_{.3}).$. The goal is to estimate the mean of $Y_{.1}$, without specifying the distribution of the missing-data mechanism and using only the observed data. 

Assumption \ref{hyp1MAR} allows to obtain  linear link between the MAR variable $Y_{.1}$ and the pivot variables ($Y_{.2},Y_{.3}$). In particular, one has
$$Y_{.1} = \beta_{1\rightarrow2,3[0]} + \beta_{1\rightarrow2,3[2]}Y_{.2} + \beta_{1\rightarrow2,3[3]}Y_{.3} + \zeta,
$$
with $\beta_{1\rightarrow2,3[0]}$, $\beta_{1\rightarrow2,3[2]}$ and $\beta_{1\rightarrow2,3[3]}$ the intercept and coefficients standing for the effects of $Y_{.1}$ on $Y_{.2}$ and $Y_{.3}$, and with 
$$\zeta=-\mathcal{B}_{1\rightarrow 2,3[2]} \epsilon_{.2} - \mathcal{B}_{1\rightarrow 2,3[3]} \epsilon_{.3} + \epsilon_{.1}$$

Assumption \ref{hyp2MAR}, \textit{i.e.} $Y_{.1} \independent \Omega_{.1} | Y_{.2}, Y_{.3}$, is required to obtain identifiable and consistent parameters of the distribution of $Y_{.1}$ given $Y_{.2},Y_{.3}$ in the complete-case when $\Omega_{.1}=1$, denoted as $\beta_{1\rightarrow2,3[0]}^c$, $\beta_{1\rightarrow2,3[2]}^c$ and $\beta_{1\rightarrow2,3[3]}^c$,
$$
(Y_{.1})_{| \Omega_{.1}=1} = \beta_{1\rightarrow2,3[0]}^c + \beta_{1\rightarrow2,3[2]}^c Y_{.2} + \beta_{1\rightarrow2,3[3]}^c Y_{.3} + \zeta^c,
$$
with 
$$\zeta^c=-\mathcal{B}^c_{1\rightarrow 2,3[2]} \epsilon_{.2} - -\mathcal{B}^c_{1\rightarrow 2,3[3]} \epsilon_{.3} + \epsilon_{.1}$$
(In the MNAR case, the regression of $Y_{.1}$ on $(Y_{.2},Y_{.3})$ is prohibited, as \ref{hyp2MAR} does not hold. That is why we used the regression of $Y_{.2}$ on $Y_{.1}$ and $Y_{.3}$.);

Using again \ref{hyp2MAR}, one has 
$$\mathbb{E}\left[Y_{.1}|Y_{.2}, Y_{.3},\Omega_{.1}=1\right]=\mathbb{E}\left[\beta_{1\rightarrow2,3[0]}^c + \beta_{1\rightarrow2,3[1]}^c Y_{.2} + \beta_{1\rightarrow2,3[3]}^c Y_{.3}|Y_{.2},Y_{.3}\right]+\mathbb{E}[\zeta^c|Y_{.2},Y_{.3}],$$
and taking the expectation leads to
$$\mathbb{E}\left[Y_{.1}\right]=\beta_{1\rightarrow2,3[0]}^c + \beta_{1\rightarrow2,3[1]}^c\mathbb{E}\left[Y_{.2}\right] + \beta_{1\rightarrow2,3[3]}^c\mathbb{E}\left[Y_{.3}\right],$$
given that $\mathbb{E}[\epsilon_{.k}]=0, \: \forall k\in \{1,2,3\}$.

One obtains
$$\alpha_1=\beta_{1\rightarrow2,3[0]}^c + \beta_{1\rightarrow2,3[1]}^c\alpha_2 + \beta_{1\rightarrow2,3[3]}^c\alpha_3$$
A natural estimator for $\alpha_1$ is 
$$\hat{\alpha}_1=\hat{\beta}_{1\rightarrow2,3[0]}^c + \hat{\beta}_{1\rightarrow2,3[1]}^c\hat{\alpha}_2 + \hat{\beta}_{1\rightarrow2,3[3]}^c\hat{\alpha}_3,$$
which is consistent using Assumption \ref{hyp3MAR} and \ref{hyp4}.
\end{proof}

\end{document}